\documentclass[11pt]{amsart}

\usepackage[pagewise]{lineno}
\usepackage{hyperref} 
\usepackage[dvips]{graphicx}
\usepackage{caption}
\usepackage{subcaption}

\usepackage[T1]{fontenc}
\usepackage{mathptmx}

\usepackage{color}

\usepackage{comment}

\usepackage{soul}

\DeclareMathOperator{\trace}{trace}
\DeclareMathOperator{\Ad}{Ad}

\newcommand{\bbar}{\begin{pmatrix}}
\newcommand{\ebar}{\end{pmatrix}}

\newcommand{\bdm}{\begin{displaymath}}
\newcommand{\edm}{\end{displaymath}}
\newcommand{\beq}{\begin{equation}}
\newcommand{\beqa}{\begin{eqnarray}}
\newcommand{\beqas}{\begin{eqnarray*}}
\newcommand{\eeq}{\end{equation}}
\newcommand{\eeqa}{\end{eqnarray}}
\newcommand{\eeqas}{\end{eqnarray*}}
\newcommand{\dd}{\textup{d}}

\DeclareMathOperator{\rank}{Rank}

\newcommand{\C}{{\mathbb C}}

\newcommand{\real}{{\mathbb R}}
\newcommand{\SSS}{{\mathbb S}}

\newcommand{\mv}{\mathcal{V}}

\newcommand{\mfk}{\mathfrak{k}}
\newcommand{\mfp}{\mathfrak{p}}

   \newtheorem{theorem}{Theorem}[section]
   \newtheorem{proposition}[theorem]{Proposition}
   
   \newtheorem{lemma}[theorem]{Lemma}

 \theoremstyle{remark}
   \newtheorem{example}[theorem]{Example}
   \newtheorem{remark}[theorem]{Remark}
   
\numberwithin{equation}{section}

\begin{document}
\author{David Brander}
\address{Department of Applied Mathematics and Computer Science\\ 
Richard Petersens Plads, Building 324\\
Technical University of Denmark\\
DK-2800 Kgs. Lyngby\\ Denmark}
\email{dbra@dtu.dk}

\author{Farid Tari}
\address{Instituto de Ci\^encias Matem\'aticas e de Computa\c{c}\~ao - USP\\
Avenida Trabalhador S\~ao-Carlense, 400 - Centro \\
CEP: 13566-590 - S\~ao Carlos - SP, Brazil}
\email{faridtari@icmc.usp.br}

\title[Wave maps and constant curvature surfaces: singularities and bifurcations]{Wave maps and constant curvature surfaces: \\
singularities and bifurcations}

\begin{abstract}
Wave maps (or Lorentzian-harmonic maps) from a $1+1$-dimensional Lorentz space into the $2$-sphere are associated
to constant negative Gaussian curvature surfaces in Euclidean 3-space via the Gauss map,
which is harmonic with respect to the metric induced by the second fundamental form.
We give a method for constructing germs of analytic Lorentzian-harmonic maps from their $k$-jets and use this construction to study the singularities of such maps. 
We also show how to construct pseudospherical surfaces with prescribed singularities using loop groups. 
We study the singularities of pseudospherical surfaces and obtain their bifurcations in generic 1-parameter families of such surfaces. 
 \end{abstract}

\keywords{Bifurcations, constant Gauss curvature, Cauchy problem, differential geometry, discriminants, frontals, integrable systems, loop groups, pseudospherical surfaces,  singularities, wave fronts, wave maps.}
\subjclass[2010]{Primary 53A05; Secondary 53C43, 57R45}

\date{\today}

\maketitle


\section{Introduction}
We study in this paper Lorentzian-harmonic maps into the 2-sphere and their associated pseudospherical surfaces.
Let $S$ be a connected surface with a Lorentz structure and with universal cover $\tilde S$.
A harmonic map $N: S \to \SSS^2$
determines a geometrically unique map $f: \tilde S \to \real^3$, that, with the metric induced from $\real^3$,
has constant Gauss curvature $K=-1$ at all regular points, and has Gauss map $N$. We will use the term
\emph{pseudospherical surface} for such a map $f$. Conversely, the Gauss map
of a regular constant negative curvature surface is harmonic with respect to the metric induced by the second fundamental form.

We are concerned here with the local singularities of analytic Lorentzian-harmonic maps into the 2-sphere and of their associated pseudospherical surfaces. In general, the singularities are determined by a certain jet of these mappings at the singular point. A 
natural question arises. Suppose that a given polynomial map $P$ of degree $k$ from a Lorentz surface into the 2-sphere satisfies the Lorentzian-harmonic condition up to order $k$: is there a germ of a Lorentzian-harmonic map into the 2-sphere with $k$-jet $P$?
We prove that this is the case in Theorem \ref{theo:kjetLorentzH}. 
This opens the way to the study, from a singularity theory point of view, of the singularities of harmonic and Lorentzian-harmonic map-germs into a target space which is not flat (so far, to our knowledge, the only work on singularities such map-germs is the one on  
harmonic maps between 2-dimensional Riemannian manifolds carried out by J.C. Wood \cite{wood1977}).

A pseudospherical surface is, by definition, a frontal:  it has, even at non-immersed points, a well-defined normal. We distinguish between two different types of singularities of these surfaces, one where the map $f$ is a wave front (sometimes called a \emph{weakly regular} pseudospherical surface) and the other where it is not.
We deal in \S \ref{sec:wavefront_case} with wave front pseudospherical surfaces. These are parallels of regular linear Weingarten surfaces. 
We use Theorem \ref{theo:kjetLorentzH}, results from singularity theory \cite{ArnoldWavefront,bruceParallel} and the recognition criteria in \cite{dbft} 
to obtain the bifurcations in generic 1-parameter families of pseudospherical surfaces (Proposition \ref{prop:NoD_4} and Theorem \ref{theo:bif_Wavefront_case}). In this case, the stable singularities of pseudospherical surfaces are cuspidal edges and swallowtails (Figure \ref{fig:StablePSS}; see \cite{ishimach}). We prove that the bifurcations in generic 1-parameter families of such surfaces are the so-called cuspidal lips ($A_3^+$), cuspidal beaks ($A_3^-$) and cuspidal butterfly ($A_4$); see Figure \ref{fig:familiesPSS}.

\begin{figure}[htp]
\begin{center}
\includegraphics[height=2.5cm]{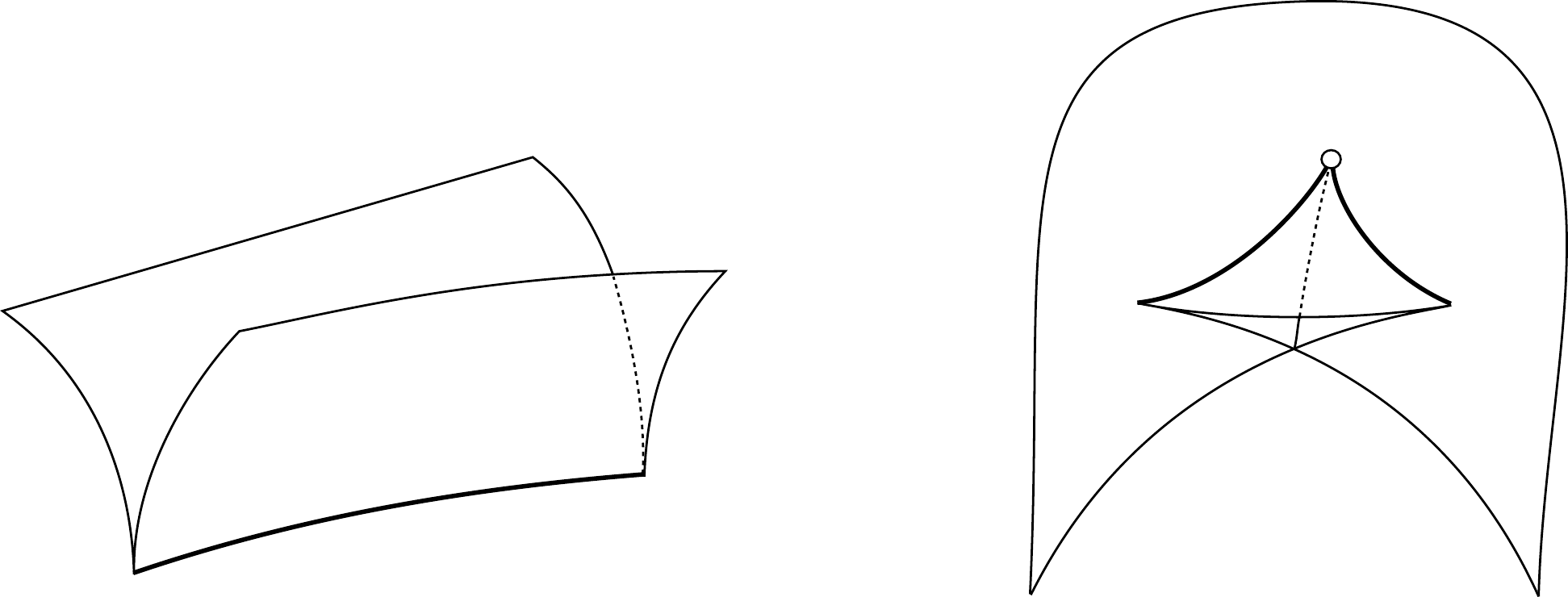}
\caption{Stable singularities of wave fronts and parallels: cuspidal edge ($A_2$, left) and swallowtail ($A_3$, right). (The singular set of an ordinary cuspidal edge is drawn as a thicker curve.)
	 Both cases occur on pseudospherical surfaces.}
\label{fig:StablePSS}

\end{center}
\end{figure}

\begin{figure}[htp]
\begin{center}
\includegraphics[height=6.75cm]{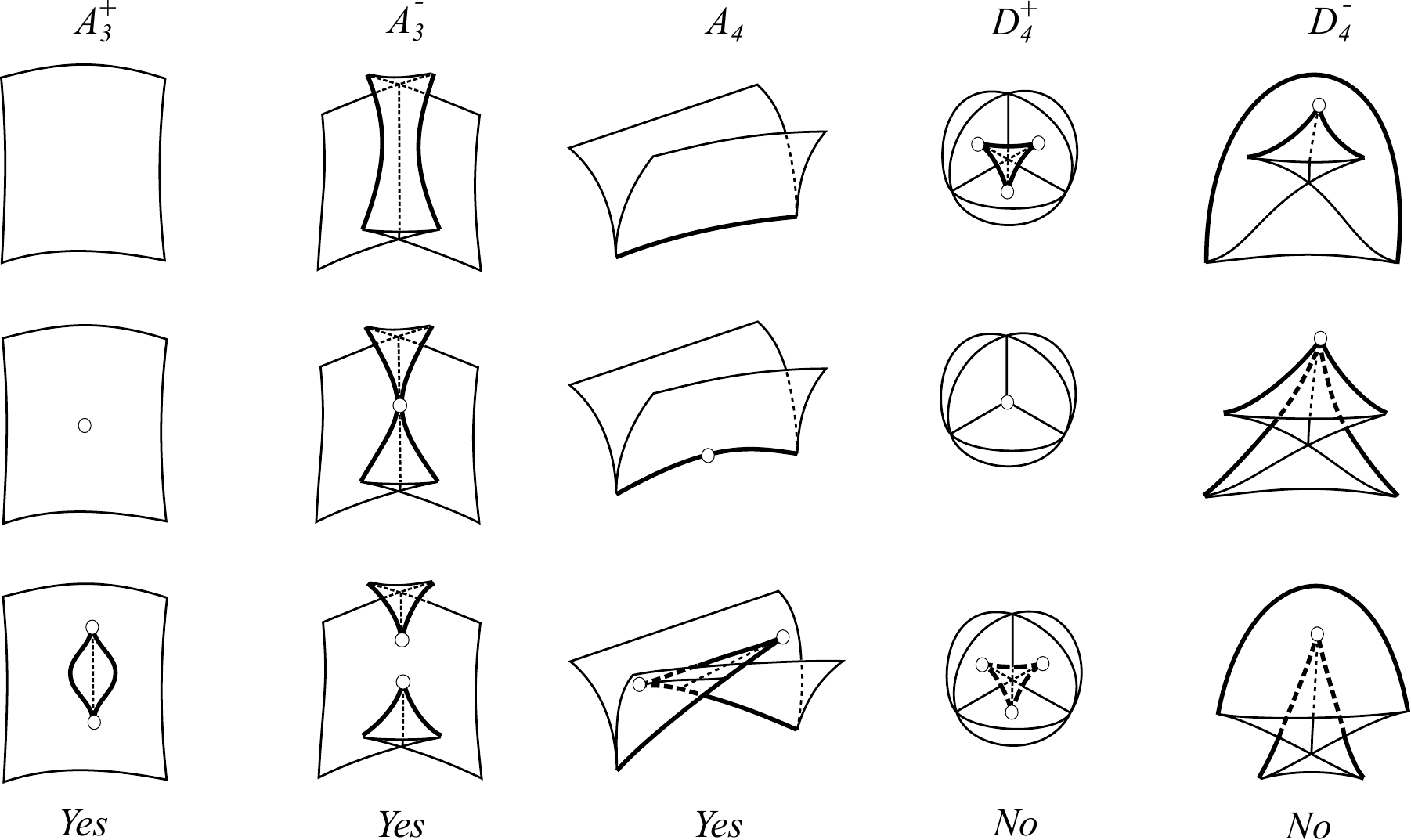}
\caption{Generic bifurcations of parallels from (\cite{ArnoldWavefront} and \cite{bruceParallel}). ``Yes'' for those that can occur in families of pseudospherical surfaces and ``No'' for those that do not.}
\label{fig:familiesPSS}
\end{center}
\end{figure}

\begin{figure}[htp]
\begin{center}
\includegraphics[height=2cm, width=10cm]{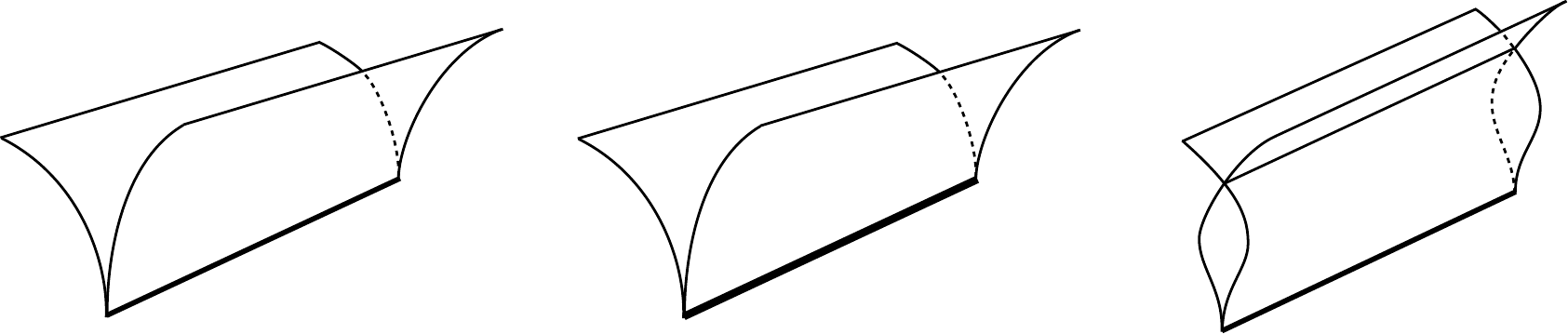}
\caption{Generic bifurcations of a pseudospherical surface with a 2/5-cuspidal edge singularity.
(The 2/5-cuspidal edge is the thick curve in the middle figure.)
}
\label{fig:bif25CuspEdge}
\end{center}
\end{figure}

\begin{figure}[htp]
\begin{center}
\includegraphics[height=2cm, width=10cm]{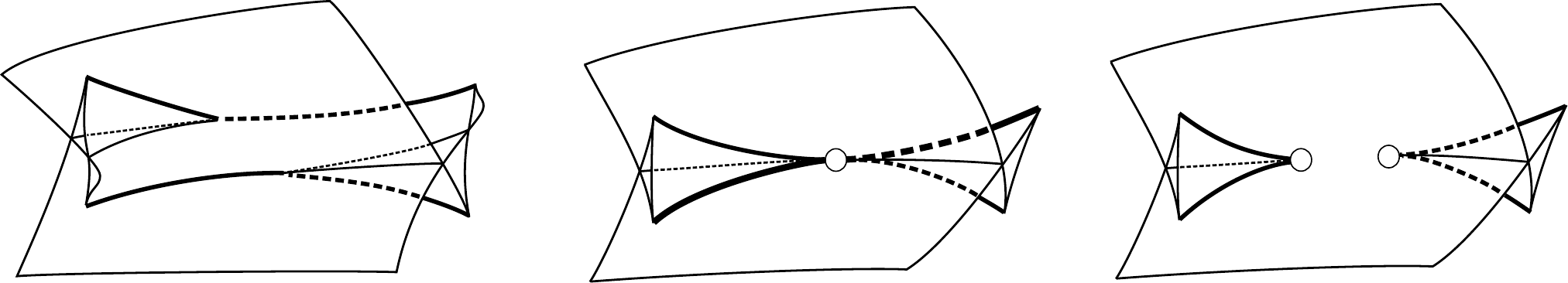}
\caption{Generic bifurcations of a pseudospherical surface with a Shcherbak singularity. 
(The 2/5-cuspidale edge is the thicker curve in the middle figure.)}
\label{fig:Shcherbak}
\end{center}
\end{figure}

In \S \ref{sec:frontal} we deal with singularities at a point where the surface is not a wave front.
There is so far no general theory that deals with bifurcations in families of frontals with non-isolated singularities (compare \cite{Ishikawa(2005)}), unlike the wave front case where one uses generating families of functions (\cite{arnoldetal}); the difficulty being that the map $(f,N)$ from the surface to the unit cotangent bundle $T_1^*\mathbb R^3$ is not necessarily an immersion. 
See \cite{Ishikawa_Frontals} for a survey article, new results and references  on frontals.
We define a map to the $k$-jet space of Lorentzian-harmonic map-germs and use it to define the codimension of a singularity of the associated pseudospherical surface as well as the notion of generic families of such surfaces. 
There are no stable non-wave front singularities. The codimension 1 singularities are the $2/5$-cuspidal edges which bifurcate as in Figure \ref{fig:bif25CuspEdge} (Theorem \ref{theo:2/5cuspidaledge}), and the Shcherbak surface singularity which bifurcates as Figure \ref{fig:Shcherbak} (Theorem \ref{theo:2/5CuspidalBeaks}).  Thus, the five bifurcations illustrated in Figures \ref{fig:familiesPSS}, \ref{fig:bif25CuspEdge} and \ref{fig:Shcherbak} is the complete list of bifurcations in $1$-parameter families of generic pseudospherical surfaces.

In \S \ref{sec:Loop_Groups} we provide a method for constructing the bifurcations explicitly, using loop groups. Previously, in \cite{singps}, loop group methods were used to study singularities by solving the Cauchy problem along the singular curve itself. That approach has the complication that a different treatment is needed if the curve is tangent to a characteristic direction.
Here we use a new approach, solving the Cauchy problem along an arbitrary non-characteristic curve that passes through the singular point of interest. 
This gives a simple unified formulation that works for all types of singularities.

In Theorem \ref{gctheorem} we solve, using loop groups, the Cauchy problem for a Lorentzian-harmonic map $N :\real^{1,1} \to \SSS^2$, with $N$ and one of its transverse derivatives prescribed along a non-characteris\-tic curve, by giving the formula for a potential that produces the solution. This allows one to compute the solutions numerically.  The theorem represents all Lorentzian-harmonic maps in terms of a triple of functions $(a(t),b(t),c(t))$ of one variable. Theorems \ref{wavefronttheorem} and \ref{nonwavefronttheorem} then give conditions on $a$, $b$ and $c$ that characterize the different types of generic singularities, and bifurcations in generic one parameter families of pseudospherical surfaces that were classified in the preceding sections.

In \S \ref{sec:Classification} we turn our attention to the singularities of the harmonic map itself. We obtain the $\mathcal A$-singularities of map-germs from the plane into the 2-sphere that can be realized by Lorentzian-harmonic maps, where $\mathcal A$ is the Mather group of germs of changes of coordinates in the source and target.  Finally, in \S\ref{sec:Classification_plane}, we discuss as a matter of side interest the $\mathcal A$-singularities of Lorentzian-harmonic maps into the plane.  

Concluding remarks:
A harmonic map $N: S \to \SSS^2$ and its associated pseudospherical surface $f: S \to \real^3$ have the same singular set in the coordinate domain $S$: but, since the map and the target space are not the same, the $\mathcal{A}$-singularities are completely different.  The unifying theme of this work is that, in both cases,  the essential tool is the characterization of $k$-jets of harmonic maps in Theorem \ref{theo:kjetLorentzH}.

For pseudospherical surfaces, 
we have given two representations of the generic singularities and bifurcations: one via the $k$-jet characterization, and the other by explicit formulas for the loop group potentials that produce them.  The first representation enables one to directly read off the codimension of a given singularity, whilst the loop group potentials give simple formulas in terms of geometric data along a curve and can be used to compute solutions.
While it is of course impossible to describe all singularities, a classification of the generic ones and those of codimension 1 gives
a good understanding of the singularities usually encountered on global solutions. One of the challenges of this work is that there is no general theory for the bifurcations of frontals. Although
non-wave front singularities are not stable, it is necessary to deal with them when one considers bifurcations of $1$-parameter families of pseudospherical surfaces, as illustrated by the Shcherbak singularity. The results in this paper provide a motivation for seeking to develop such theory.


\section{Basic properties of Lorentzian-harmonic maps and their associated pseudospherical surfaces}  \label{sec:harmonicmaps}

\subsection{Lorentzian harmonic maps, pseudospherical frontals and wave fronts}

Let $(S,h)$  be a simply connected Lorentz surface and
$N: S \to \SSS^2$  a smooth  map.  
Then $N$ is Lorentzian-harmonic if and only if for any null coordinate system $(x,y)$ the mixed partial derivative
$N_{xy}$ is proportional to $N$, i.e.,
\[
N \times N_{xy}=0.
\]
 Lorentzian-harmonic maps are also commonly
known as wave maps: see \cite{taturu2004} for a survey on wave maps, and \cite{tkmilnor1982, pohlmeyer} for the connection with special surfaces.

The pseudospherical surface associated to $N$, unique up to a translation,
 is the solution  $f: S \to \real^3$ to the system:
\beq  \label{assocfrontaleqn}
f_x = N \times N_x , \quad
f_y = - N \times N_y.
\eeq
The compatibility of the system \eqref{assocfrontaleqn}, i.e.
$\partial _y ( N \times N_x ) = \partial _x (- N \times N_y)$,
 is in fact equivalent to the Lorentzian-harmonic map equation $N \times N_{xy}=0$. 
Moreover, $f$ is well-defined by $N$ independent of the choice of null-coordinates.

A differentiable map $g: S\to \real^3$ from a surface into Euclidean space is called a \emph{frontal map} and its image a \emph{frontal}
if there is a differentiable map $\nu:S \to \SSS^2 \subset \real^3$ such that $\dd g$ is orthogonal to
 $\nu$ (see \cite{Ishikawa_Frontals} for references).  
This means that the Legendrian lift $L=(g, \nu): S \to \real^3 \times \SSS^2$ is an isotropic map, where $\real^3 \times \SSS^2$ is identified locally with the unit cotangent bundle $T_1^*\mathbb R^3$ equipped with the canonical contact structure: that is, the pull-back of the contact form by $L$ vanishes on $S$.
 
The map $g$ is a Legendrian map and its image a \emph{wave front} (or simply a front) if the map $L$ is an immersion (this means that the image of $L$ is a Legendrian surface); see \cite{arnoldetal}.

In this paper we use the same notation for a parameterization of a surface and for the surface itself, so frontal (resp. wave front) also indicates frontal map (resp. Legendrian map).

\begin{proposition}\label{prop:f_frontal_wavefront}
A pseudospherical surface $f$ is always a frontal. It is a wave front if and only if both $N_x$ and $N_y$ are non-vanishing, or, equivalently, both
$f_x$ and $f_y$ are non-vanishing.
\end{proposition}

\begin{proof} For a pseudospherical surface $f$, $\dd f$ is  orthogonal to $N$, which makes $f$ a frontal
with a Legendrian lift  $L=(f,N): S \to \real^3 \times \SSS^2$. Since $L_x = (N \times N_x, N_x)$ and $L_y=(-N \times N_y, N_y)$, the map $L$ is an immersion if and only if $N_x\ne 0\ne N_y$, which is equivalent to $f_x\ne 0\ne f_y$ by \eqref{assocfrontaleqn}.
\end{proof}

From the equations \eqref{assocfrontaleqn} for $f$, we have $|f_x|=|N_x|$, $|f_y|=|N_y|$ and
\[
\langle f_x, f_y \rangle = -\langle N_x , N_y \rangle,
\]
so $\dd f$ and $\dd N$ have the same rank at each point and the set of points where $f$ is regular is precisely the set of points where $N$ is
regular.  Around such a point we can write
\[
|f_x|=|N_x|=:A \neq 0, \quad |f_y|=|N_y|= : B\neq 0, 
\quad \langle f_x, f_y \rangle = AB \cos \phi,
\]
so the first fundamental form of $f$, with the metric induced from $\real^3$,
is 
$$I = A^2 \dd x^2 + 2AB \cos \phi \dd x \dd y + B^2 \dd y^2.$$

By the definition of $f$
and by the harmonicity of $N$ we have
\[
\langle f_x , N_x \rangle = \langle f_y, N_y \rangle =0, 
\quad f_{xy} = N_y \times N_x  = f_x \times f_y = (AB \sin \phi) N,
\]
so the second fundamental form of $f$ is 
$$II = 2AB \sin \phi \, \dd x \dd y.$$ 
Hence
the null coordinates $(x,y)$ are asymptotic coordinates for $f$, and
the Gaussian curvature of $f$ is constant $K=-1$.
Conversely, it is well known that 
 global asymptotic coordinates exist for any regular constant negative curvature
surface $f$ (see e.g. \cite{docarmo}), and that these coordinates can be chosen with $\langle f_x, f_x\rangle = \langle f_y, f_y\rangle =1$. Asymptotic coordinates satisfy
$\langle f_x, N_x\rangle = \langle f_y , N_y\rangle=0$.  From these conditions it follows that $N_{xy}$ is parallel to $N$, i.e., the Gauss map is  harmonic with respect to the Lorentz structure defined by the  second fundamental form.
Thus the maps defined by \eqref{assocfrontaleqn} are a
generalization of regular pseudospherical surfaces.

\subsection{The local singular set for a pseudospherical wave front}
In a neighbourhood of a point where $f$ is a pseudospherical wave front, the angle $\phi$ is well-defined and the set
\[
\Sigma = \{ (x,y) ~|~ \sin \phi = 0 \}
\]
is the set of points where $f$ (and hence $N$ also) fails to be an immersion. 
At a regular point of $\Sigma$, 
the \emph{null-direction} is the kernel
of $\dd f$.  Since, along the singular curve 
$\dd f = (A \dd x + \varepsilon_1 B \dd y) f_x/|f_x|$, where $\varepsilon_1 = \pm 1$ 
depending on whether 
$\phi$ is an even or odd multiple of $\pi$,
the null direction is given by:
\[
\eta = B \partial_x - \varepsilon_1 A \partial_y,
\qquad \varepsilon_1 = \hbox{sign}(\cos \phi),
\]
which would be a principal direction if the surface were regular.

Geometric recognition criteria for identifying the singularities of wave fronts (and some frontals) 
$f$ with a Legendrian lift $L=(f,\nu)$ 
are established in \cite{izumiyasaji,izsata,krsuy} using 
the singular set and the null-direction. We reproduce them here for completeness. 
Let $\sigma(x,y)=\det(f_x,f_y,\nu)$ be the function whose zero set gives the singular set of $f$. 
When $\rank(\dd f_p)=1$, there is a unique vector field $\eta$ along the singular set of $f$ 
parameterised by the null-directions.

\begin{theorem}\label{theo:critiria}
Let $f : U\to \mathbb R^3$  be a wave front and $p$ a singular point of $f$.

{\rm (i)} The germ of $f$ at $p$ is a cuspidal edge $(A_2)$ if and only if
the singular set is regular and
$\eta\sigma(p)\ne 0$; \cite{krsuy}. 

{\rm (ii)} The germ of $f$ at $p$ is a swallowtail $(A_3)$ if and only if
the singular set is regular and $\eta\sigma(p)=0$ and $\eta\eta\sigma(p)\ne 0$; \cite{krsuy}.

{\rm (iii)} The germ of $f$ at $p$ is a cuspidal butterfly $(A_4)$ if and only if
the singular set is regular and $\eta\sigma(p)=\eta\eta\sigma(p)=0$ and $\eta\eta\eta\sigma(p)\ne 0$; \cite{izumiyasaji}.

{\rm (iv)} The germ of $f$ at $p$ is a cuspidal lips $(A_3^-)$ if and only if 
$\sigma$ has a Morse singularity of index zero or two at $p$; \cite{izsata}.

{\rm (v)} The germ of $f$ at $p$ is a cuspidal beaks $(A_3^-)$ if and only if $\sigma$  has a Morse singularity of index one at $p$ and $\eta\eta\sigma(p)\ne 0$; \cite{izsata}.
\end{theorem}

\begin{remark}\label{rem:sing_f&N}
It is worth observing that for a singular pseudospherical wave front surface $f$, although the singular sets of $f$ and $N$ coincide, by \eqref{assocfrontaleqn} the kernel direction of $\dd f_p$ and that of $\dd N_p$ are orthogonal at points $p$ on their singular set. Therefore, we should not expect the singularities of $f$ and those of $N$ to be related in general. (See \cite{KabataR2R2} for recognition criteria for singularities of maps-germs from the plane to the plane.)  
\end{remark}

\subsection{Construction of germs of analytic Lorentzian-harmonic maps}
To analyze the local singularities of Lorentzian harmonic maps and their associated pseudospherical surfaces it will be useful to have a characterization of an arbitrary $k$-jet of such maps.
Let $N:\Omega\subset \mathbb{R}^{1,1}\to S^2$ be a Lorentzian-harmonic map and $(x,y) $ a null coordinate system in $\Omega$. 
We are interested in the local singularities of $N$, so we suppose that $O=(0,0)\in \Omega$ 
and that $N(0,0)=(0,0,1)$. Then we can write $N$ locally at $O$ in the form 
$N(x,y)=\delta(x,y)(u(x,y),v(x,y),1)$, with $u,v$ analytic functions 
on $\Omega$ vanishing at the origin, and $\delta=(1+u^2+v^2)^{-\frac{1}{2}}$. The Lorentzian-harmonic condition $N\times N_{xy}=0$ is equivalent to the following system of semi-linear PDEs:
\begin{equation}\label{sys:Pdes}
\begin{cases}
\displaystyle{u_{xy}-\frac{1}{1+u^2+v^2}\left(2uu_xu_y+v(u_xv_y+u_yv_x)\right)=0,}\\
\displaystyle{v_{xy}-\frac{1}{1+u^2+v^2}\left(2vv_xv_y+u(u_xv_y+u_yv_x)\right)=0.}
\end{cases}
\end{equation}

We write $j^nu(x,y)=\sum_{k=1}^n\sum_{i=0}^ka_{ki}x^{k-i}y^i$ and 
$j^nv(x,y)=\sum_{k=1}^n\sum_{i=0}^kb_{ki}x^{k-i}y^i$ for the $n$-jets, at the origin, of $u$ and $v$ respectively. Then
using \eqref{sys:Pdes}, one can show that 
\begin{equation}\label{n-jets(u,v)}
\begin{split}
j^nu(x,y)&=\sum_{k=1}^n a_{k0}x^{k}+\sum_{k=1}^na_{kk}y^k+
\sum_{k=3}^n\sum_{i=1}^{k-1}P_{ki}(a,b)x^{k-i}y^i,\\
j^nv(x,y)&=\sum_{k=1}^n b_{k0}x^{k}+\sum_{k=1}^nb_{kk}y^k+
\sum_{k=3}^n\sum_{i=1}^{k-1}P_{ki}(b,a)x^{k-i}y^i,
\end{split}
\end{equation}
where $a_{ki}=P_{ki}(a,b), 1\le i\le k-1$, are polynomial functions in $a_{l0},a_{ll},b_{l0},b_{ll}$, 
with $1\le l\le k-2$, and $b_{ki}=P_{ki}(b,a)$. 
(Observe that $u_{xy}=v_{xy}=0$ at the origin, so $a_{21}=b_{21}=0$, and in \eqref{sys:Pdes}, the second equation can be obtained from the first one by interchanging $u$ and $v$; that is why $b_{ki}=P_{ki}(b,a)$.) 
For the explicit examples considered below (Example \ref{examples1}) we only need the $3$-jet,
which has the form:
$$
\begin{array}{rcl}
j^3u(x,y)&=&
a_{10}x+a_{11}y+a_{20}x^2+a_{22}y^2+
a_{30}x^3+a_{33}y^3+\\
&&(a_{10}a_{11}^2+\frac{1}{2}a_{10}b_{1 1}^2+\frac{1}{2}a_{1 1}b_{1 0}b_{11})xy^2+(a_{1 0}^2a_{1 1}+\frac{1}{2}a_{1 0}b_{10}b_{1 1}+\frac{1}{2}a_{11}b_{10}^2)x^2y,\\
j^3v(x,y)&=& b_{10}x+b_{11}y+b_{20}x^2+b_{22}y^2+
b_{30}x^3+b_{33}y^3+\\
&&(b_{10}b_{11}^2+\frac{1}{2}b_{10}a_{1 1}^2+\frac{1}{2}b_{1 1}a_{1 0}a_{11})xy^2
+(b_{1 0}^2b_{1 1}+\frac{1}{2}b_{1 0}a_{10}a_{1 1}+\frac{1}{2}b_{11}a_{10}^2)x^2y.
\end{array}
$$

The above considerations suggest that the $n$-jet space of germs of such Lorentzian-harmonic maps with $N(0)=N_0$ is a germ of a smooth manifold of dimension $4n$ and can be parametrised by $a_{i0},a_{ii},b_{i0},b_{ii}$, for $1\le i\le n$, all in $\mathbb R$. (If we allow $N_0$ to vary in $S^2$, then the dimension becomes $4n+2$.) Indeed,

\begin{theorem}\label{theo:kjetLorentzH} For any 
$({\bf a}_1,{\bf a}_2,{\bf b}_1,{\bf b}_2)\in \mathbb R^{4n}$, with 
${\bf a}_1=(a_{i0})_{1\le i\le n}$, 
${\bf a}_2=(a_{ii})_{1\le i\le n}$, 
${\bf b}_1=(b_{i0})_{1\le i\le n}$, and
${\bf b}_2=(b_{ii})_{1\le i\le n}$, 
there is a local analytic Lorentzian-harmonic map into the 2-sphere determined by $({\bf a}_1,{\bf a}_2,{\bf b}_1,{\bf b}_2)$ with $n$-jet as in \eqref{n-jets(u,v)}.
\end{theorem} 

\begin{proof}
We consider the Cauchy problem given by the PDE \eqref{sys:Pdes} with the following conditions along the non characteristic curve $(t,t)$:
\begin{equation}\label{PDEintialdata}
\begin{cases}
u(t,t)=\sum_{i=1}^n\alpha_it^i,\\ 
u_x(t,t)=\sum_{i=0}^{n-1}\beta_{i+1}t^i,\\ 
v(t,t)=\sum_{i=1}^n\lambda_it^i,\\ 
v_x(t,t)=\sum_{i=0}^{n-1}\mu_{i+1}t^i.
\end{cases}
\end{equation}

We set $\alpha=(\alpha_i)_{1\le i\le n}$, 
$\beta=(\beta_i)_{1\le i\le n}$,
$\lambda=(\lambda_i)_{1\le i\le n}$,
$\mu=(\mu_i)_{1\le i\le n}$.
Let $(u,v)$ be an analytic solution to the above Cauchy problem (which exists by Cauchy-Kowalevski's Theorem).
Then the $n$-jets of $u$ and $v$ must be in the form \eqref{n-jets(u,v)}. 
We have by \eqref{PDEintialdata}
$j^nu(t,t)=\sum_{i=1}^n\alpha_it^i,$ $j^{n-1}u_x(t,t)=\sum_{i=0}^{n-1}\beta_{i+1}t^i,$ 
$j^n v(t,t)=\sum_{i=1}^n\lambda_it^i,$  $j^n v_x(t,t)=\sum_{i=0}^{n-1}\mu_{i+1}t^i.$
Comparing coefficients in \eqref{PDEintialdata} gives
\begin{equation}\label{diffeo}
\begin{cases}
\alpha_i=a_{i0}+a_{ii}+\sum_{j=1}^{i-1}a_{ij},\\ 
\beta_i=ia_{i0}+\sum_{j=1}^{i-1}(i-j)a_{ij}\\ 
\lambda_i=b_{i0}+b_{ii}+\sum_{j=1}^{i-1}b_{ij},\\ 
\mu_i=ib_{i0}+\sum_{j=1}^{i-1}(i-j)b_{ij}.
\end{cases}
\end{equation}

Using the fact that $a_{ij}=P_{ij}(a,b)$ and  $b_{ij}=P_{ij}(b,a)$, $1\le j\le i-1$, are polynomial functions in $a_{l0},a_{ll},b_{l0},b_{ll}$, for $1\le l\le i-2$, it is clear that the map
$\psi:\mathbb R^{4n}\to \mathbb R^{4n}$, with 
$(\alpha,\beta,\lambda,\mu)=
\psi({\bf a}_1,{\bf a}_2,{\bf b}_1,{\bf b}_2)$ defined by the relations in \eqref{diffeo}
is a diffeomorphism. 
We solve the Cauchy problem for  $(\alpha,\beta,\lambda,\mu)=\psi({\bf a}_1,{\bf a}_2,{\bf b}_1,{\bf b}_2)$ in \eqref{PDEintialdata}. A solution to this problem gives the required analytic Lorentzian-harmonic map into the 2-sphere with $n$-jet determined by $({\bf a}_1,{\bf a}_2,{\bf b}_1,{\bf b}_2)$.
\end{proof}


\section{Singularities and bifurcations for the wave front case} \label{sec:wavefront_case}

\subsection{Pseudospherical surfaces as parallels of regular surfaces}\label{subs:pSpherical_parallels}

We seek a regular surface $g$ which has $f$ as one of its parallels. This is equivalent to finding a scalar $r$ for which the surface $g=f+r N$ is regular. With notation as in \S\ref{sec:harmonicmaps}, we have 
$$
g_x\times g_y=AB\left((1-r^2)\sin\phi +2r \cos\phi \right)N.
$$

The quadratic equation $(1-r^2)\sin\phi +2r \cos\phi =0$ in $r$ has two solutions  
$r_i=(\cos\phi +(-1)^i)/\sin\phi$, $i=1,2$.
At a regular point $p_0$ of $f$, $r_1$ and $r_2$ are the radii of curvature of $f$ at $p_0$.
At a singular point $p_0$ of $f$, $r_2$ goes to infinity and $r_1=0$. In both cases $p_0$ regular or singular point of $f$, we can choose a neighbourhood $U$ of $p_0$ on $f$ such that $r$ is not a solution of the above quadratic equation for all points in $U$. 
Then $g$ is a regular surface at points in $U$ and $f=g-rN$ is its parallel with distance $-r$. 

In $U$, the Gaussian and mean curvatures of $g$ are given by 
$$
K_g=-\frac{\sin\phi}{(1-r^2)\sin\phi +2r \cos\phi},\qquad H_g=\frac{r\sin\phi-\cos\phi}{(1-r^2)\sin\phi +2r \cos\phi}.
$$

We have some immediate consequences from this: 

(a) The parabolic set of $g$ corresponds to the singular set of $f$. 

(b) We have $H_g^2-K_g=1/((1-r^2)\sin\phi +2r \cos\phi)^2$, so $g$ restricted to $U$ is an umbilic free surface. Its principal directions are given by $dy^2-dx^2=0$ and its asymptotic directions are the solutions of 
$$r dy^2-2(r \cos\phi+\sin\phi)dxdy-r dx^2=0.
$$

(c) Finally, we have
$$
(1+r^2)K_g+2r H_g +1=0,
$$
so $g$ is a linear Weingarten surface.

\subsection{Bifurcations in generic 1-parameter families}

As we suppose in this section that $f$ is a wave front, it is a 
parallel of a regular surface (see \S\ref{subs:pSpherical_parallels}), so we can use
the results in \cite{bruceParallel} to study its singularities in a similar way to the case of spherical surfaces in \cite{dbft}. 

Singularities of parallels of a general surface $g: 
\Omega \to \mathbb R^3$  
are studied by Bruce in \cite{bruceParallel} (see also \cite{fukuihasegawa}).
Bruce considered the family of distance squared functions $F_{t_0}:\Omega\times \mathbb R^3\to \mathbb R$ 
given by $F_{t_0}((x,y),q)=|g(x,y)-q|^2-t_0^2$. 
A parallel $W_{t_0}$ of $g$ is 
the discriminant of $F_{t_0}$, that is,
$$
W_{t_0}=\left\{ q\in \mathbb R^3:\exists (x,y)\in \Omega 
\,\,\mbox{\rm such that}\,\, 
F_{t_0}((x,y),q)=\frac{\partial F_{t_0}}{\partial x}((x,y),q)=\frac{\partial F_{t_0}}{\partial y}((x,y),q)=0 
\right\} .
$$

For $q_0$ fixed, the function $F_{q_0,t_0}(x,y)=F_{t_0}(x,y,q_0)$ 
gives a germ of a function at a point on the surface. Varying
$q$ and $t$ gives a $4$-parameter family of functions $F$. 
Let $\mathcal R$ denote the group of germs of diffeomorphisms from the plane to the plane.
Then, by a transversality theorem in \cite{looijenga}, for a generic surface, the possible singularities of $F_{q_0,t_0}$ 
are those of $\mathcal R$-codimension 4, and these are as follows (with  $\mathcal R$-models, up to a sign, in brackets): 
$A_1^{\pm}$ ($x^2\pm y^2$), $A_2$ ($x^2+ y^3$), 
$A_3^{\pm}$ ($x^2\pm y^4$), $A_4$ ($x^2+ y^5$) and $D_4^{\pm}$ ($y^3\pm x^2y$).  

When the family $F_{t_0}$ is an $\mathcal R$-versal deformation of the $A_3$-singularity, the parallel is a swallowtail. It can happen that $F_{t_0}$ fails to be an $\mathcal R$-versal deformation of the $A_3$-singularity. In this case we denote the singularities by {\it non-transverse $A_3^{\pm}$}.

Bruce showed that $F$ is always an $\mathcal R$-versal family 
of the $A_1^{\pm}$ and $A_2$ singularities. Consequently the parallels at such singularities are, respectively, regular surfaces or cuspidal edges. In particular, the $A_2$-transitions in wave fronts (\cite{ArnoldWavefront}) do not occur on parallels of surfaces (\cite{bruceParallel}).

For the codimension 1 singularities in parallels, we observe that, for pseudospherical surfaces, once $g=f+rN$ is fixed (i.e., once $r$ is fixed), the only parallel of $g$ with constant negative Gaussian curvature is $f$. Therefore, one needs to consider  1-parameter families of pseudospherical surfaces
in order to possibly realise the generic bifurcations that occur in parallels of surfaces by varying $r$.

\begin{proposition}\label{prop:NoD_4}
The $D_4^{\pm}$-singularities do not occur on pseudospherical surfaces, either in the wave front or non-wave front cases.
\end{proposition}

\begin{proof} 
The surface $f$ has a $D_4^{\pm}$-singularity if it is locally diffeomorphic to the image of the map $(s,t)\mapsto (st,s^2 \mp 3t^2,s^2t \mp  t^3)$ (see for example \cite{fukuihasegawa}). Then $f$ has rank 0 at the $D_4^{\pm}$-singularity and we can show that $N$ has rank 2 at that point. This cannot happen on pseudospherical surfaces as $f$ and $N$ have the same rank. 
\end{proof}

\begin{remark}
In the case when $f$ is a wave front, 
\S \ref{subs:pSpherical_parallels}
provides another argument why the $D_4^{\pm}$-singularities do not occur on pseudospherical surfaces. Such singularities occur at umbilic points of the surface $g$, and we pointed out that $g$ is an umbilic free surface. 
\end{remark}

\begin{theorem}\label{theo:bif_Wavefront_case}
(1) The stable singularities $A_2$ (cuspidal edge) and $A_3$ (swallowtail) can occur on pseudospherical surfaces (Figure \ref{fig:StablePSS} and Figure \ref{figCSB}, \cite{singps,ishimach}). 

(2) The codimension 1-singularities 
non-transverse $A_3^+$ (cuspidal lips), non-transverse $A_3^-$ (cuspidal beaks) and $A_4$ (cuspidal butterfly) can occur on 
pseudospherical surfaces. 

(3) The evolution of parallels at a non-transverse $A_3^{\pm}$ and $A_4$ can be  realized in generic  1-parameter families of pseudospherical surfaces;  see Figure \ref{fig:familiesPSS} and Figure \ref{figbutterflybif}. 

(4) The cuspidal lips, beaks and butterfly bifurcations are the only ones that can occur in generic 1-parameter family of pseudospherical wave fronts.
\end{theorem}

\begin{proof}
The type of a singularity of $f$ at a given point $p$ is determined by a certain $n$-jet of $f$ at $p$. By Theorem \ref{theo:kjetLorentzH}, 
we have all the possible $n$-jets of $N$, and hence of $f$ at any given point on $f$.  
We use the setting in Theorem \ref{theo:kjetLorentzH} to express the conditions for a given singularity (of a wave front) to occur on $f$ in terms of the coefficients 
that determine $N$. Theorem \ref{theo:kjetLorentzH} assures that such Lorentzian-harmonic maps exist.

Items (1) and (2): Using the setting in Theorem \ref{theo:kjetLorentzH}, the 1-jet of a defining equation of the singular set $\Sigma$ of $f$, which is the same as that of $N$, is given by
$$
j^1\sigma=b_{11}a_{10}-b_{10}a_{11}+2(a_{20}b_{11}-a_{11}b_{20})x+2(a_{10}b_{22}-a_{22}b_{10})y.
$$

The pseudospherical surface is singular at the origin if, and only if, 
\begin{equation}\label{eq:Sigma}
b_{11}a_{10}-b_{10}a_{11}=0.
\end{equation}

Suppose that $f$ is singular at the origin and that  $\Sigma$ is a regular curve, that is, 
$a_{20}b_{11}-a_{11}b_{20}\ne 0$ or 
$a_{10}b_{22}-a_{22}b_{10}\ne 0$. We compute the 3-jet of $\sigma$ as well as that of a vector field $\eta$ giving the null direction 
along $\Sigma$.
We have $\eta$ transverse to $\Sigma$ at the origin if, and only if, 
\begin{equation}\label{open_Cond_CuspidalEdge}
\begin{cases}
b_{11}a_{10}-b_{10}a_{11}=0,\\
b_{11}(a_{20}b_{11}-a_{11}b_{20})+b_{10}(a_{10}b_{22}-a_{22}b_{10})\ne 0.
\end{cases}
\end{equation}
Then $f$ is a cuspidal edge by Theorem \ref{theo:critiria}(i). 

The null direction $\eta$ has first order contact with $\Sigma$ at the origin if, and only if, 
\begin{equation}\label{opnen_Cond_Swall}
\begin{cases}
b_{11}a_{10}-b_{10}a_{11}=0,\\
b_{11}(a_{20}b_{11}-a_{11}b_{20})+b_{10}(a_{10}b_{22}-a_{22}b_{10})=0,\\
6b_{33}a_{10}b_{10}^2-6b_{30}a_{11}b_{11}^2-
6a_{33}b_{10}^3+
6a_{30}b_{11}^3+\\
4b_{20}b_{22}a_{10}b_{11}-
4b_{20}b_{22}a_{11}b_{10}+
12a_{20}b_{22}b_{10}b_{11}-12a_{22}b_{20}b_{10}b_{11}+\\
a_{10}^3b_{11}^3+3a_{10}^2a_{11}b_{10}b_{11}^2-3a_{10}a_{11}^2b_{10}^2b_{11}
+6a_{10}b_{10}^2b_{11}^3- a_{11}^3b_{10}^3-
6a_{11}b_{10}^3b_{11}^2\ne 0.
\end{cases}
\end{equation}
Then $f$ is a swallowtail by Theorem \ref{theo:critiria}(ii). 
(Observe that, in general, the second condition in \eqref{opnen_Cond_Swall} is  distinct from that for $N$ to have a cusp singularity, see the proof of Theorem \ref{theo:singLOrzHar} and Remark \ref{rem:sing_f&N}.)

The null direction $\eta$ has second order contact with $\Sigma$ at the origin if, and only if, the first two conditions in \eqref{opnen_Cond_Swall} are satisfied, the left hand side of the third vanishes and a polynomial in $a_{i,0},a_{i,i},b_{i,0},b_{i,i}$, $i\le i\le 4$, does not vanish (the polynomial is too lengthy to reproduce here, but we can choose $N$ so that it does not vanish).
When this happens, the singularity of $f$ is a cuspidal butterfly by Theorem \ref{theo:critiria}(iii).

\smallskip

Consider now $\Sigma$ singular, that is 
$b_{11}a_{10}-b_{10}a_{11}= a_{20}b_{11}-a_{11}b_{20}=a_{10}b_{22}-a_{22}b_{10}=0$. We can take, without loss of generality, $a_{10}\ne 0$, so that the 2-jet of $\sigma$ becomes
$$
j^2\sigma=-3\frac{a_{11}}{a_{10}}
(a_{10}b_{30}-a_{30}b_{10})x^2+
3(a_{10}b_{33}-a_{33}b_{10})y^2.
$$

Clearly, $\sigma$ can have a Morse singularity or type $A_1^+$ or $A_1^-$ 
(provided $(a_{10}b_{30}-a_{30}b_{10})\ne 0$ and $(a_{10}b_{33}-a_{33}b_{10})\ne 0$. We cannot have $a_{11}=0$ as that would imply $b_{11}=0$, that is, $N_y(0,0)=0$). When it has an 
$A_1^+$-singularity,  $f$ has a cuspidal lips by  Theorem \ref{theo:critiria}(iv). 
At an $A_1^-$-singularity, the null direction $(a_{11},a_{10})$ 
is transverse to the branches of $\Sigma$ if, and only if, 
$$
(a_{10}b_{30}-a_{30}b_{10})a_{11}^3-
(a_{10}b_{33}-a_{33}b_{10})a_{10}^3\ne 0.$$ When this is the case, the singularity of $f$ is a cuspidal beaks by Theorem \ref{theo:critiria}(v).

Item (3): For the realization of the generic bifurcations of the singularities non-transverse $A_3^{\pm}$ and $A_4$, we established in Theorem 4.1 and Theorem 4.2 in \cite{dbft} general geometric criteria for determining when such bifurcations are realized. The criteria depend on certain $n$-jets of the family of surfaces. 
We use those criteria 
to construct generic 1-parameter families of Lorentzian-harmonic maps using Theorem \ref{theo:kjetLorentzH}. These give the 1-parameter families of pseudospherical surfaces 
which realize the generic bifurcations of parallels at non-transverse $A_3^{\pm}$ and  $A_4$ singularities.   
In \S \ref{sec:Loop_Groups} we use the loop group method to construct examples of 1-parameter families of pseudospherical surfaces that realise these evolutions.

Item (4):
A priori, singularities more degenerate than those occurring in generic 1-parameter families of parallels could occur generically in 1-parameter families of pseudospherical surfaces. However, for such singularities to occur, we will require at least 4 conditions on the coefficients of $j^kN$. A transversality argument (see \S \ref{sec:frontal} for details) shows that such singularities can be avoided  on pseudospherical wave fronts.  Therefore, the only non-stable local singularities that can occur in generic 1-parameter families of pseudospherical wave fronts are those in items (2) and (3) above.
\end{proof}

%

\section{The non-wave front case}\label{sec:frontal}

As pointed out in the introduction, there is so far no analogous theory of generating families for wave fronts (\cite{arnoldetal}) that deals with 
deformations of frontals. 
When the pseudospherical surface $f$ is not a wave front we will use transversality in the jet space to define the notions of codimension and  generic families.

A pseudospherical surface $f$ is determined by a Lorentzian-harmonic map, which in turn is determined 
locally at each point by a pair of functions $(u,v)$. According to Theorem \ref{theo:kjetLorentzH}, the $k$-ket of
$N$ at a given point is determined by the $k$-jets at that point of the four 
functions obtained from $u,v$ by fixing one of the variables.

Given an $m$-parameter family of pseudospherical surfaces $f^s$, we associate to each member 
of the family a pair of functions $(u^s,v^s)$ which determine the Lorentzian-harmonic map associated to $f^s$. 
Let $J^k(p,q)$ denote the space of $k$-jets of maps from $\mathbb R^p$ to $\mathbb R^q$,
and define the 
family of Monge-Taylor maps,
$$
\Phi: (\mathbb R\times \mathbb R\times \mathbb R^m,((0,0),0))\to J^k(1,2)\times J^k(1,2),
$$
given by $\Phi(x,y,s)=\phi^s(x,y)$, where 
$$
\phi^s(x,y)=\left((j^k_xu^s(-,y), j^k_xv^s(-,y)), (j^k_yu^s(x,-), j^k_yv^s(x,-))\right)
$$
with $j^k_xu^s(-,y)$ (resp. $j^k_yu^s(x,-)$)  denoting the Taylor polynomial of order $k$ at $x$ (resp. $y$) of $u^s$ with $y$ (resp. $x$) fixed.

The conditions for a pseudospherical surface $f$ to have a certain type of singularity at a point $(x,y)$ are expressed in terms of the   
coefficients in the Taylor expansions of the 
functions obtained from $u,v$ by fixing one of the variables. These conditions define a variety $V$ in $J^k(1,2)\times J^k(1,2)$ (we take $k$ sufficiently large). We say that a singularity of $f=f^0$ at the origin with $\phi^0(0,0)\in V$ is of {\it codimension $m$} if $m$ is the least integer for which there exists an $m$-parameter family 
$f^s$ of pseudospherical surfaces, with $s$ near zero in $\mathbb R^m$, such that the associated family $\Phi$ above is transverse to $V$. We call the family $f^s$ a {\it generic deformation of} the singularity of $f^0$.

We identify $J^k(1,2)\times J^k(1,2)$ with $\mathbb R^{2k}\times \mathbb R^{2k}$.

\begin{proposition}\label{prop:Tg_MTmap}
The tangent space to the image of the Monge-Taylor map $\phi=\phi^0$ is generated by $\phi_x(0,0)$ and $\phi_y(0,0)$ with
$$
\begin{array}{rcl}
\phi_x(0,0)&=&\left((j^k_xU_1(-,0), j^k_xV_1(-,0)), (j^k_yU_1(0,-), j^k_yV_1(0,-))\right),\\
\phi_y(0,0)&=&\left((j^k_xU_2(-,0), j^k_xV_2(-,0)), (j^k_yU_2(0,-), j^k_yV_2(0,-))\right),
\end{array}
$$
where
$$
\begin{array}{rcl}
U_1&=&u_x(x,y)-u_x(0,0)-(u_x(0,0)u(x,y)+v_x(0,0)v(x,y))u(x,y),\\
U_2&=&u_y(x,y)-u_y(0,0)-(u_y(0,0)u(x,y)+v_y(0,0)v(x,y))u(x,y),\\
V_1&=&v_x(x,y)-v_x(0,0)-(u_x(0,0)u(x,y)+v_x(0,0)v(x,y))v(x,y),\\
V_2&=&v_y(x,y)-v_y(0,0)-(u_y(0,0)u(x,y)+v_y(0,0)v(x,y))v(x,y).
\end{array}
$$
\end{proposition}

\begin{proof}
The functions $u,v$ give the coordinates $(u,v,1)$ of $N$ in the tangent plane to the sphere at $(0,0,1)$. Denote by $N_0=\delta_0(u_0,v_0,1)$ the vector $N(x_0,y_0)$, for $(x_0,y_0)$ near the origin. Then the coordinates $(\tilde{u}(X,Y),\tilde{v}(X,Y),1)$ of $N(X+x_0,Y+y_0)$ in the tangent plane to the sphere at $N_0$ are given by
$$
\begin{array}{rcl}
\tilde{u}(X,Y)&=&\frac{1}{\lambda(X,Y)}(u(X+x_0,Y+y_0)-u_0),\\
\tilde{v}(X,Y)&=&\frac{1}{\lambda(X,Y)}(-u_0v_0u(X+x_0,Y+y_0)+(1+u_0^2)v(X+x_0,Y+y_0)-v_0),
\end{array}
$$
with $\lambda(X,Y)=\delta_0(1+(u_0)^2)^{\frac{1}{2}}(1+u_0u(X+x_0,Y+y_0)+v_0v(X+x_0,Y+y_0))$.
The result follows by differentiating $\tilde{u}$ and $\tilde{y}$ with respect to $x_0$ and $y_0$ 
and evaluating at the origin.
\end{proof}

\begin{lemma} \label{lem:rank0_codim>=2}
A singularity at $p$ of a  pseudospherical surface $f$ with ${\rm rank} (\dd f_p)={\rm rank} (\dd N_p)=0$ is of codimension $\ge 2$.
\end{lemma}

\begin{proof}
In the setting of \S \ref{sec:harmonicmaps} with $p$ the origin, ${\rm rank} (\dd f_p)={\rm rank} (\dd N_p)=0$ if, and only if, $a_{10}=a_{11}=b_{10}=b_{11}=0$. This gives a codimension $4$ variety $V$ in $J^k(2,1)\times J^k(2,1)$, $k\ge 1$.  Using Proposition \ref{prop:Tg_MTmap}, we find that any family $f_s$ with $f=f_0$ has to be of at least 2-parameters for $\Phi$ to be transverse to $V$.
\end{proof}

As we are interested here in codimension 1 singularities,  we must have ${\rm rank} (\dd f_p)=1$. 
Also,$f$ is not a wave front at $p$ if and only if
 either $N_x$ or $N_y$ vanishes at $p$ (Proposition \ref{prop:f_frontal_wavefront}). 

\begin{theorem}\label{theo:2/5cuspidaledge}
Suppose that $N$ has rank 1, $N_y=0$ or $N_x=0$ and that $\Sigma$ is regular (so it is locally a null curve \cite{singps}). Then generically $f$ is locally a $2/5$-cuspidal edge, i.e., it is locally diffeomorphic to the image of the map $(x,y)\mapsto (x,y^2,y^5)$.  

The bifurcations of $f$ in a generic 1-parameter family $f^s$ of pseudospherical surfaces with $f^0=f$, are as shown in Figure \ref{fig:bif25CuspEdge}. The singular set of $f^s$ is an ordinary 2/3 cuspidale edge for $s\ne 0$, and there is a birth of a double point curve on the surface $f^s$ on one side of the bifurcation ($s<0$ or $s>0$). The singular curve in the source is timelike on one side of the transition and spacelike of the other side.
\end{theorem}

\begin{proof}
We suppose, without loss of generality, that the point of interest is the origin and that $N_y=0$ and $N_x\ne 0$ at that point. As $\Sigma$ is supposed to be regular (equivalently $[N,N_x,N_{yy}]\ne 0$ at the origin), it is locally the null curve $y=0$, see \cite{singps}.

We consider the orthonormal frame $N(0), N_x(0)/A, N(0)\times N_x(0)/A$, with $A=|N_x(0)|$. The tangent to the image of the singular set at the origin is $f_x(0)=N(0)\times N_x(0)$. We intersect the image of $f$ with the plane at the origin generated by 
$N(0), N_x(0)$. This gives a curve $\gamma$ on $f$ parametrised by $\gamma(y)=(\langle f(0,y), N(0)\rangle, \langle f(0,y), N_x(0)/A\rangle)$. It follows by the hypothesis and using the fact that $N$ is Lorentzian-harmonic that
$$
\gamma(y)=\left(\frac{2}{5!}[N(0),N_{yy}(0),N_{yyy}(0)]y^5+O(6),\frac{1}{2A}[N(0),N_{yy}(0),N_x(0)]y^2+O(3)\right),
$$
where $O(l)$ denotes a remainder of order $l$. 
Therefore, the curve $\gamma$ is $\mathcal A$-equivalent to $(y^2,y^5)$ provided 
$[N(0),N_{yy}(0),N_x(0)]\ne 0$ (which we assumed already) and $[N(0),N_{yy}(0),N_{yyy}(0)]\ne 0$. These conditions are satisfied by generic pseudospherical surfaces. 
Observe that in the above calculations there is nothing special about the origin on the singular set, that is, all the local transverse sections of $f$ yield curves with singularities $\mathcal A$-equivalent to $(y^2,y^5)$. 
It follows that 
$f$ is $\mathcal A$-equivalent to $(x,y^2h(x,y),y^5k(x,y))$, with $h,k$ germs of smooth functions not vanishing at the origin. Further changes of coordinates set 
$f\sim_{\mathcal A}(x,y^2,y^5\tilde k(x,y))$, with $\tilde k(0,0)\ne 0.$ Using Theorem 4.1.1 in \cite{Mond}, we 
get $f\sim_{\mathcal A}(x,y^2,y^5(1+p(x,y^2)))$. The diffeomorphism $(u,v,w)\mapsto (u,v,w/(1+p(u,v)))$ gives $f\sim_{\mathcal A}(x,y^2,y^5)$ as required. 

The non-wave front stratum in $J^k(1,2)\times J^k(1,2)$, $k\ge 1$, is the union of $F_1: a_{11}=0, b_{11}=0$ with 
$F_2: a_{10}=0, b_{10}=0$. As we assumed $N_y=0$ and $N_x\ne 0$ ($N$ has rank 1), the component of interest is $F=F_1: a_{11}=0, b_{11}=0$.

The tangent space to $F$ is the intersection of the kernels of the 1-forms $\xi_1=da_{11}$ and 
$\xi_2=db_{11}$.

We can work in $J^1(1,2)\times J^1(1,2)\equiv\mathbb R^2\times \mathbb R^2$. Then $\phi_x(0,0)=((2a_{20},2b_{20}),(0,0))$, and  $\phi_y(0,0)=((0,0),(2a_{22},2b_{22}))$. It is clear that $\xi_1(\phi_x(0,0))=\xi_2(\phi_x(0,0))=0$, so $\phi$ is not transverse to $F$. In particular, a pseudospherical surface which is a 2/5-cuspidal edge is not stable.

Consider a 1-parameter family of pseudospherical surfaces generated by a pair of 1-parameter family of functions $u^s,v^s$ with $u^0=u$, $v^0=v$. We have
\begin{equation}\label{Derivative_Phi_s}
\small{
\frac{\partial \Phi}{\partial s}((0,0),0)=
((j^k_x\frac{\partial u^s}{\partial s}((-,0),0), j^k_x\frac{\partial v^s}{\partial s}((-,0),0)), (j^k_y\frac{\partial u^s}{\partial s}((0,-),0), j^k_y\frac{\partial v^s}{\partial s}((0,-),0)))}.
\end{equation}

For $k=1$, and working in $J^1(1,2)\times J^1(1,2)$, we have  
$\frac{\partial \Phi}{\partial s}=((\frac{\partial^2 u^s}{\partial s\partial x},\frac{\partial^2 v^s}{\partial s\partial x}),(\frac{\partial^2 u^s}{\partial s\partial y},\frac{\partial^2 v^s}{\partial s\partial y}))$ at the origin.
Then $\Phi$ fails to be transverse to $F$ at $\Phi((0,0),0)$ if, and only if, there exist scalars $c,d$ such that 
$\xi_1(c\phi_y+d\frac{\partial \Phi}{\partial s})=\xi_2(c\phi_y+d\frac{\partial \Phi}{\partial s})=0$ at the origin, alternatively,  $\xi_1(\phi_y)\xi_2(\frac{\partial \Phi}{\partial s})-
\xi_1(\frac{\partial \Phi}{\partial s})\xi_2(\phi_y)=0$ at the origin. Therefore, $\Phi$ is  transverse to $F$ at $\Phi((0,0),0)$ if, and only if,
$$
(a):\quad  \left(b_{22}\frac{\partial^2 u^s}{\partial s\partial y}-a_{22}\frac{\partial^2 v^s}{\partial s\partial y}\right)((0,0),0)\ne 0.
$$ 

Clearly, we can choose $u^s,v^s$ with this property. 
This shows that the 2/5-cuspidal edge singularity in pseudospherical surfaces is a codimension 1 phenomenon.
We need to consider the intersection of $\Phi$ with the stratum of singular Lorentzian-harmonic maps, given by $\Sigma: a_{11}b_{10}-a_{10}b_{11}=0$ (see \eqref{eq:Sigma}). At a point on $F$, the tangent space to $\Sigma$ is the kernel of the 1-form 
$\mu=-b_{10}da_{11}+a_{10}db_{11}$. 
We have $\mu(\phi_x(0,0))=0$ and $\mu(\phi_y(0,0))=-b_{10}a_{22}+a_{10}b_{22}$. 
Therefore $\phi=\phi^0$ is transverse to $\Sigma$ if, and only if, $b_{10}a_{22}-a_{10}b_{22}\ne 0$.
When this is the case, $(\phi^s)^{-1}(\Sigma)$ is a regular curve on the pseudospherical surface $f^s$ for all $s$ near $0$. 
We know from \cite{singps} that $(\phi^0)^{-1}(\Sigma)$ is a null curve of $f^0$. We take a 
family of pseudospherical surfaces with $\Phi$ transverse to $F$. Then, 
if 
$$(b): \quad  \left(b_{10}\frac{\partial^2 u^s}{\partial s\partial x}-a_{10}\frac{\partial^2 v^s}{\partial s\partial x}\right)((0,0),0)\ne 0,
$$ 
$(\phi^s)^{-1}(\Sigma)$ is timelike for $s<0$ and spacelike for $s>0$ or vice-versa. We can choose $(u^s,v^s)$ satisfying conditions (a) and (b). 

To complete our study, we need to consider the multi-local singularities of the members of the family $f^s$, namely the double point curve of $f^s$. Multi-local local singularities are hard to deal with using the jet space method as they do not depend on a finite number of the coefficients of the Taylor expansion of the map-germ. As the double point curve is $\mathcal A$-invariant, we proceed as follows.

Any 1-parameter family $f^s$ of $f=f^0$ is $\mathcal A$-equivalent a family 
$g(x,y,s)=(x,y^2,q(x,y,s))$ (this follows from the fact the map-germ from the plane to the plane $(x,y)\mapsto (x,y^2)$ is stable).
We can take $q(x,y,0)=y^5$. 
It follows from the fact that the singular set is $\mathcal A$-invariant that $q_y(x,y,s)=y(5y^3+sh(x,y,s))$ 
for some germ of a family of analytic functions $h$. 
Therefore, $g(x,y,s)=(x,y^2,y^5+sy^2k(x,y,s))$.

The map $g^s(x,y)=g(x,y,s)$ has a double point 
$g^s(x_1,y_1)=g^s(x_2,y_2)$ with $(x_1,y_1)\ne (x_2,y_2)$ if, and only if, $x_1=x_2=x$, $y_2=-y_1$ and
\begin{equation}\label{eq:dpt2/5}
y_1^5+sy_1^2k(x,y_1,s)=-y_1^5+sy_1^2k(x,-y_1,s).
\end{equation}
Writing $k(x,y,s)=yp(x,y^2,s)+l(x,y^2,s)$, equation \eqref{eq:dpt2/5} becomes $y_1^3(y_1^2+sp(x,y_1^2,s))=0$.
Therefore, if $p(0,0,0)<0$ (resp. $p(0,0,0)>0$), there are two regular curves
 in the source which are mapped to the same regular curve in the target (the double point curve of $g^s$) for $s>0$ and none for $s<0$ (resp. vice-versa).

Using our initial setting, and reducing the initial $k$-jets of $f^s$ to the required form, we find that the condition $p(0,0,0)\ne 0$ is given by
$$
b_{10}b_{22}\frac{\partial^2 u^s}{\partial s\partial y}((0,0),0)+
(2a_{22}b_{10}-3a_{10}b_{22})\frac{\partial^2 v^s}{\partial s\partial y}((0,0),0)\ne 0.
$$

Clearly, we can choose the family $f^s$ to satisfy also this condition.
\end{proof}

A surface 
$f:\mathbb R^2,0\to \mathbb R^3,0$ is called a {\it Shcherbak surface} if $f$ is $\mathcal A$ equivalent to the map-germ $(x,y)\mapsto (x,xy^2+y^3,xy^4+\frac{6}{5}y^5)$; see Figure \ref{fig:Shcherbak} (middle) and Figure \ref{figshcherbakbif}.

\begin{theorem}\label{theo:2/5CuspidalBeaks}
Suppose that $N$ has rank 1 and that $N_y=0$ or $N_x=0$. In generic 1-parameter families of 
pseudospherical surfaces $f^s$, the singular set of $f=f^0$ can have a Morse $A_1^{-}$-singularity, where one of the branches is a null curve and the other is  transverse to both null curves at the singular point. 
Then $f^0$ is a Shcherbak surface.
The deformations in the family $f^s$ as $s$ varies near zero are as shown in  Figure \ref{fig:Shcherbak} and Figure \ref{figshcherbakbif}. We have a birth of two swallowtail singularities on one side of the transition and none on the other side. On the side where there are no swallowtail singularities, we have a birth of two cusp singularities of the double point curve.
\end{theorem}

\begin{proof}
Following the proof of Theorem \ref{theo:2/5cuspidaledge} and the notation there, the singular set is singular when $\phi=\phi^0$
is not transverse to the variety $\Sigma$ in $J^k(2,1)\times J^k(2,1)$, and this happens when $b_{10}a_{22}-a_{10}b_{22}=0$.
Then, the family $\Phi$ is transverse to $\Sigma$ if, and only if, $\mu(\frac{\partial \Phi}{\partial s})\ne 0$, that is,
$b_{10}\frac{\partial^2 u^s}{\partial s\partial y}-a_{10}\frac{\partial^2 v^s}{\partial s\partial y}\ne 0$. As, $b_{10}a_{22}-a_{10}b_{22}=0$, transversality occurs  if, and only if, 
$b_{22}\frac{\partial^2 u^s}{\partial s\partial y}-a_{22}\frac{\partial^2 v^s}{\partial s\partial y}\ne 0$, that is if, and only if, $\Phi$ is transverse to $F$. We can choose a 1-parameter family $f^s$ so that $\Phi$ is transverse to $\Sigma$ (and hence to $F$), so the singularity of $f=f^0$ is of codimension 1.

The swallowtail stratum is given by the first two equations in \eqref{opnen_Cond_Swall}, namely
$$
SW:
\begin{cases}
b_{11}a_{10}-b_{10}a_{11}=0,\\
b_{11}(a_{20}b_{11}-a_{11}b_{20})+b_{10}(a_{10}b_{22}-a_{22}b_{10})=0.
\end{cases}
$$

Clearly $\phi$ intersects the $SW$-stratum when $\phi(0,0)\in F\cap \Sigma$.
Following similar calculations to those in the proof of Theorem \ref{theo:2/5cuspidaledge}, we can show that $SW$ is a regular codimension 2 variety. 
As $\phi$ is not transverse to $\Sigma$ at $\phi(0,0)$, it is not transverse to the $SW$-stratum. The family $\Phi$ is transverse to this stratum if, and only if, it is transverse to $\Sigma$.

\begin{figure}[tp]
    \centering
    \includegraphics[height=5cm]{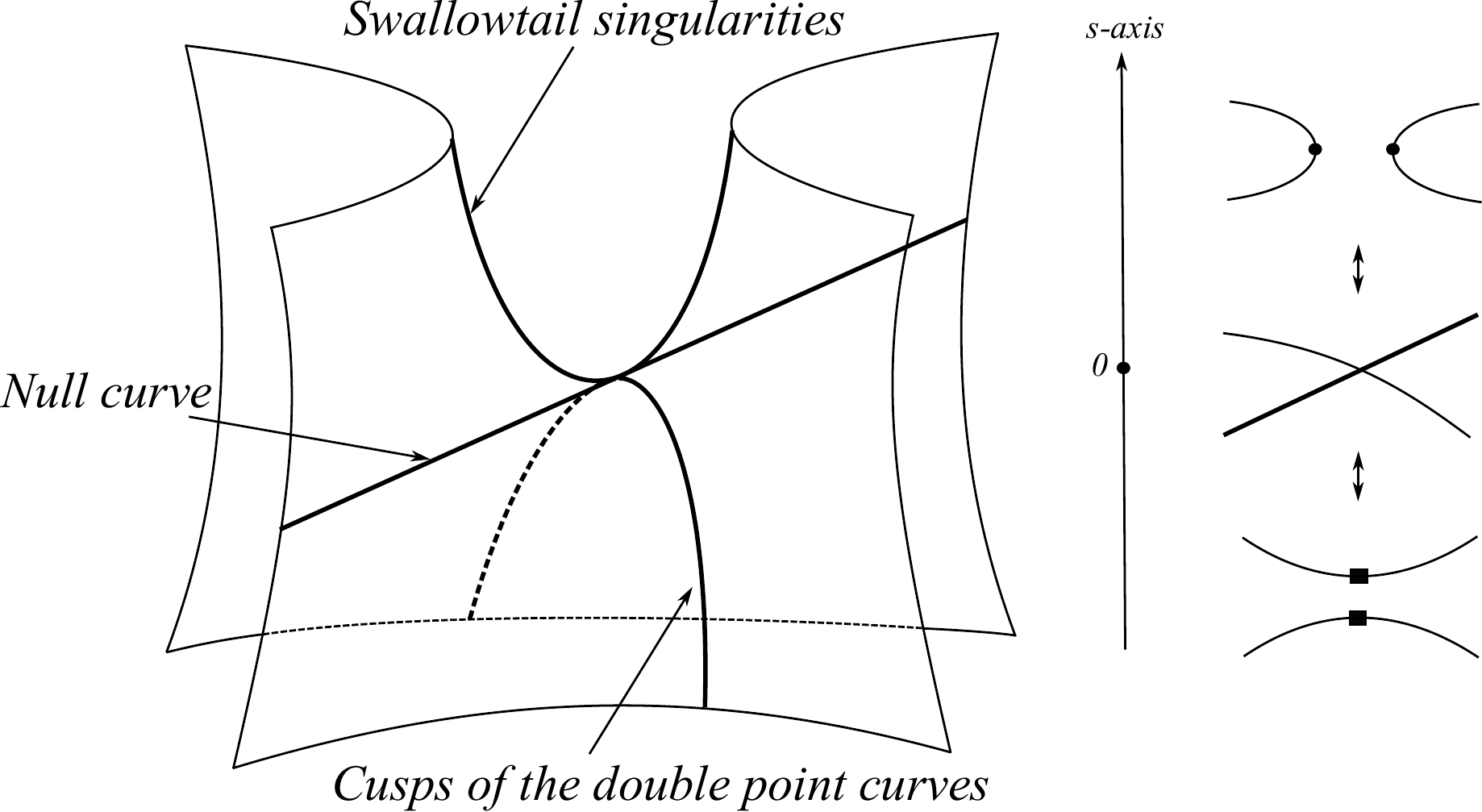}
    \caption{The surface formed by the singular sets of $f^s$ (left) and its sections by the planes $s=constant$ (right). The swallowtail singularities (resp. cusps of the double point curve) are represented by the circular (resp. square) dots.}
    \label{fig:BigSingSet}
\end{figure}

We choose a 1-parameter family of pseudospherical surfaces $f^s$, with $f=f^0$ so that $\Phi$ is transverse to $\Sigma$. 
Then $\Phi^{-1}(\Sigma)$ is a smooth surface and the pre-images of $F$ and $SW$ strata are regular curves on $\Phi^{-1}(\Sigma)$  (see Figure \ref{fig:BigSingSet}). 

The plane $s=0$ has generically a Morse $A_1^{-}$-contact with $\Phi^{-1}(\Sigma)$ and a Morse contact with the curve $\Phi^{-1}(SW)$. Indeed, assuming without loss of generality that $a_{10}\ne 0$, the
2-jet of the equation of the singular set is given, for $s=0$, by
$$-4\frac{a_{22}}{a_{10}}(a_{10}b_{20}-a_{20}b_{10})xy+(3a_{10}b_{33}-3a_{33}b_{10})y^2,
$$
so we have a Morse $A_1^-$-contact if, and only if, ${a_{22}}(a_{10}b_{20}-a_{20}b_{10})\ne 0$. Suppose that this is the case. Then 
following the algorithm for recognition of the swallowtail singularity, we find that the 2-jet of a parametrisation of $\Phi^{-1}(SW)$ has the form
$$
\left(\frac{3a_{10}(a_{10}b_{33}-a_{33}b_{10})}{2a_{22}(a_{10}b_{20}-a_{20}b_{10})}y+
\Lambda y^2,y,-3y^2(a_{10}b_{33}-a_{33}b_{10})\right),
$$
with $\Lambda$ an irrelevant constant. 
The curve above has a Morse contact with the plane $s=0$ if, and only if, $a_{10}b_{33}-a_{33}b_{10}\ne 0$, that is, the branch of the singular set 
which is not a null curve is transverse to the other null curve.
The above conditions are also satisfied in generic 1-parameter families of pseudospherical surfaces. 

To show that $f^0$ is a Shcherbak surface, we use Theorem \ref{theo:kjetLorentzH} to obtain the 5-jet of $N$ given all the conditions imposed on $a_{ij}$ and $b_{ij}$ in this proof (geometrically, we impose that the singular set has a Morse $A_1^-$-singularity with one branch a null curve and 
the other branch transverse to the other null curve). We then integrate  \eqref{assocfrontaleqn} to obtain the 5-jet of $f^0$. We find that $j^5f^0
\sim_{\mathcal A^{(5)}}(x,xy^2+y^3,xy^4+\frac{6}{5}y^5)$ (see \S \ref{sec:Classification} for notation). 
As one of the cuspidal edges of $f^0$ is a $2/5$-cuspidal edge and the other one is an ordinary cuspidal edge, it follows by Ishikawa's criteria for recognition of a Shcherbak surface (Theorem 3.19 in \cite{IshikawaRecogFrontal}) that 
$f^0\sim_{\mathcal A}(x,xy^2+y^3,xy^4+\frac{6}{5}y^5)$, that is, $f^0$ is a Shcherbak surface.

We now deal with the double point curve of $f^s$. Following the method in \cite{BruceMarar}, the equation of double point curve of the model Shcherbak surface $(x,y)\mapsto (x,xy^2+y^3,xy^4+\frac{6}{5}y^5)$ is given by $y^2(x^2-2xy-4y^2)=0$. This means that double point curve of $f^0$ 
has a non-isolated singularity, so it does not have a model for its deformations with a finite number of parameters. (Also, the method in \cite{BruceMarar} is hard to use here as we do not have an explicit formula for the defining equation of the image of $f^s$.) In what follows, we determine the locus of points in the $(x,y,s)$-space where the double point curve of $f^s$ has a cusp singularity, i.e., where $f^s$ has a multi-local singularity of type $A_0A_2$, which consists of a local  intersection of a cuspidal edge with an immersion.

We take the family $f^s$ as above, so its singular set undergoes a generic Morse transition. Then $f^s$ is equivalent, by parametrised changes of coordinates in the source and target, to a family $g(x,y,s)=(x,y^3+xy^2+sy,q(x,y,s))$ with 
$q(x,y,0)=xy^4+\frac{6}{5}y^5$. (This follows from the fact that 
the family of map-germs $(x,y,s)\mapsto (x,y^3+xy^2+sy)$ is an $\mathcal A_e$-versal family of the beaks-singularity $(x,y)\mapsto (x,y^3+xy^2)$.)
The singular set of $g^s$ is given by 
$3y^2+2xy+s=0$ ($g^s$ is $\mathcal A$-equivalent to $f^s$), therefore
$q_y(x,y,s)=(3y^2+2xy+s)(2y^2+sk(x,y,s))$ for some germ of a family of analytic functions $k$.
Integrating by parts, we get that 
$q(x,y,s)=P(x,y,s)+sQ(x,y,s)$ with
$$
\begin{array}{l}
P(x,y,s)=xy^4+\frac{6}{5}y^5+\frac{2}{3}sy^3,\\
Q(x,y,s)=(3y^2+2xy+s)h_{yy}(x,y,s)-(6y+2x)h_y(x,y,s)+6h(x,y,s),
\end{array}
$$
for some germ of a family of analytic functions $h$. 
Suppose that $g^s$ has a cuspidal edge singularity at $(x_1,y_1)$ 
and there exists $(x_2,y_2)\ne (x_1,y_1)$ such that $g(x_1,y_1,s)=g(x_2,y_2,s)$, i.e., $g^s$ has an $A_0A_2$-singularity. 
Writing $g=(g_1,g_2,g_3)$, then $x_1=x_2=x$, and the conditions $g^2_y(x,y_1,s)=0$ and 
$(g_2(x,y_1,s)-g_2(x,y_2,s))/(y_1-y_2)=0$ give 
\begin{equation}\label{eq:xs}
x=-2y_1-y_2, \quad s=y_1^2+2y_1y_2.
\end{equation}
Now, with $x$ and $s$ as in \eqref{eq:xs}, we get 
$$
\frac{P(x,y_1,s)-P(x,y_1,s)}{y_1-y_2}=(y_1-y_2)^3(2y_1+3y_2).
$$
We prove, by considering the homogeneous part of $Q$ of any degree, for $x$ and $s$ as in \eqref{eq:xs}, that
$$
\frac{Q(x,y_1,s)-Q(x,y_1,s)}{y_1-y_2}=(y_1-y_2)^3L(y_1,y_2)
$$
for some analytic function $L$.
Thus, from $(g_3(x,y_1,s)-g_3(x,y_2,s))/(y_1-y_2)=0$, and with $x,s$ as in \eqref{eq:xs}, we get $2y_1+3y_2+O_2(y_1,y_2)=0$. 
It follows that the set of points where $g^s$ has an $A_0A_2$-singularity is a regular analytic curve in the $(x,y,t)$-space parametrised by
$y\mapsto (2y+O_2(y),y,-3/4y^2+O_3(y)).$ (It is worth observing that there is no extra condition on $g$ for the existence and regularity of this curve.)
On the other hand, the swallowtail singularities of $g^s$ occur along the curve in $(x,y,s)$-space parametrised by $y\mapsto (-3y,y,3y^2).$ 
Clearly, the above two curves lie on opposite side of the tangent plane at the origin of the surface of the singular sets $3y^2+2xy+s=0$. 
All the above properties are $\mathcal A$-invariant, so 
we get the configuration in Figure \ref{fig:BigSingSet} (left)
for the family $f^s$ of pseudospherical surfaces.
Using this, and taking into consideration the deformation of $2/5$-cuspidal-edge (Figure \ref{fig:bif25CuspEdge}), we can draw the generic bifurcations 
in the family $f^s$ as shown in Figure \ref{fig:Shcherbak}.
\end{proof}

\begin{remark}
The $2/5$-cuspidal edge and Shcherback bifurcations are the only ones that can occur in generic 1-parameter families of pseudospherical surfaces $f^s$ with $f^0$ not a wave front. This follows from the fact that 
any other bifurcation would mean the singular set of $f^0$ has either a singularity more degenerate than Morse or a Morse singularity with one branch a null curve and the other branch tangent to the other null curve. These cases determine varieties of codimension $\ge 4$ in the jet space, so cannot occur generically in 1-parameter families of pseudospherical surfaces.
\end{remark}

\section{Construction via loop groups}\label{sec:Loop_Groups}
In this section we will show how to construct pseudospherical
surfaces with prescribed singularities using loop groups.

\subsection{Loop group construction of pseudospherical frontals}
 See \cite{singps} for more details of the following outline.
We identify $\real^3 = \mathfrak{su}(2)$ with inner product
$\langle X, Y \rangle = -2\trace(XY)$ and an orthonormal basis
\bdm
e_1=\frac{1}{2}\bbar 0 & i \\ i & 0 \ebar, \quad
e_2=\frac{1}{2}\bbar 0 & -1 \\ 1& 0 \ebar, \quad
e_3=\frac{1}{2}\bbar i & 0 \\ 0 & -i \ebar.
\edm

Let $\mathcal{G}$ denote the group of smooth maps (loops) $\gamma: \SSS^1 \to SL(2,\C)$ that satisfy the
\emph{twisting condition}
$\gamma(\lambda)= \Ad_P \gamma(-\lambda)$, where $P = \hbox{diag}(-1,1)$ and
the \emph{reality condition} $\gamma(\lambda) = \overline{\gamma^t(\bar \lambda)}^{-1}$.
The loops are assumed to be of a suitable class such that each loop extends holomorphically
to an annulus around $\SSS^1 \subset \C^*$. The reality condition means that a loop takes values in $SU(2)$ for
real values of the loop parameter $\lambda$. The twisting condition is standard (see, e.g. \cite{guest1997}) in the loop group
representation of harmonic maps into the symmetric space 
$\SSS^2 = SU(2)/K$ where  the diagonal subgroup $K$ is the fixed point set of
the involution $[x \mapsto \Ad_P x]$ of $SU(2)$.

Let $\Omega \subset \real^{1,1}$ be a simply connected open set.
An \emph{admissible frame} is a map $\hat F: \Omega \to \mathcal{G}$ such that the 
Fourier expansion of the \emph{Maurer-Cartan form} $\hat F^{-1} \dd \hat F$ is a 
Laurent polynomial of the form
\[
\hat F^{-1} \dd \hat F = (A_{1} \dd x) \lambda + (A_0 \dd x + B_0 \dd y) + (B_{-1} \dd y) \lambda^{-1},
\]
where $(x,y)$ is any local null-coordinate system. For any map $\hat Z$ into $\mathcal{G}$,
write $Z = \hat Z \big|_{\lambda =1}$.  Given an admissible frame, define
$N: \Omega \to \SSS^2$  and $f: \Omega \to \real ^3 = \mathfrak{su}(2)$ by:
\[
N = \Ad_F e_3, \quad \quad f = \left( \lambda \frac{\partial \hat F}{\partial \lambda} \hat F ^{-1} \right) \Big|_{\lambda =1}.
\]
Then $N$ is a Lorentzian-harmonic map and $f$ is the associated pseudospherical surface with Gauss map $N$.
The problem of constructing pseudospherical frontals $\Omega \to \real^3$ is equivalent to constructing admissible frames, and an admissible frame $\hat{F}$ is determined uniquely by $f$ (or $N$) if we choose a
basepoint $p$ at which $\hat F(p)=I$.

\subsection {The generalized d'Alembert representation}
\label{gdarsection}
M.~Toda \cite{todaagag} gave a method that allows one to
construct all admissible frames from pairs of arbitrary $\C$-valued functions of one variable. 
If $U=(a,b) \times (c,d)$ is a local \emph{box chart} in $\real^{1,1}$, i.e. the coordinates
$x$ on $(a,b)$ and $y$ on $(c,d)$ are oriented null coordinates, a
\emph{potential pair} $(\hat \chi, \hat \psi)$ is a pair of 1-forms with values in the 
Lie algebra $\hbox{Lie}(\mathcal{G})$ of $\mathcal{G}$ of the form
\[
\hat \chi = \sum_{n=-\infty}^1 A_n(x) \lambda^n \dd x, \quad
\hat \psi = \sum_{n=-1}^\infty B_n(y) \lambda^n \dd y.
\]
Because of the twisting and reality conditions, the ``leading''  terms of $\hat \chi$ and $\hat \psi$ are:
\[
\hat \chi_1 = \bbar 0 & \zeta (x) \\ - \overline{\zeta (x)} & 0 \ebar \lambda \dd x, \quad
\hat \psi _{-1} = \bbar 0 &  \xi (y) \\ - \overline{\xi (y)} & 0 \ebar \lambda^{-1} \dd y,
\]
where $\zeta(x)=\zeta_1(x) + i \zeta_2(x)$ and $\xi(y)=\xi_1(y) + i \xi_2(y)$
are arbitrary functions.  A potential pair is called \emph{regular} at a point $(x,y)$ if
both $\hat \chi _1$ and $\hat \psi_{-1}$ are non-zero at $(x,y)$, and a potential pair is
\emph{normalized} if these are the only terms,
i.e. if $\hat \chi = \hat \chi _1$ and $\hat \psi = \hat \psi_{-1}$.

Define subgroups of $\mathcal{G}$ 
\[
\mathcal{G}^\pm := \{x \in \mathcal{G} ~|~ x = \sum_{n=0}^\infty a_n \lambda^{\pm n} \},
\qquad 
\mathcal{G}^\pm_* := \{ x \in \mathcal{G}^\pm ~|~
a_0 = I \}.
\]
The \emph{Birkhoff decomposition} $\mathcal{G} = \mathcal{G}_*^{\pm} \cdot \mathcal{G}^{\mp}$ (see \cite{PreS,jgp}) gives real analytic
 diffeomorphisms $\mathcal{G} \to \mathcal{G}_*^{\pm} \times \mathcal{G}^{\mp}$. Using this, we construct an admissible frame from a potential pair as follows:
\begin{enumerate}
\item Solve $\hat X_+^{-1} \dd \hat X_+ = \hat \chi$, and
$\hat Y_-^{-1} \dd \hat Y_- = \hat \psi$, with initial conditions $\hat X(x_0)=I$ and $\hat Y(y_0)=I$.
\item At each point $(x,y)$ perform the Birkhoff decomposition 
$\hat X(x)^{-1} \hat Y(y) = \hat H_-(x,y) \hat H_+(x,y)$, with  $\hat H_- (x,y) \in \mathcal{G}_* ^-$,
and set
\[
\hat F(x,y) = \hat X(x) \hat H_-(x,y).
\]
\end{enumerate}
Then $\hat F$ is an admissible frame.
The wave front condition for the associated frontal is equivalent to the non-vanishing of the leading order terms $\hat \chi_1$ and $\hat \psi_{-1}$. More precisely:
\[
N_x(x_0,y)\neq 0 \hbox{ for all } y 
\quad \Leftrightarrow \quad  \hat \chi_1 (x_0) \neq 0, 
\quad \hbox{and } N_y(x,y_0) \neq 0 
 \hbox{ for all } x 
 \quad \Leftrightarrow \quad \hat \psi_{-1} (y_0) \neq 0.
\]

Conversely, given an admissible frame $\hat F$, the
two normalized Birkhoff splittings
\[
\hat F = \hat X_+ G_-  = \hat Y_- G_+, \quad
\hat X_+(x,y) \in \mathcal{G}^+_*, \,\,
\hat Y_-(x,y) \in \mathcal{G}^-_*,
\]
gives a normalized potential pair,
$\hat \chi = \hat \chi_1 = \hat X_+^{-1} \dd \hat X_+$
and
$\hat \psi = \hat \psi_{-1} = \hat Y_-^{-1} \dd \hat Y_-$.
Thus, once a basepoint $p=(x_0,y_0)$ (where $\hat F(p)=\hat X_+(p)=\hat Y_-(p)=I$) is chosen, there is a one-one
correspondence between admissible frames and normalized potential pairs.


\subsection{Solving the Cauchy problem} \label{gcpsection}
In Theorem \ref{theo:kjetLorentzH} we saw that the Cauchy problem for a Lorentzian-harmonic map $N$, with $N$ and $N_x$ prescribed along a non-characteristic curve, 
has a solution, and this can be used to construct the different types of singularities via the $k$-jets.  We now want to show how to construct potential pairs for such solutions. This can be done by modifying the constructions
in \cite{dbms1, singps}, where the Cauchy data was $f$ and $N$ rather than $N$ and $N_x$.

Assuming that $|N_x|>0$, we can choose coordinates such that $|N_x|=1$. An $SU(2)$-frame
$F$ is defined by:
\beq \label{adaptedframe}
N= \Ad_F e_3, \quad N_x = - \Ad_F e_2, \quad N_y = \Ad_F (ae_1 + be_2),
\eeq
where $a:= \langle N_y , N \times N_x \rangle$ and $b:= -\langle N_y, N_x \rangle$.
For this choice, the corresponding admissible frame $\hat F: \Omega \to \mathcal{G}$ has Maurer-Cartan form:
\[
\hat F ^{-1} \dd \hat F = (U_\mfk + U_\mfp \lambda )\dd x
  + (V_\mfk + V_\mfp \lambda^{-1}) \dd y,
  \]
where $\mfk = \hbox{span}\{e_3\}$ and $\mfp = \hbox{span}\{e_1, e_2\}$.
Using $N_x = \Ad_F[F^{-1} F_x, e_3]$
and $N_y = \Ad_F [F^{-1} F_y, e_3]$ we have:
\[
 \quad U_\mfp =  e_1,  \quad V_\mfp = -b e_1 + a e_2.
\]
Using the Lorentzian-harmonic map equations $N \times N_{xy} = 0$  and differentiating the above formulae for $N_x$ and $N_y$
respectively with respect to $y$ and $x$, we obtain:
\[
V_\mfk =0, \quad U_\mfk = c e_3, 
\]
and the integrability of $\hat \alpha=\hat F^{-1} \dd \hat F$,
i.e., $\dd \hat \alpha + \hat \alpha \wedge \hat \alpha=0$ is
equivalent to 
\beq \label{integrability2}
b_x =-ac, \quad a_x =bc, \quad c_y = a.
\eeq

Note: by Lorentzian-harmonicity $N_{xy}$ is parallel to $N$,
which (being $\SSS^2$-valued) is orthogonal to $N_y$. Hence $0=\langle N_{xy},N_y\rangle= 2\frac{\partial}{\partial x} \langle N_y, N_y \rangle$, so $\langle N_y, N_y \rangle = a^2 + b^2$ is constant in $x$, and $a a_x + bb_x = 0$.

To solve the Cauchy problem along the curve $y=x$, use coordinates
$t=(x+y)/2$, $s=(x-y)/2$; then this is the curve $s=0$.   If we are
given $\hat F_0(t) :=\hat F(t,0)$ along this curve, then we can set
$\alpha_0 (t) := \hat F_0 ^{-1} \dd \hat F_0 =   (U_\mfk(t,0) + U_\mfp(t,0) \lambda )\dd t
  + (V_\mfk(t,0) + V_\mfp (t,0) \lambda^{-1}) \dd t, $ and the potential pair
\[
\hat \chi (x) =  \hat \alpha_0(x), \quad \hat \psi(y) = \hat \alpha_0 (y),
\]
will solve the given Cauchy problem (see Section 4.2 of \cite{singps}).

\begin{theorem}\label{gctheorem}
Let $I \subset \real$ be an open interval. Given analytic Cauchy data $N_0: I \to \SSS^2$ and $\mv: I \to \SSS^2$,
with 
\[
 \langle \mv(t) , N_0(t)\rangle =0, \quad \hbox{and} \quad
\langle \mv(t) - N_0^\prime(t) \, , \,  \mv (t) \rangle \not \equiv 0,
\]
 for all $t \in I$, there is a Lorentzian-harmonic map 
$N: I \times I \to \SSS^2$,  unique up to an isometry of $\SSS^2$, and with null coordinates $(x,y)$ on $I \times I$, 
satisfying the initial conditions:
\[
N(t, t)=N_0(t), \quad N_x(t,t) = \mathcal{V}(t).
\]
The solution is produced by the generalized d'Alembert method with potential pair $(\hat \chi, \hat \psi)$
as described above, with 
\[
\hat \alpha_0  = \left(c e_3 +  e_1 \lambda 
  + ( -b e_1 + a e_2) \lambda^{-1}\right) \dd t,
  \]
where
\beqas
c = \varepsilon |\mv^\prime|,  \quad 
  a = \left \langle N_0^\prime\, , \, N_0 \times \mv \right \rangle, \quad
 b=  \left\langle \mv - N_0^\prime \, , \, \mv \right\rangle,
  \eeqas
and $\varepsilon = \hbox{sign}\langle \mv^\prime \times N_0, \mv \rangle$. Further:

\begin{enumerate}
    \item The corresponding pseudospherical surface $f$ is a wavefront at $(t,t)$ if and only if $(a(t),b(t)) \neq (0,0)$.
\item    The solution $N$ (and hence $f$) is regular at the point
 $(x,y)=(t,t)$ if and only if $a(t) \neq 0$.
 \item At a singular point $(t,t)$ on the initial curve, the singular set in $I \times I$ is locally regular if and only if $\left(b(t)c(t), a^\prime(t)\right)~\neq~(0,0)$. 
\end{enumerate}
\end{theorem}

\begin{proof}
It is enough to find formulas for  $a$, $b$ and $c$, described above, in
terms of the initial data $N_0(t)= N(t,t)$ and $\mv(t)=N_x(t,t)$.
Along $(t,t)$,  in the coordinates $(t,s)=(x+y,x-y)/2$
we have $N_y=N_t-N_x=N_0^\prime - \mv $. Substituting into
 $a:= \langle N_y , N \times N_x \rangle$ and $b:= -\langle N_y, N_x \rangle$
 gives the formulae for $a$ and $b$.  To find  $c=a_x/b$ we can use the Lorentzian-harmonic
 map equations to write
 $\langle N_y, N \times N_{xt} \rangle = \langle N_y, N \times N_{xx} \rangle$
 and so we have
 \[
 a_x = \langle N_y, N \times N_{xt}\rangle = \langle N_0^\prime - \mv, N_0 \times \mv^\prime \rangle.
 \]
 
To find $c$, since $\mv^\prime\perp \mv$, we have $N_0 \times \mv ^\prime$ is perpendicular to both
 $N_0 \times \mv$ and $N_0$, hence parallel to $\mv$, i.e.,
 $N_0 \times \mv ^\prime = -\varepsilon |\mv^\prime| \mv$.  Thus,
 \[
 c = \frac{a_x}{b} = \frac{ -\varepsilon |\mv^\prime| \langle N_0^\prime - \mv, \mv \rangle}{
   \left\langle \mv - N_0^\prime \, , \, \mv \right\rangle}
  = \varepsilon |\mv^\prime|.
  \]
  
This formula is well-defined, since we have assumed that 
$\langle \mv(t) - N_0^\prime(t) \, , \,  \mv (t) \rangle \not \equiv 0$.
Since $N_x \neq 0$ by assumption, the wave front condition is that $N_y$ does not vanish,
and this is equivalent to $(a,b) \neq (0,0)$.
Concerning regularity, coordinates have been chosen such that
$N_x \times N_y = a \Ad_F e_3$, so the singular set of $N$ is give by $\{a(x,y)=0\}$.
 We also have, along $(t,t)$, that
$\dd a = (bc \dd x + (a_t -bc) \dd y$, so the curve $\{a(x,y)=0\}$ is regular at $(t,t)$
if and only if $(bc, a_t-bc) \neq (0,0)$, and this is equivalent to $(bc, a_t) \neq (0,0)$.
\end{proof}

\begin{remark}
Real analyticity is not strictly needed in the above theorem: it's enough for the 
functions to be differentiable, provided the formula for $c$ is well-defined. Moreover, if we 
just take $|\mv(t)|=A(t)>0$ instead of constant, then we replace $\hat \alpha_0$ in the theorem 
with 
\beq  \label{potentialformula1}
\hat \alpha_0 = \left(c e_3 + A e_1 \lambda 
  + ( -b e_1 + a e_2) \lambda^{-1}\right) \dd t,
 \eeq
 where 
 \beq  \label{potentialformula2}
  c = \frac{\langle N_0^\prime - \mv \, ,\, N_0 \times \mv^\prime \rangle - \langle N_0^\prime \,,\, 
  N_0 \times \mv \rangle A^\prime / A}{\langle \mv- N_0^\prime, \mv \rangle},    \quad 
  a = \frac{\left \langle N_0^\prime\, , \, N_0 \times \mv \right \rangle }{A}, \quad
 b=  \frac{\left\langle \mv - N_0^\prime \, , \, \mv \right\rangle}{A}.
 \eeq
\end{remark}

\begin{figure}[ht]
\centering
$
\begin{array}{ccccc}
\includegraphics[height=26mm]{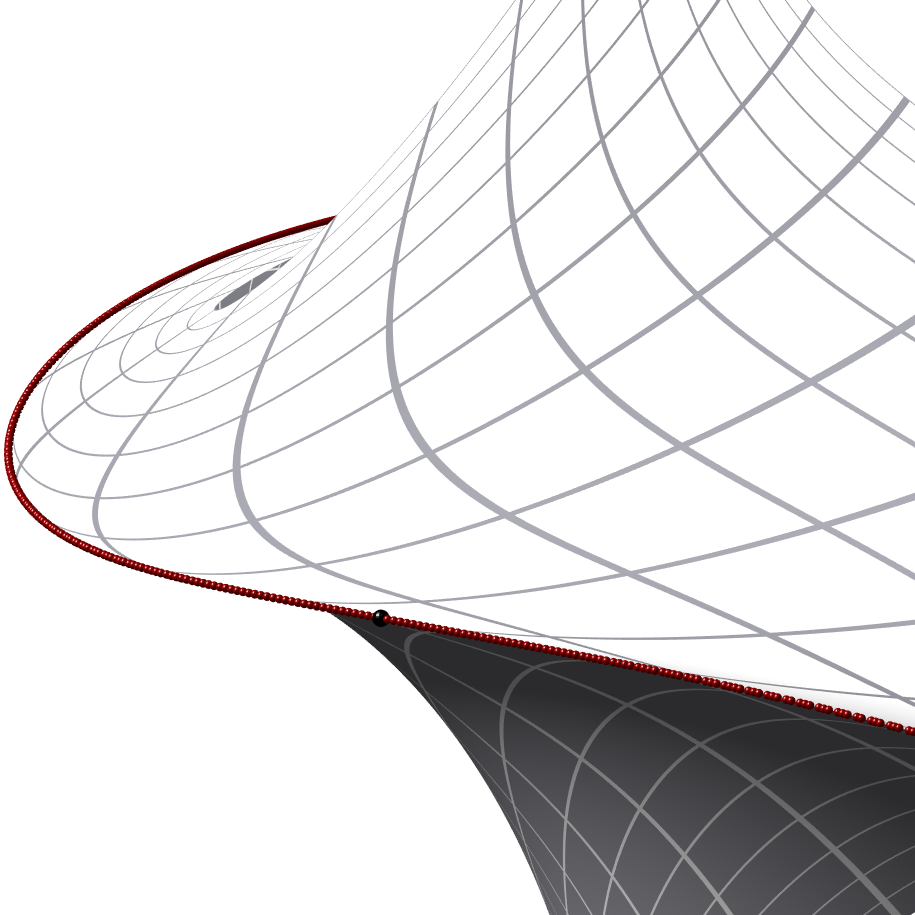}   &
\includegraphics[height=26mm]{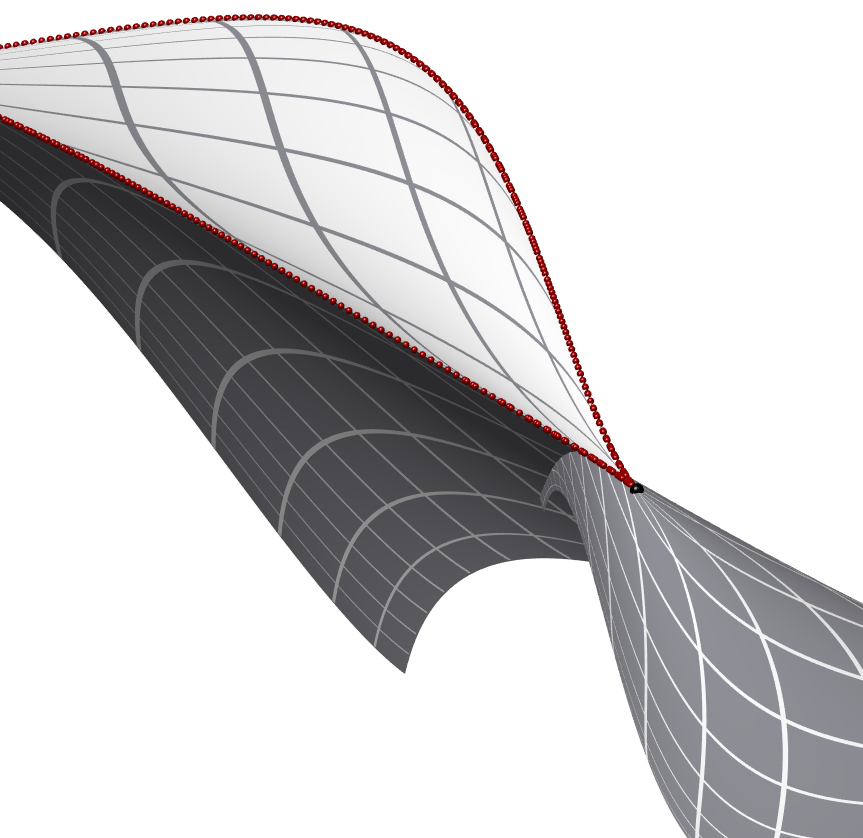}
 & 
\includegraphics[height=26mm]{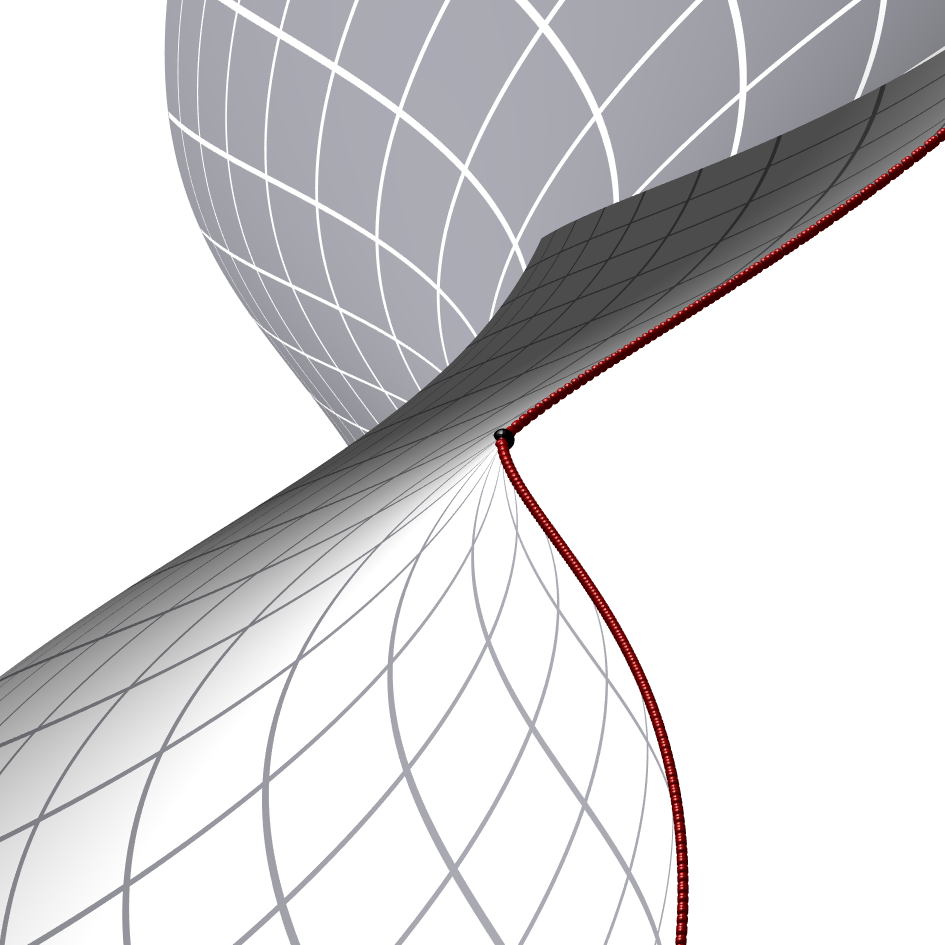} &
 \includegraphics[height=26mm]{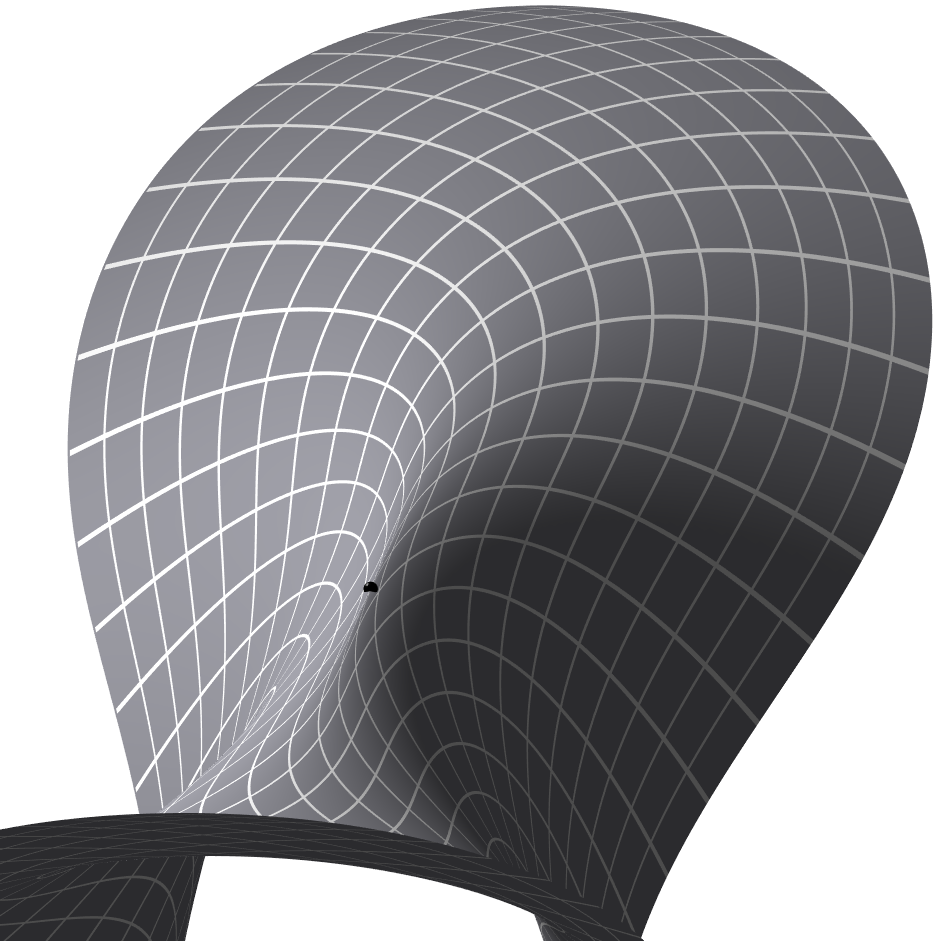}  &
\includegraphics[height=26mm]{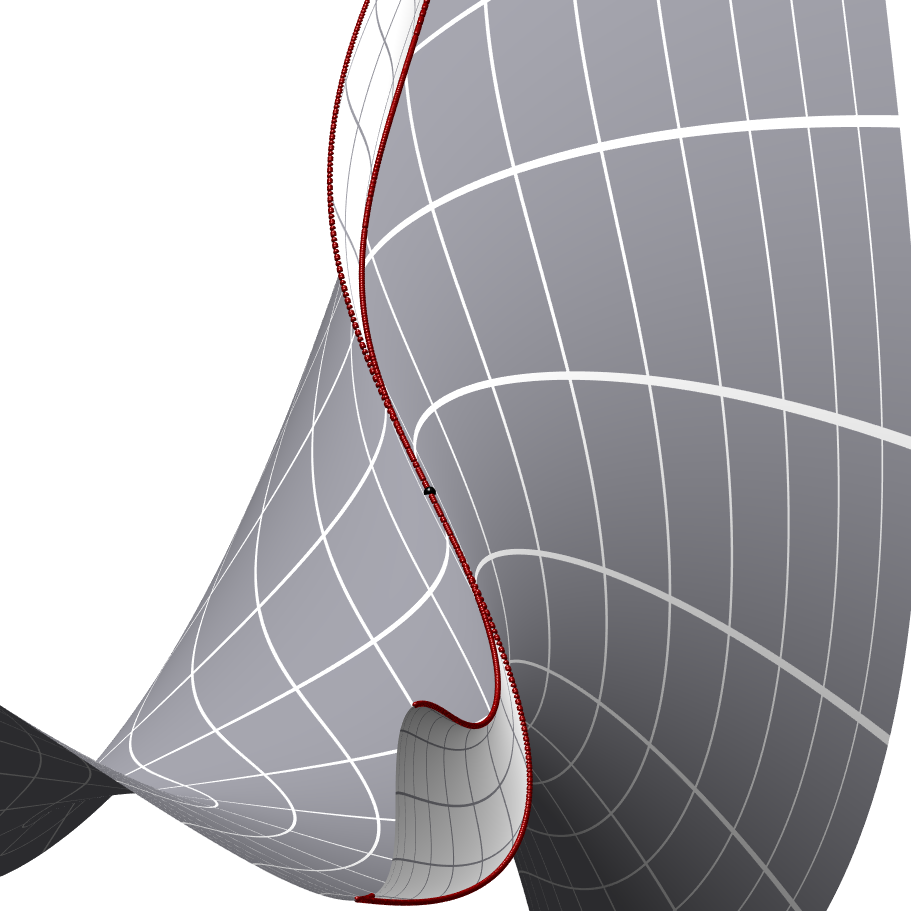} 
\vspace{2ex}\\
\includegraphics[height=22mm]{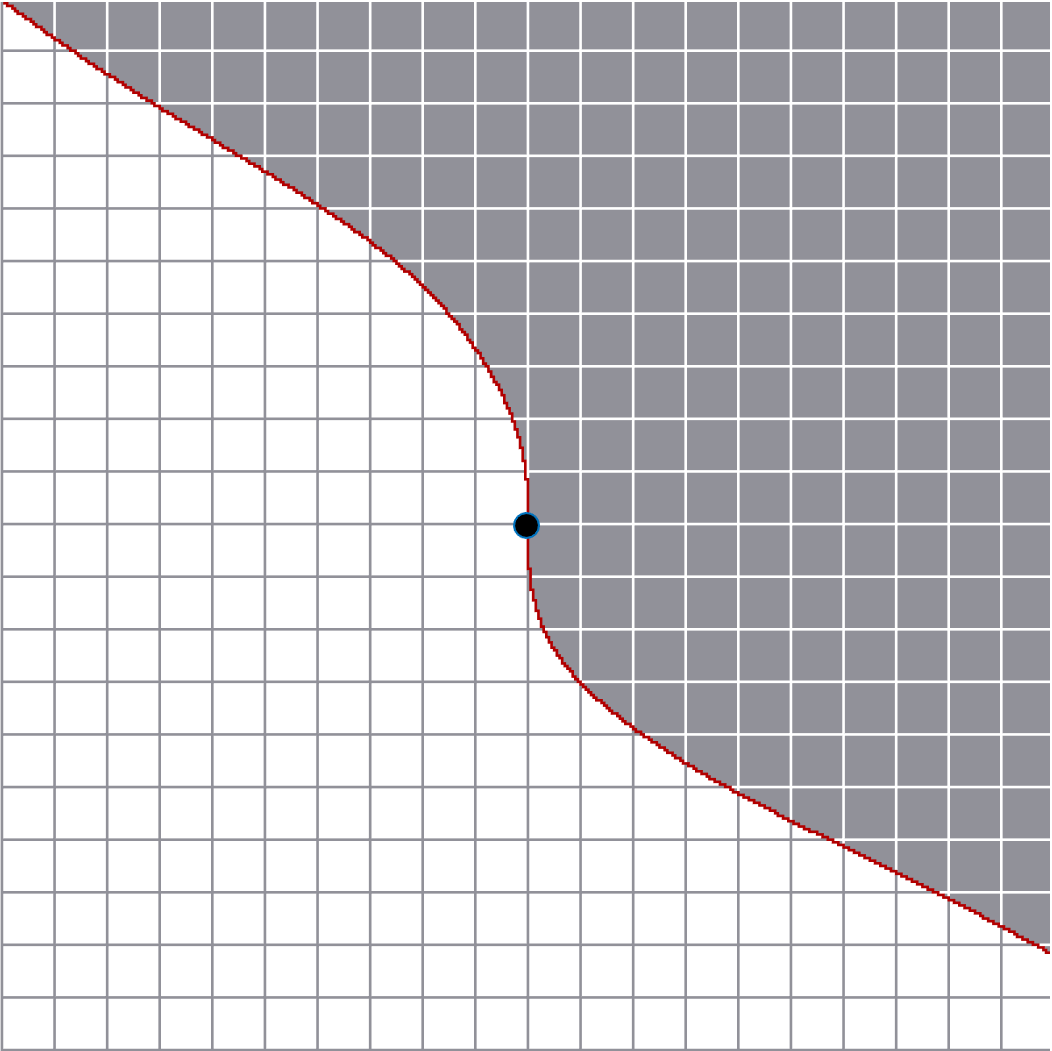}  &
\includegraphics[height=22mm]{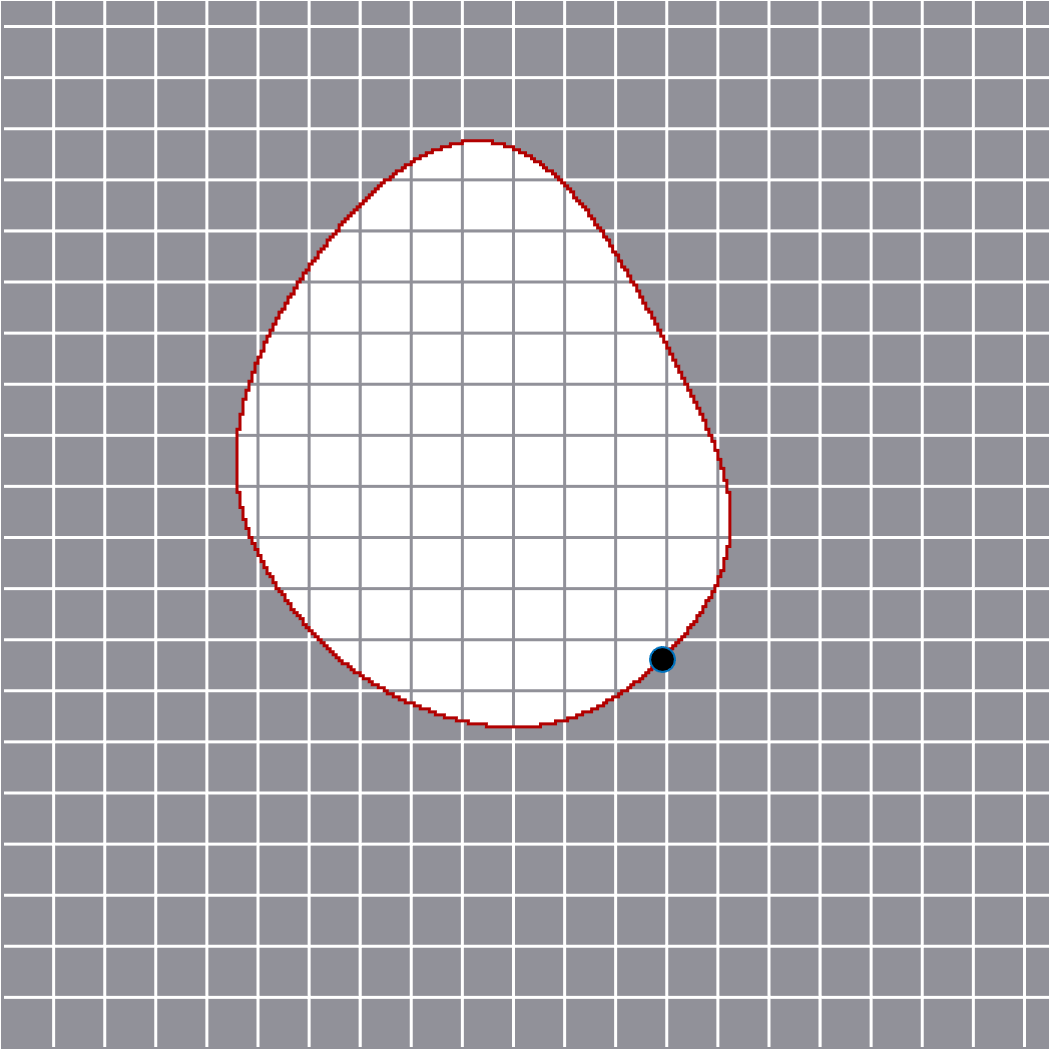} & 
 \includegraphics[height=22mm]{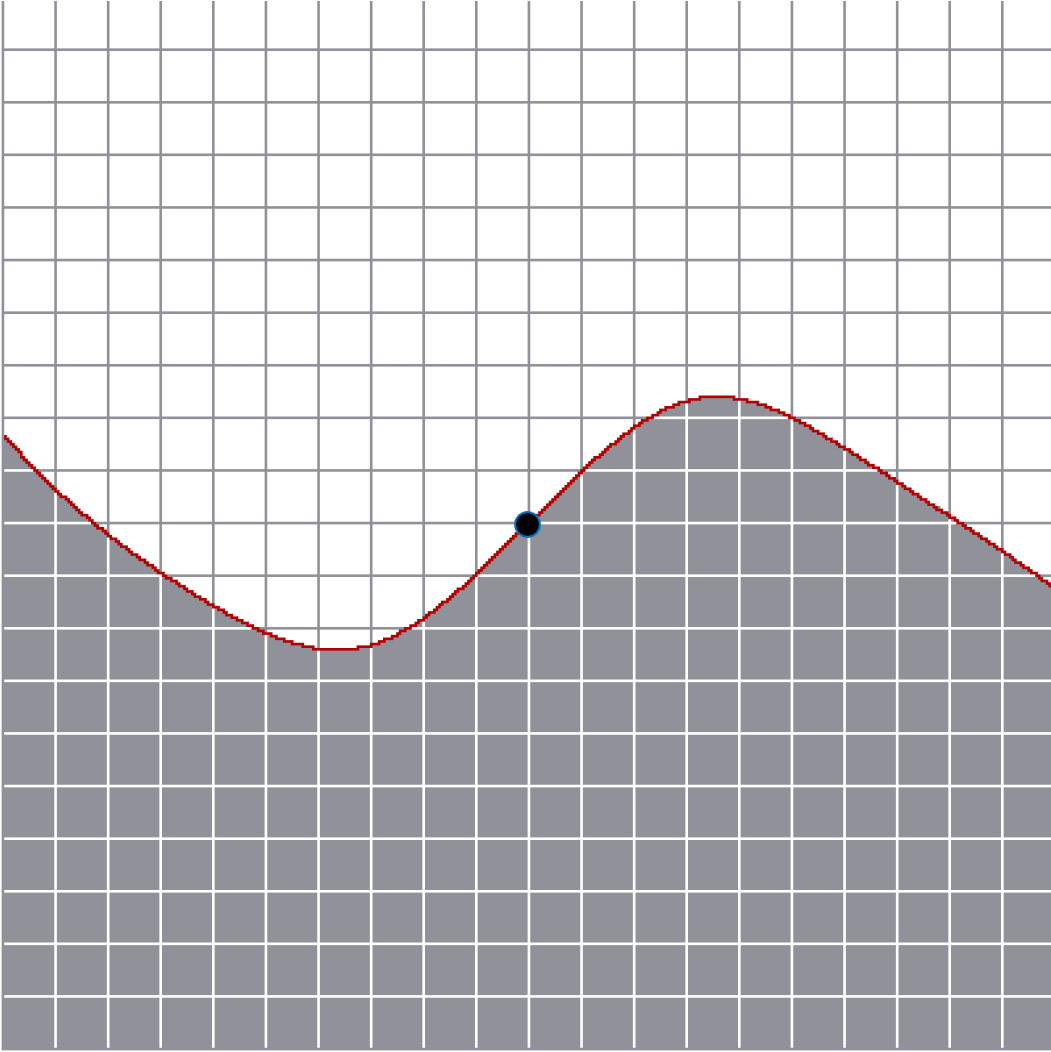}
  & 
  \includegraphics[height=22mm]{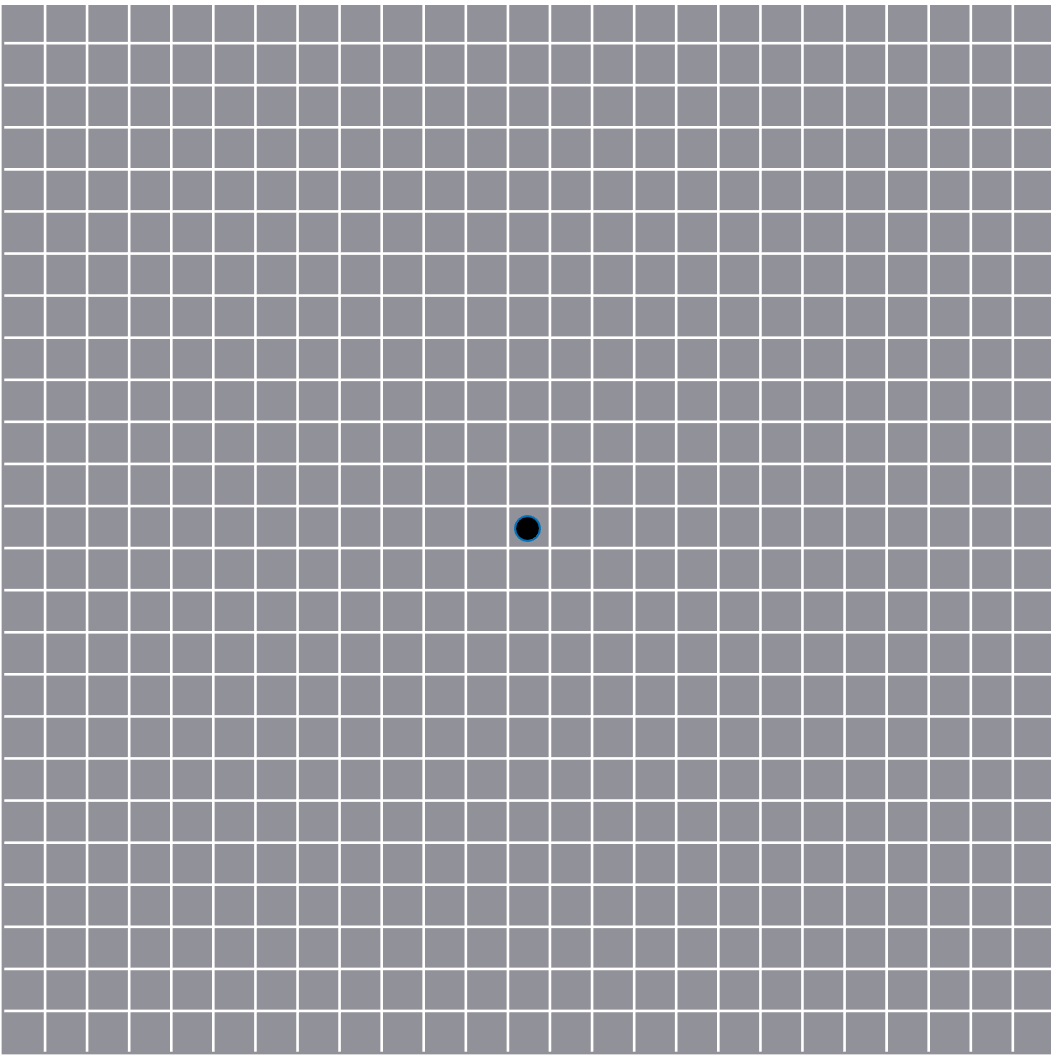}  &
\includegraphics[height=22mm]{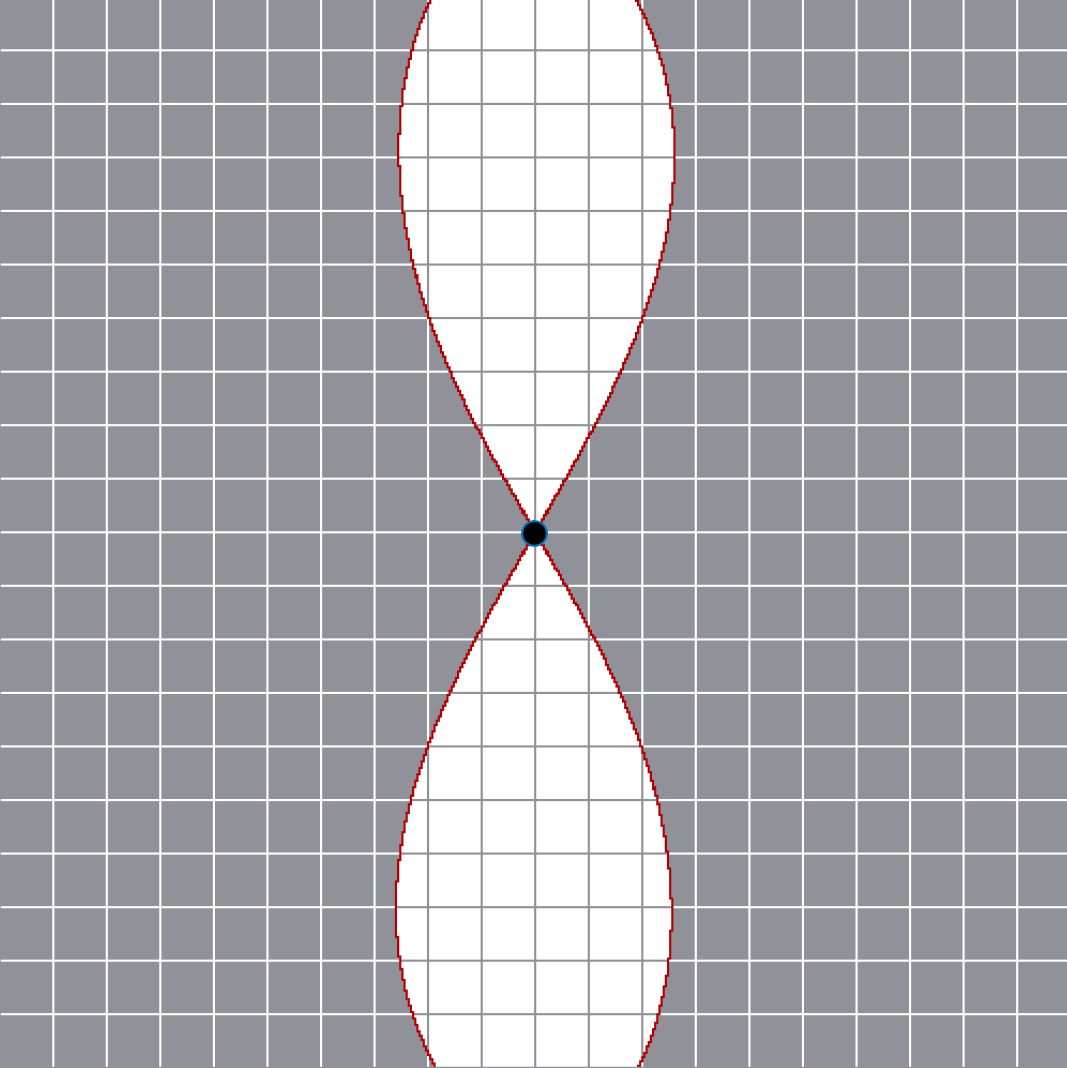}
\end{array}
$
\caption{Pseudospherical surfaces with prescribed singularities.
Top row: cuspidal edge,  swallowtail,  cuspidal butterfly, cuspidal lips, cuspidal beaks.
Bottom row: the corresponding points in the coordinate domains with the same color map (see Example \ref{examples1} for more details).
}
\label{figCSB}
\end{figure}

\begin{example} 
\label{examples1}
We can use a numerical implementation of the generalized d'Alembert method 
(Section \ref{gdarsection}) to compute solutions.
To compute examples of the five types of singularities in Theorem \ref{theo:bif_Wavefront_case},
we choose 
$$j^3({\bf a}_1 ; {\bf a}_2 ; {\bf b}_1;  {\bf b_2}) = (a_{10},a_{20},a_{30}; a_{11},a_{22},a_{33};
 b_{10},b_{20},b_{30}; b_{11},b_{22},b_{33})$$ 
 that satisfy the conditions of the theorem, substitute
 into $N(x,y)$, find $N_x$ then use $N(t,t), N_x(t,t)$ in \eqref{potentialformula2}. 
 We have computed the examples in Table \ref{extable},
and the results are displayed in Figure~\ref{figCSB}.
The $xy$-rectangle that was used to compute the solution is plotted 
beneath the corresponding solution, with exactly the same colormap.
The regions where the continuous normal $N$ is parallel to $f_x \times f_y$ and to $-f_x \times f_y$ are colored with opposite colormaps, so the faces change color when a cuspidal edge is crossed. The point $(x,y)=(0,0)$ and its image under $f$ (in this case the point where the singularity of interest occurs) are  indicated by black dots. The singular set is found numerically by finding the points on the mesh where the normal direction is ambiguous, and is also highlighted in red.

 \begin{table}[h!] 
 \begin{center}
 \begin{tabular}{ll}
 $({\bf a}_1 ; {\bf a}_2 ; {\bf b}_1;  {\bf b_2})$ & Singularity Type \\
 \hline
 $(1,1,0; 1,0,0; 1,0,0; 1,0,0)$ &  cuspidal edge \\
$(1,1,0; 1,2,0; 1,0,0; 1,1,0)$ &  swallowtail \\
$ (1,1,1; 1,1,0; 1,0,1; 1,0,0)$ & cuspidal butterfly \\
$(1,0,1; 1,0,0; 1,0,0; 1,0,1)$ &  cuspidal lips \\
 $(1,0,1; 3,0,0; 1,0,0; 3,0,-1)$ &  cuspidal beaks \vspace{1ex} \\
\end{tabular}
\caption{Example $3$-jets and the corresponding singularities.} \label{extable}
\end{center}
\end{table}
\end{example}

The converse to Theorem \ref{gctheorem} is also valid: any triple of functions $(a,b,c)$ on an interval $I$ will generate a Lorentzian-harmonic map $N$ and its corresponding pseudospherical surface $f$ with domain $I \times I$.  In the next subsections 
we will examine directly the conditions on $(a,b,c)$ to obtain
the various types of singularities.


\subsection{Prescribed singularities: the wave front case}
Consider now a point where $f$ has a wave front singularity at $p=(x_0,y_0)$.  Then, for the frame
\eqref{adaptedframe} considered above, we have $a(p)=0$ and $b(p)\neq 0$. After a change of
coordinates $(\tilde x, \tilde y) = (x, \int_{y_0}^y b(\zeta, \zeta) \dd \zeta)$ we can assume 
that $b(t,t)=-1$, which enables us to prove:
\begin{theorem}  \label{wavefronttheorem}
Any pseudospherical wave front can locally be
represented by a potential pair $(\hat \chi, \hat \psi)=(\hat \alpha_0(x), \hat \alpha_0(y))$,
where
\[
\hat \alpha_0(t)  = \left(c(t) e_3 +  e_1 \lambda 
  + (e_1 + a(t) e_2) \lambda^{-1}\right) \dd t.
  \]
The corresponding pseudospherical surface $f$ has a singularity at $p=(t_0,t_0)$
if and only if $a(t_0) = 0$. At such a point, the germ of $f$ at $p$ is a:
\begin{enumerate}
    \item  Cuspidal edge $(A_2)$ if and only if $a^\prime \neq 0$.
    
\item  Swallowtail $(A_3)$ if and only if $a^\prime = 0$,
$a^{\prime \prime} \neq 0$ and $c \neq 0$.

\item Cuspidal butterfly $(A_4)$ if and only if
$a^\prime(t_0)=a^{\prime \prime} = 0$,
$a^{\prime \prime \prime} \neq 0$ and $c \neq 0$.

\item Cuspidal lips $(A_3^-)$ if and only if 
$a^\prime=c = 0$, and $c^\prime(a^{\prime \prime}+c^\prime) <0$.

\item Cuspidal beaks $(A_3^-)$ if and only if
$a^\prime=c = 0$,  $a^{\prime \prime}\neq 0$,
and $c^\prime(a^{\prime \prime}+c^\prime) >0$.
\end{enumerate}
\end{theorem}

\begin{proof}
We are in the situation of Theorem \ref{gctheorem}, with $b(t,t)=b(t)=-1$, so the the singular set is
given by $\{a=0\}$ and the singular set is locally regular 
if and only if $(c(t),a^\prime(t)) \neq (0,0)$. We now use the criteria of
Theorem \ref{theo:critiria} for the various singularities. In the null coordinates
$(x,y)=(t+s, t-s)$, we have, at the point $t_0$ where $a=0$ and $b=-1$, the 
expression $\dd f = \Ad_F(e_1)(\dd x -\dd y) = \Ad_F(e_1) \dd s$. Hence, the null direction 
$\hbox{Ker} \, (\dd f)$ is
given at $p$ by:
\[
\eta = \frac{\partial}{\partial x} + \frac{\partial}{\partial y} = \frac{\partial}{\partial t}.
\]
Thus the conditions (i)-(iii) in Theorem \ref{theo:critiria} are equivalent to items (1)-(3) of this theorem. \\
To prove items (4) and (5), note that the function $\sigma=\det(f_x, x_y, N)$ in Theorem  \ref{theo:critiria} is here given by $\sigma(x,y)=a(x,y)$, where we write $a(t,t)=a(t)$.
To find the condition for $a(x,y)$ to have a Morse singularity, we have, as before, 
from \eqref{integrability2}:
\[
a_x=bc, \quad b_x = -ac, \quad a_y = a_t -a_x = a_t -bc,
\]
so $\dd a = bc \dd x +(a_t-bc) \dd y$, and, along $(t,t)$, we have 
$\dd a = c \dd x + (a^\prime +c) \dd y$, i.e., $a(x,y)$ has rank zero at $p$ if and only if
\[
a^\prime(t_0)=c(t_0)=0.
\]
We also have from \eqref{integrability2} that $c_y=a=0$ at $t_0$,
so $c_x(t_0)=c^\prime(t_0)$, and we find that 
\[
a_{xx}=-c^\prime, \quad a_{xy} =0, \quad a_{yy}=a^{\prime \prime} + c^\prime,
\]
at $p$. This gives the determinant of the Hessian matrix of $a$ at $p$
\[
a_{xx}a_{yy}-a_{xy}^2= -c^\prime(a^{\prime \prime}+c^\prime),
\]
and so the conditions (iv) and (v) of Theorem \ref{theo:critiria}
are equivalent to items (4) and (5).
\end{proof}

\begin{figure}[ht]
\centering
$
\begin{array}{cccc}
\includegraphics[height=32mm]{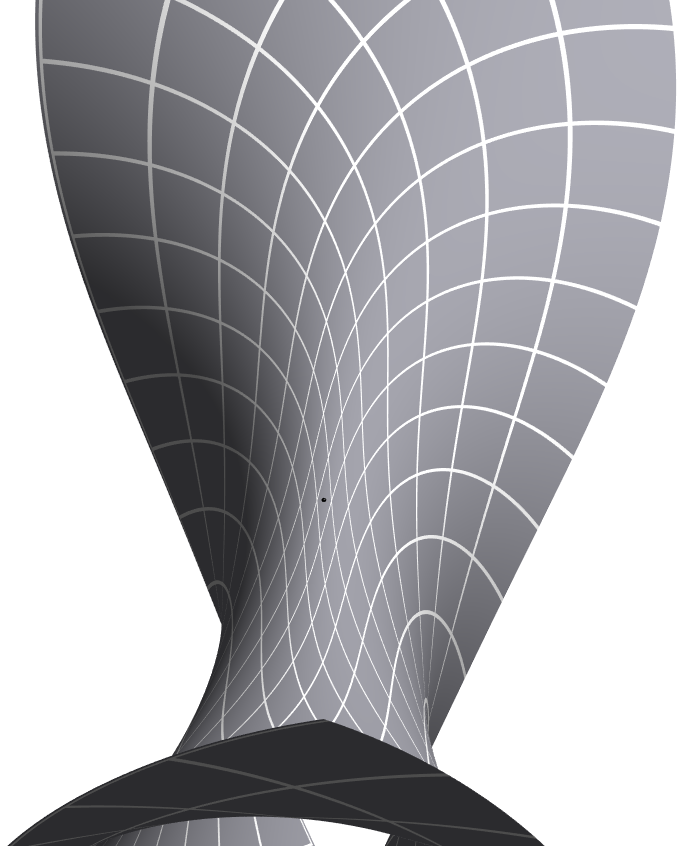}  \quad &  
\includegraphics[height=32mm]{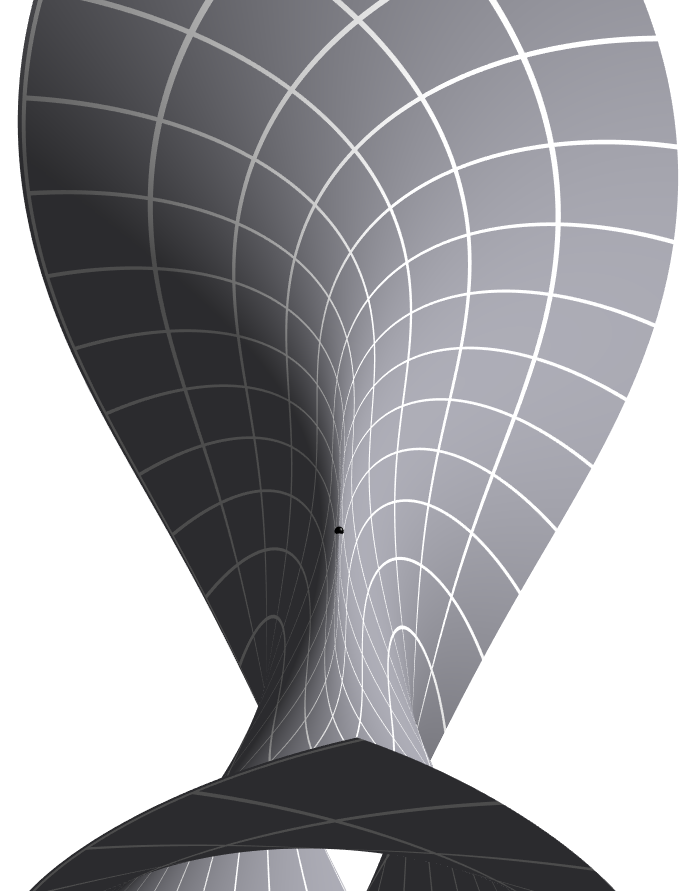} \quad
  & 
\includegraphics[height=32mm]{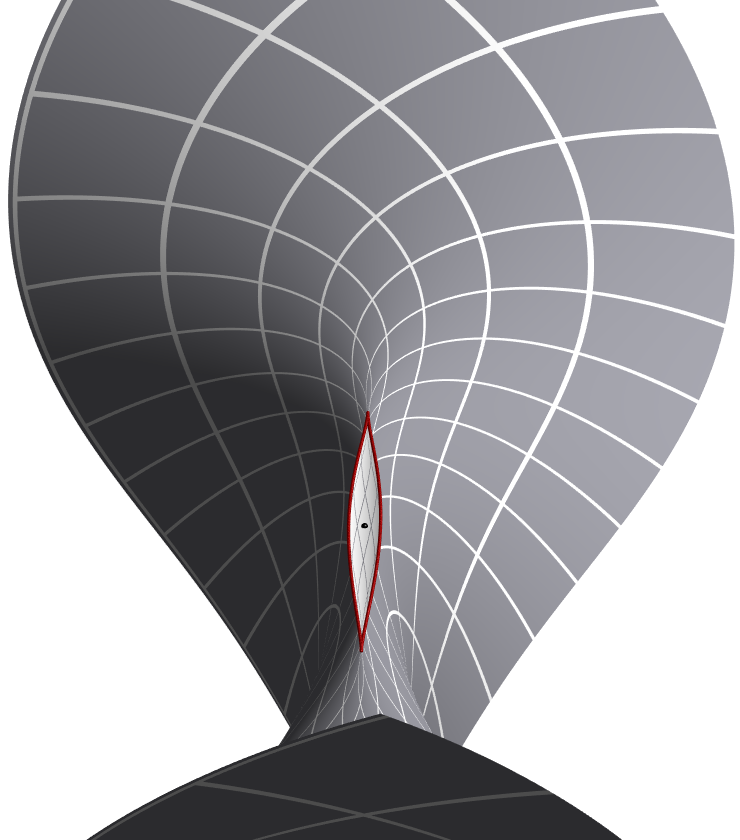}
\quad   & 
\includegraphics[height=32mm]{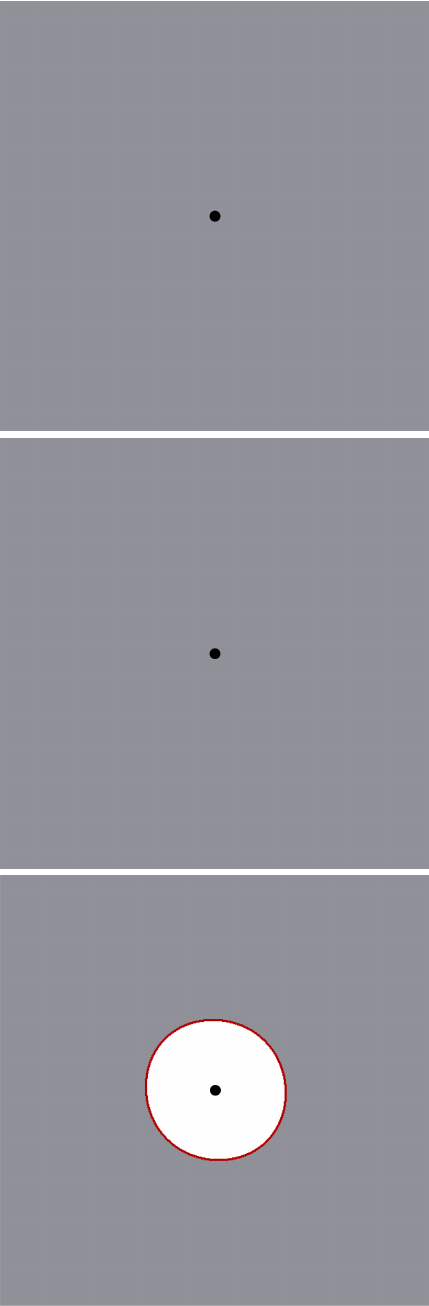} \vspace{1ex}\\
\includegraphics[height=32mm]{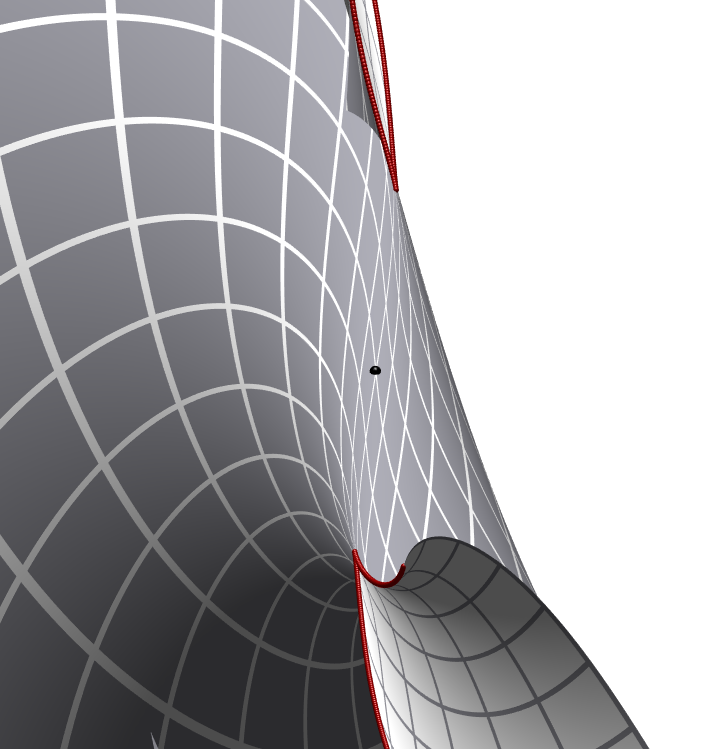} \,   & \,
\includegraphics[height=32mm]{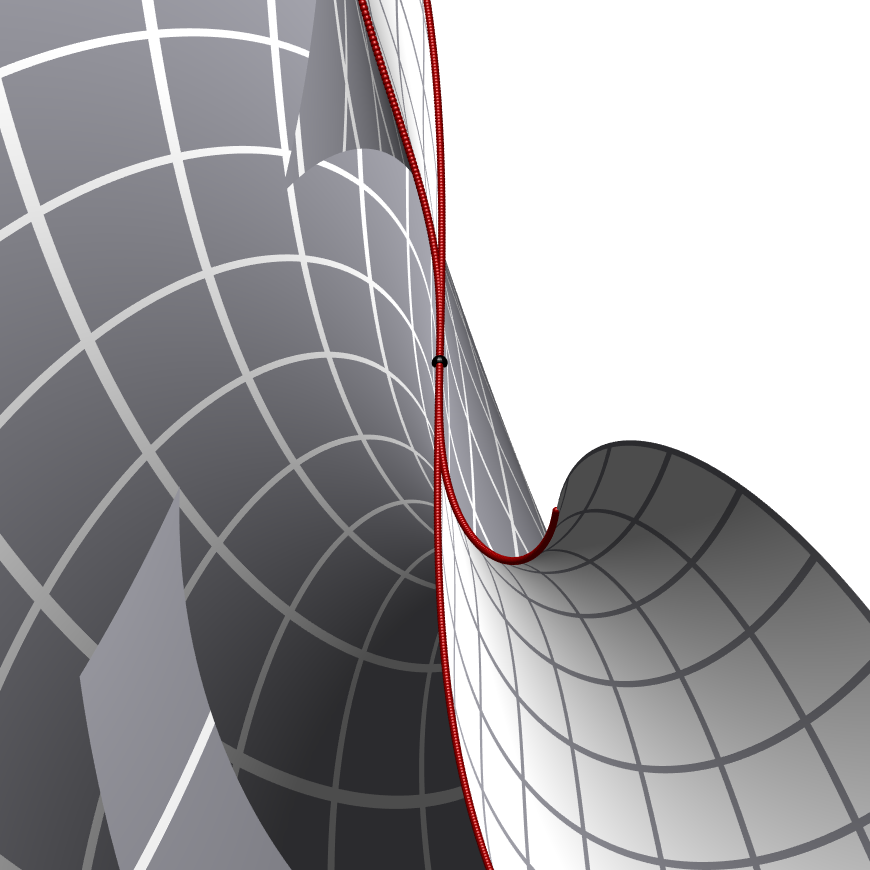}
\, & \,
\includegraphics[height=32mm]{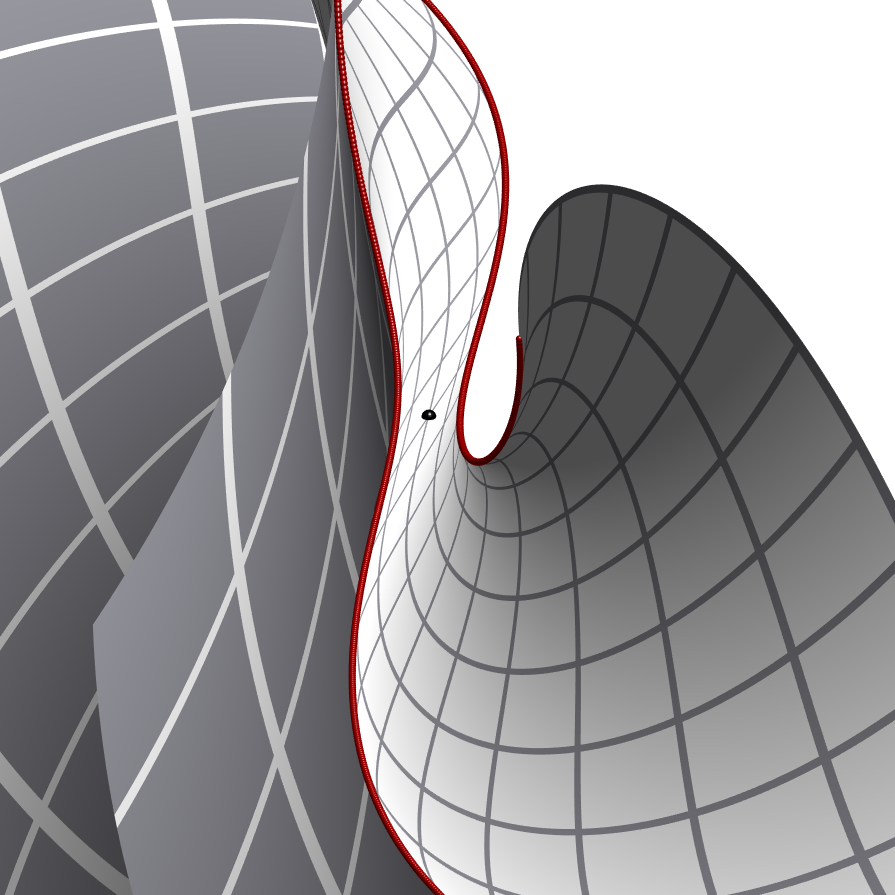}
\, & \,
\includegraphics[height=32mm]{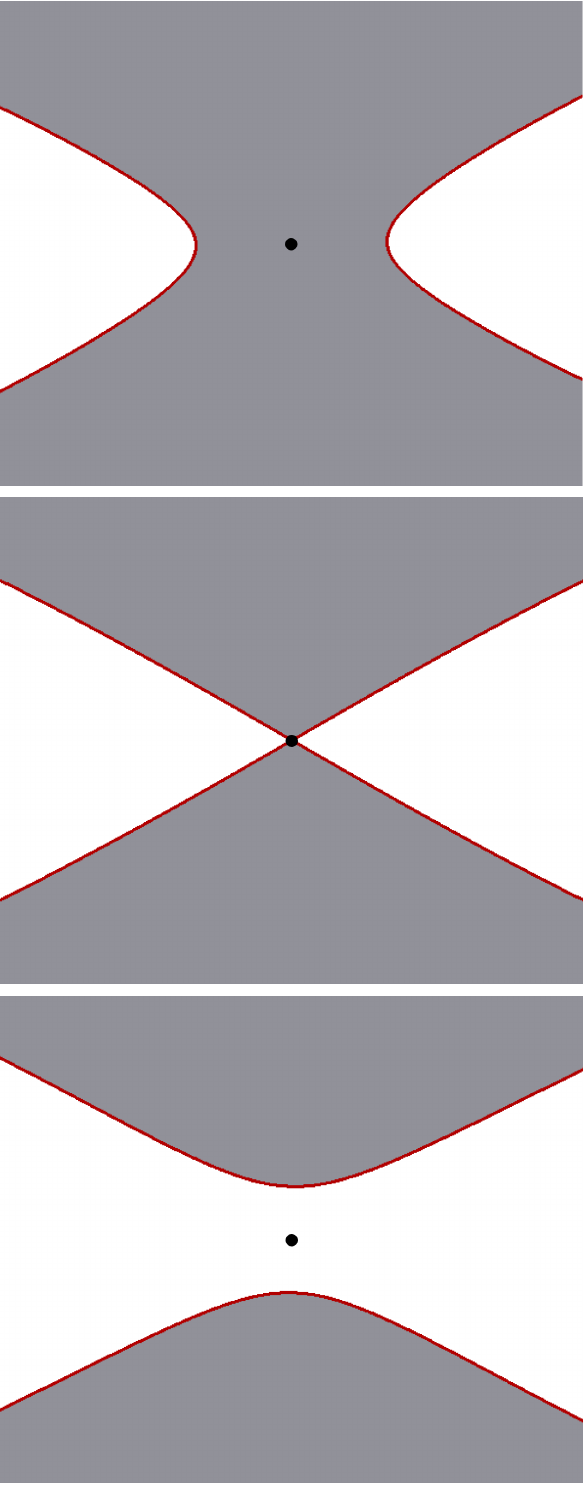} \vspace{1ex}\\
\includegraphics[height=32mm]{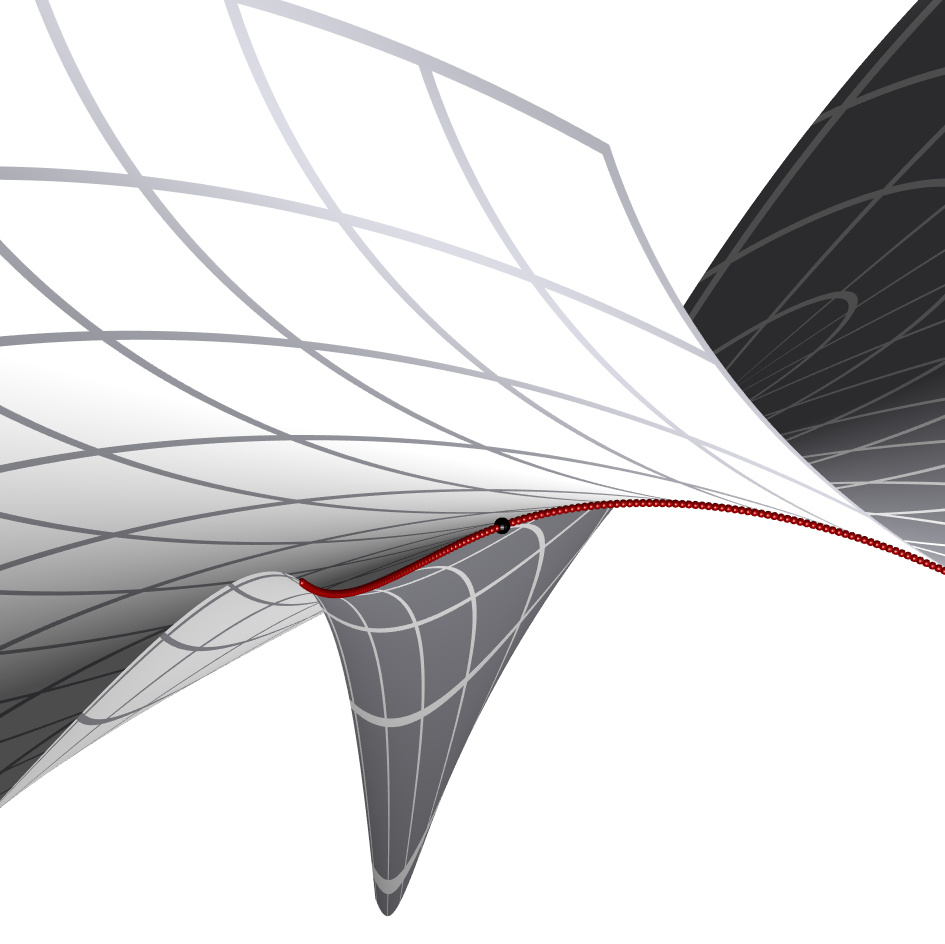}  \, & \, 
\includegraphics[height=32mm]{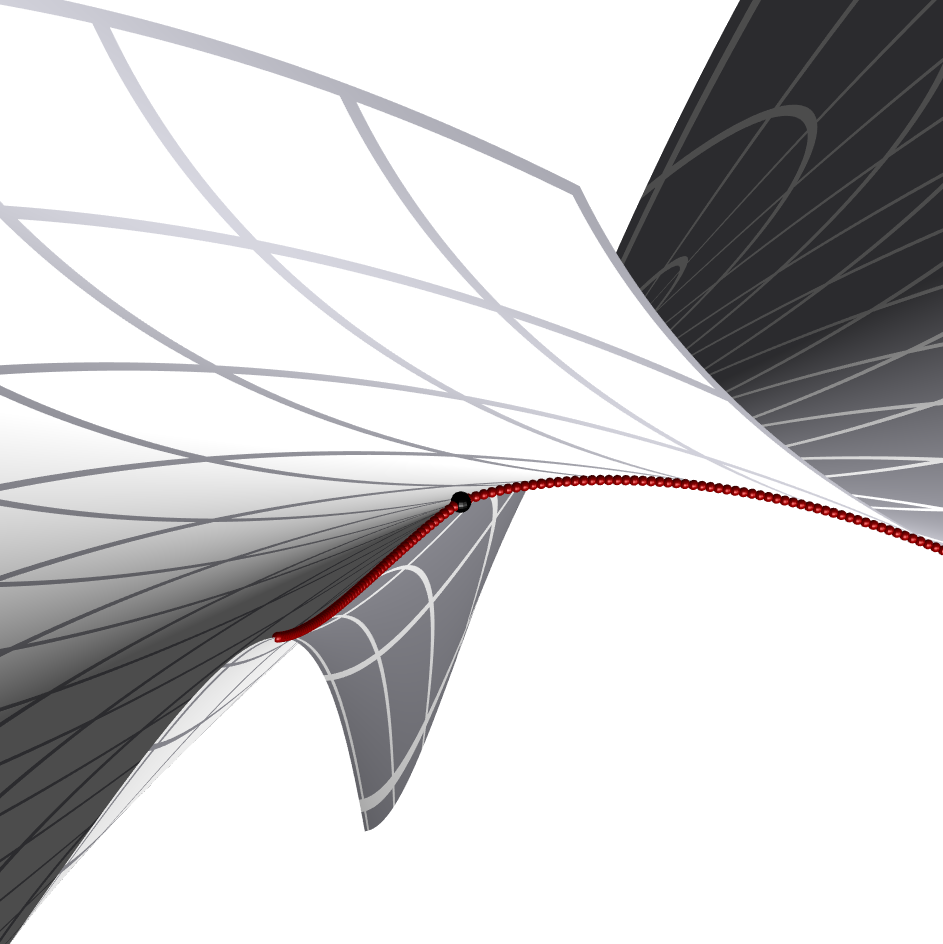} 
 \, & \,
\includegraphics[height=32mm]{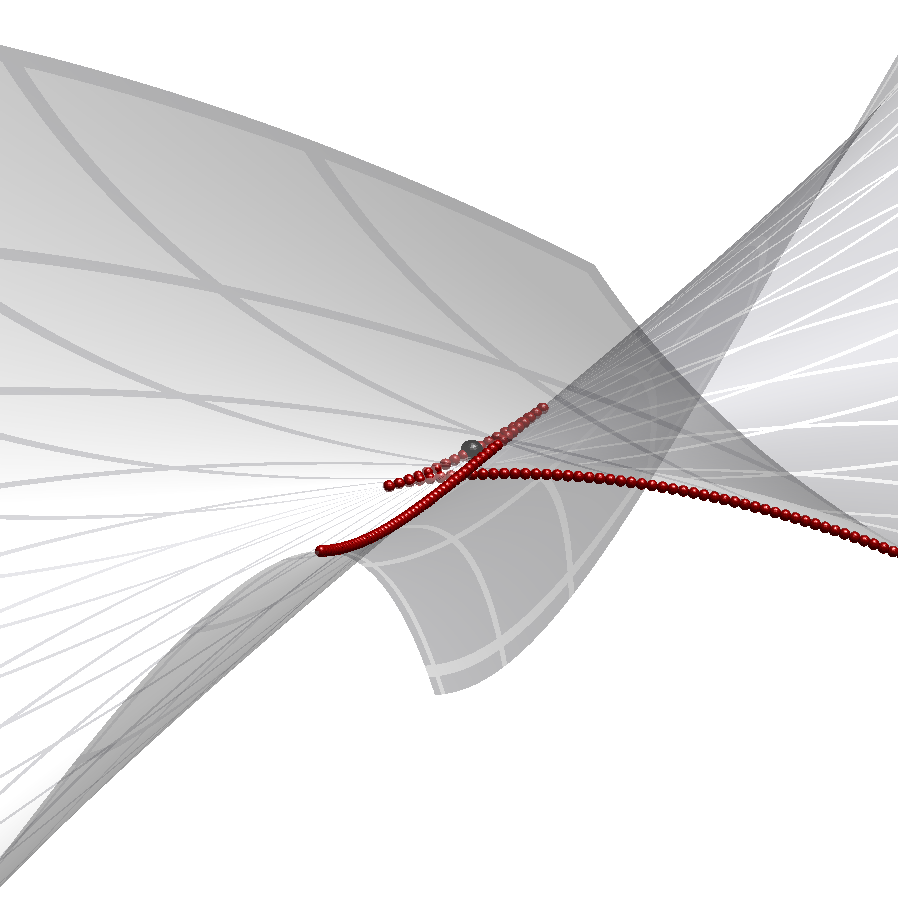}
  \, & \,
\includegraphics[height=32mm]{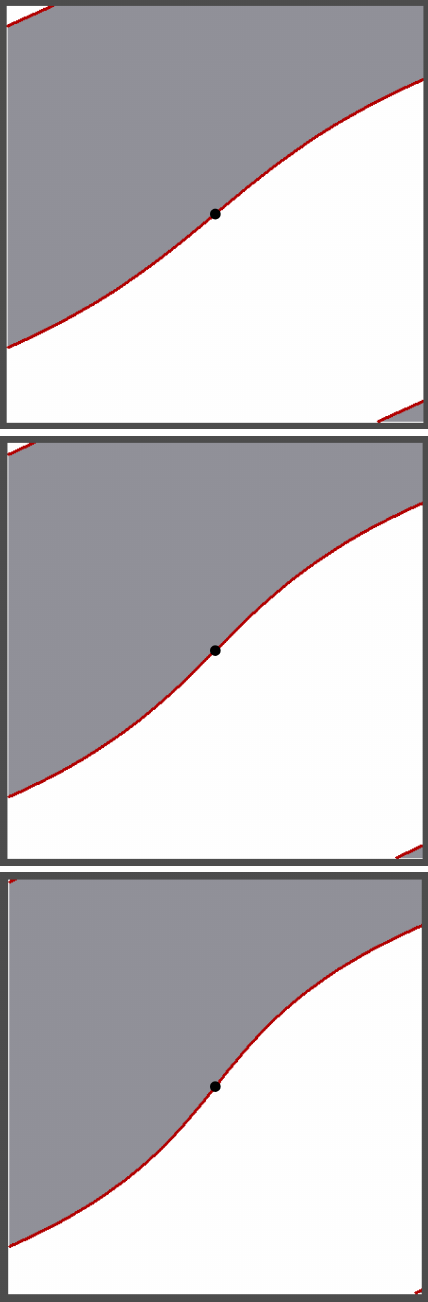} 

\end{array}
$
\caption{Cuspidal lips (top), cuspidal beaks (middle) and  Cuspidal butterfly 
 (bottom) bifurcations. (Example \ref{lipsexample}). 
For each row, the corresponding coordinate patches are shown at the far right (top to bottom). Coloring as in Example \ref{examples1}. The black dot marks the point $(x,y)=(0,0)$ and its image under $f$.}
\label{figbutterflybif}
\end{figure}

\begin{example} \label{lipsexample}
The cuspidal lips, beaks and butterfly bifurcations can be realized in generic 1-parameter families of surfaces produced as follows.
The cuspidal lips occurs in the family $f^r$ of surfaces generated by Theorem \ref{wavefronttheorem} with 
data $(a^r(t),c^r(t))=(t^2+r, -t)$.  The surfaces for $r=0.1$, 
$r=0$ and $r=-0.06$ are computed and shown in Figure \ref{figbutterflybif}. 
The cuspidal beaks bifurcation in the same figure is produced using $(a^r(t),c^r(t))=(t^2+r, t)$ instead.  The cuspidal butterfly bifurcation in the figure was produced by the 
family $(a^r(t),c^r(t)=(t^3+rt, 1)$, computed at $r=0.2$, $r=0$, and $r=-0.2$.
\end{example}

\subsection{Prescribed non-wave front singularities}
All rank 1 non-wave front singularities are locally produced as follows:

\begin{theorem} \label{nonwavefronttheorem}
Let $(f,N)$ be the maps in Theorem \ref{gctheorem}, generated by the functions $a$, $b$ and $c$ from $I \to \real$. Then $f$ is not a wave front at $p=(t_0,t_0)$ if and only if \[
a(t_0)=b(t_0)=0,
\]
At such a point, the singular set  in 
$I \times I$ is locally regular if and only if $a^\prime(t_0) \neq 0$. Moreover, the singularity at $p$ is of type:

\begin{enumerate}
    \item $2/5$-cuspidal edge if and only if $a^\prime \neq 0$ and
    $a^\prime b^{\prime \prime}-b^\prime a^{\prime \prime} + 2 c ((a^{\prime})^2 + (b^{\prime})^2) \neq 0$ at $t_0$.
    \item The singular set has a Morse $A_1^{-}$-singularity at $p$  if and only if
    $a^\prime =0$ and $b^\prime c \neq 0$ at $t_0$.
    In this case, the singularity is a Shcherbak singularity if and only if, additionally, $a''-2b'c \neq 0$.
\end{enumerate}

\end{theorem}

\begin{proof}
The non-wave front conditions and the condition for the singular set to be regular are already explained in Theorem \ref{gctheorem}. For item (1), the conditions for a $2/5$-cuspidal edge, from Theorem \ref{theo:2/5cuspidaledge}, are that the singular set
is regular (here $a^\prime(t_0) \neq 0$), and $[N(t_0),N_{yy}(t_0),N_{yyy}(t_0)] \neq 0$.
Differentiating $N_y = \Ad_F(ae_1+be_2)$ with respect to $y$, and using that
$F^{-1}F_y = (-be_1+ae_2)$ and $a_x=bc$, $b_x=ac$ and $c_y=a$, we find  at $t_0$ that
$N_{yy} = \Ad_F(a_y e_1 + b_y e_2)$ and $N_{yyy}=\Ad_F(a_{yy} e_1 + b_{yy}e_2)$, and hence $N_{yy} \times N_{yyy}$ has a component parallel to $N$ if and only if $a_y b_{yy}-b_ya_{yy} \neq 0$. In the coordinates $(x,y)=(t+s,t-s)$ this translates to the second condition given in item (1).\\

For item (2), note that the singular set is given by $a(x,y)=0$, so the condition is on the Hessian of $a$ at $p$. Given $a(t_0)=a^\prime(t_0)=b(t_0)=0$, we compute that, at $t_0$, the Hessian matrix is:
\beq \label{hessianmatrix}
\textup{Hess}(a)= \left( \begin{array}{cc} a_{xx} & a_{xy} \\ a_{xy} & a_{yy} \end{array} \right) =
\left( \begin{array}{cc} 0 & b^\prime c \\ b^\prime c & a^{\prime \prime}- 2 b^\prime c \end{array} \right),
\eeq
and the determinant is negative if and only if $b^\prime(t_0) c(t_0) \neq 0$. Finally, the two branches of the set 
$a(x,y)=0$ are tangent respectively to the $x$ and $y$-axes (i.e. both are tangent to null curves) at $p$ if and only if 
the last component $a''-2b'c$ of the Hessian matrix is zero.
Hence, the conditions in Theorem \ref{theo:2/5CuspidalBeaks}
for the Shcherbak singularity are that this term does not vanish.
\end{proof}

\begin{figure}[ht]
\centering
$
\begin{array}{ccc}
\includegraphics[height=25mm]{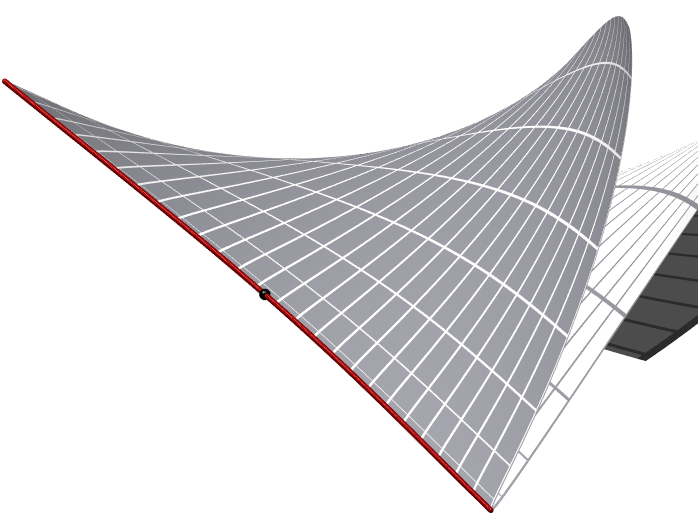}  
 \quad & \quad
\includegraphics[height=25mm]{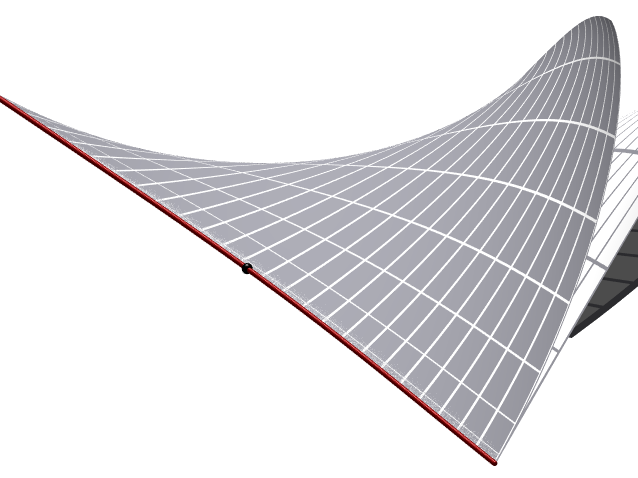}
\quad & \quad  
\includegraphics[height=25mm]{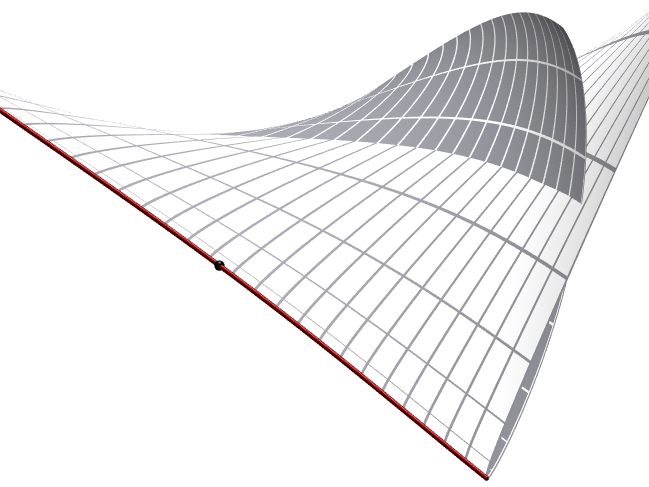}   
 \vspace{1ex}\\
\includegraphics[height=16mm]{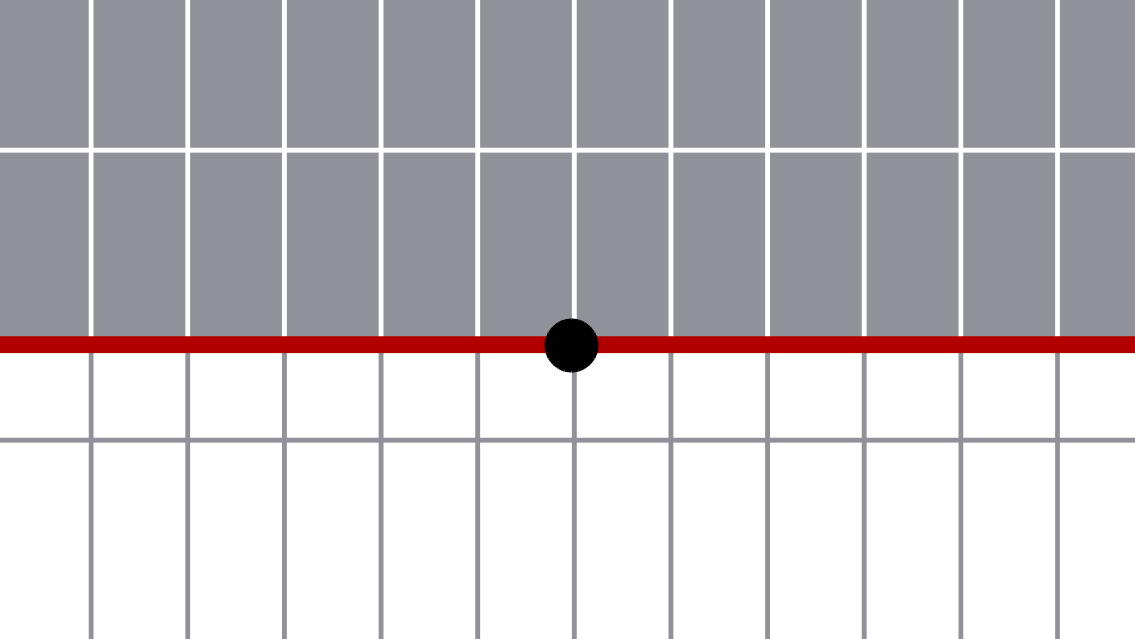} &
\includegraphics[height=16mm]{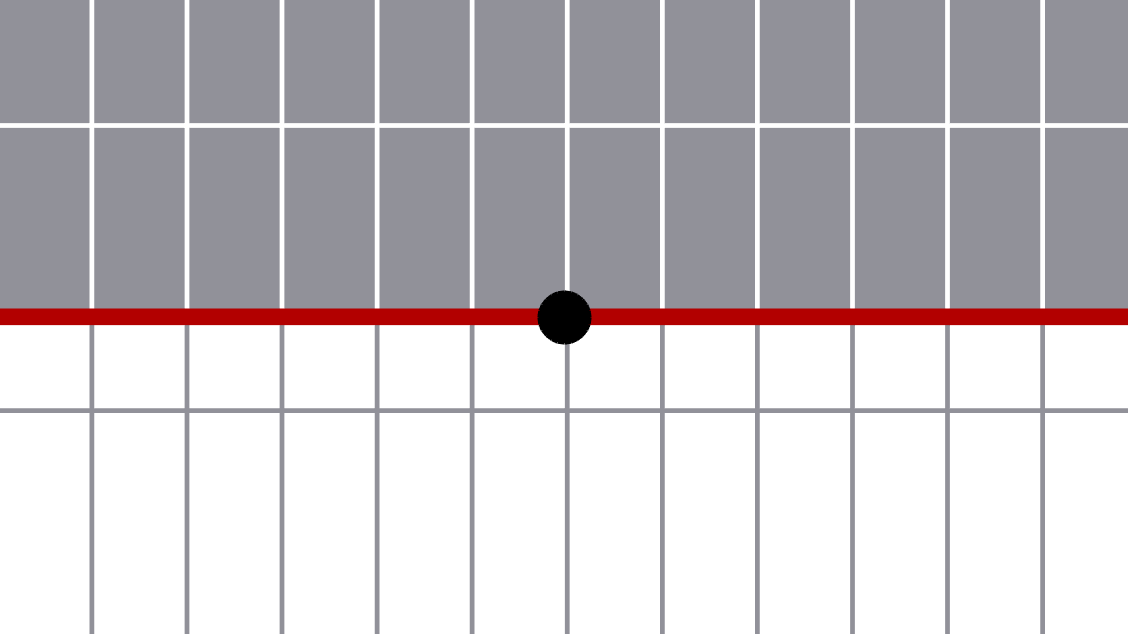}  &
\includegraphics[height=16mm]{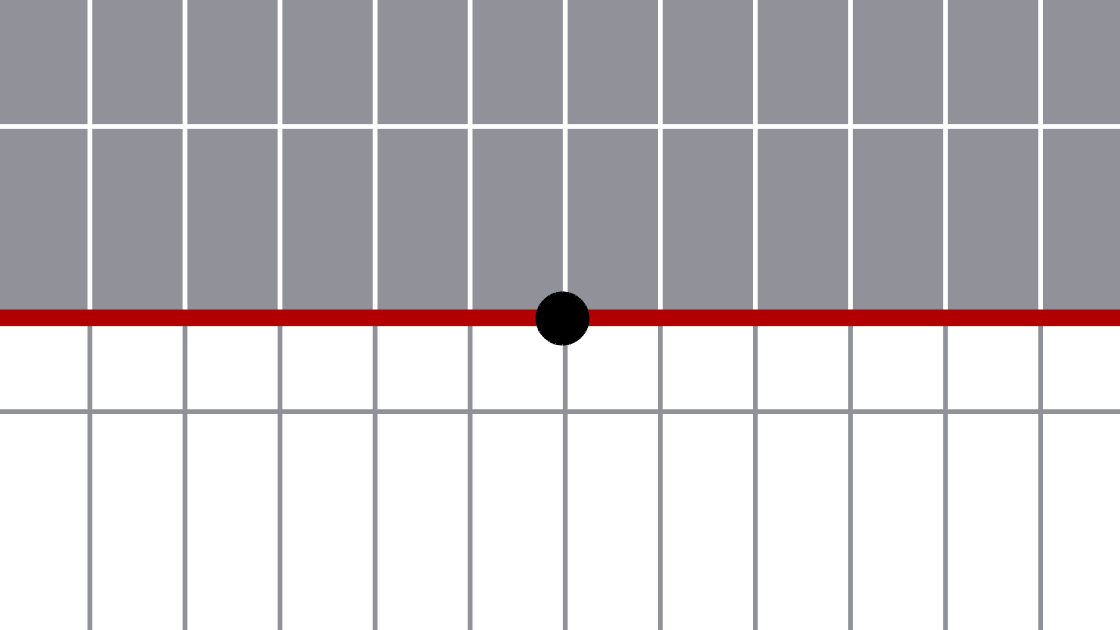} 
\end{array}
$
\caption{$2/5$-cuspidal edge bifurcation from Example \ref{shcherbakexample} (see also Figure \ref{fig:bif25CuspEdge}).}
\label{25CEbif}
\end{figure}

\begin{example} \label{shcherbakexample}
The $2/5$-cuspidal edge bifurcation can be realized by
the 1-parameter family of  surfaces generated by  
Theorem \ref{gctheorem} with 
\[
(a^\zeta(t),b^\zeta(t),c^\zeta(t)) = (t, t+\zeta, 0.1).
\]
The Shcherbak bifurcation can be realized by 
\[
(a^{\zeta}(t),b^{\zeta}(t),c^{\zeta}(t)) = (t^2+\zeta, t, -1).
\]
We have computed images for both bifurcations and displayed them
in Figures \ref{25CEbif} and \ref{figshcherbakbif}. We used the $\zeta$ values
  $0.035$, $0$ and $-0.035$  for the $2/5$-cuspidal edge, and
the values $\zeta=-0.014$, $\zeta=0$ and $\zeta=0.015$ for the Shcherbak bifurcation.
\end{example}

\begin{figure}[ht]
  \begin{tabular}{cc}
      \begin{subfigure}{0.37\columnwidth}
      \hfill
      \includegraphics[width=\textwidth]{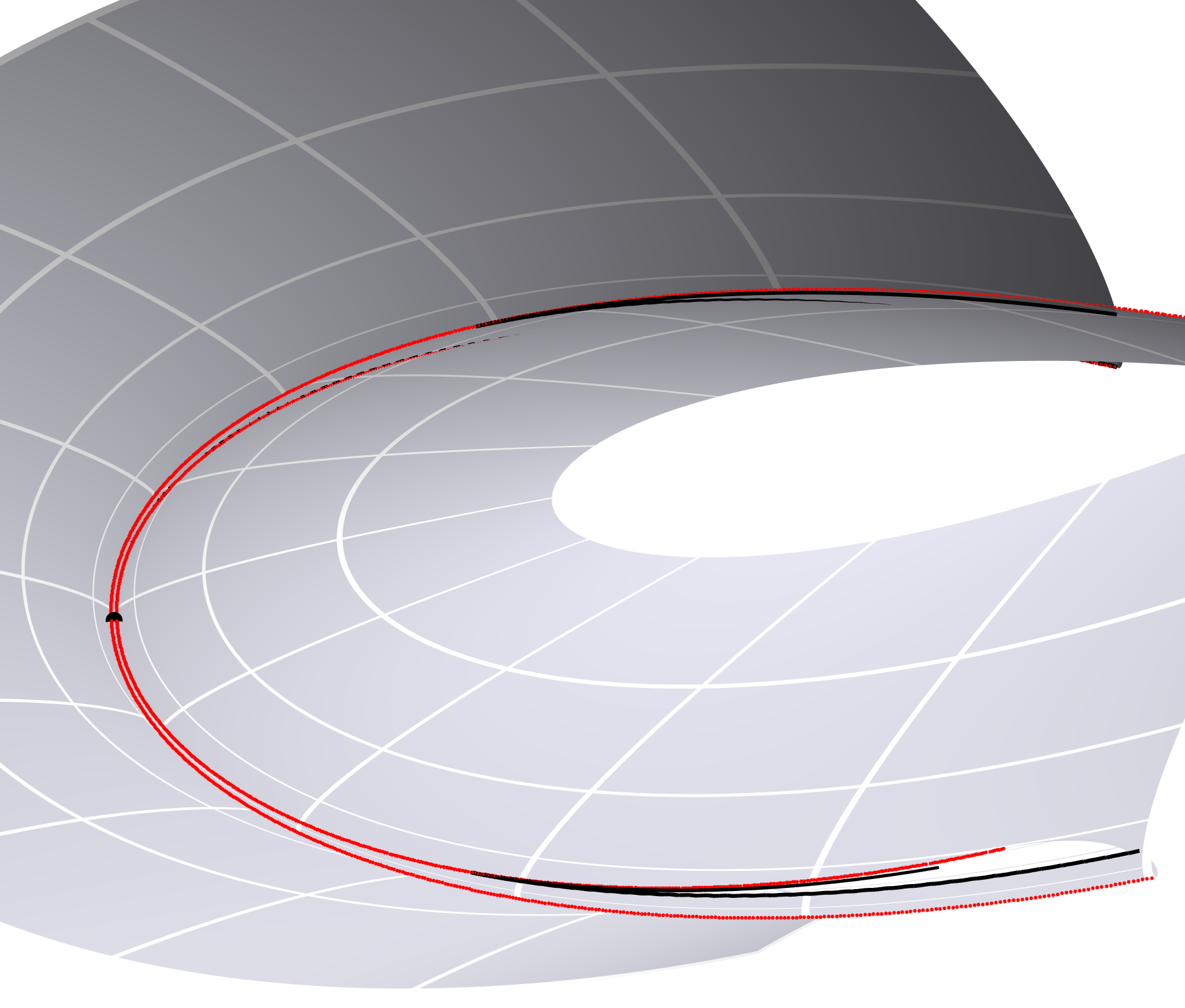}
    \end{subfigure}
    &
    \begin{tabular}{c}
      \begin{subfigure}{0.56\columnwidth}
       \hfil
        \includegraphics[width=0.475\textwidth]{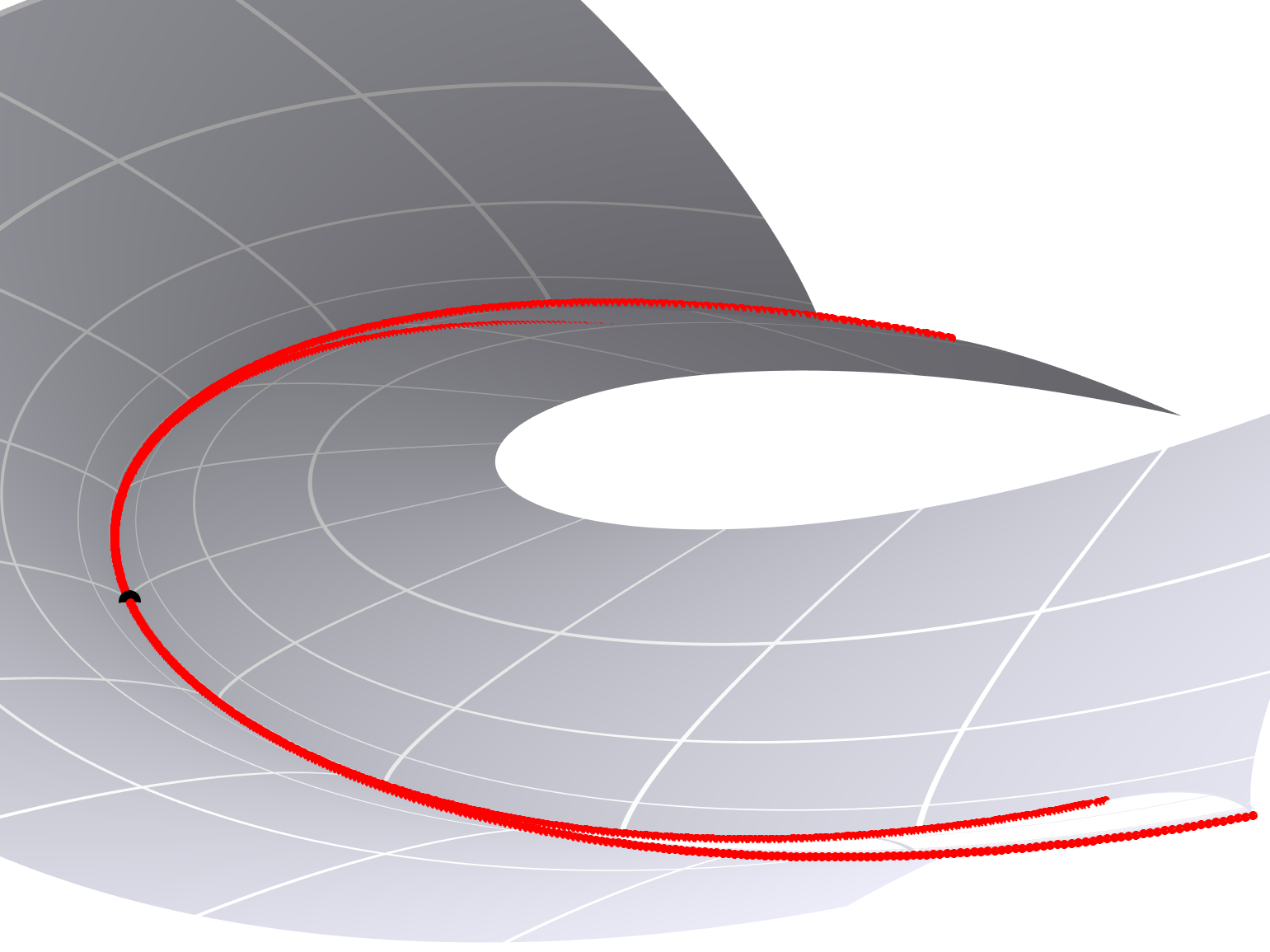}
          \includegraphics[width=0.475\textwidth]{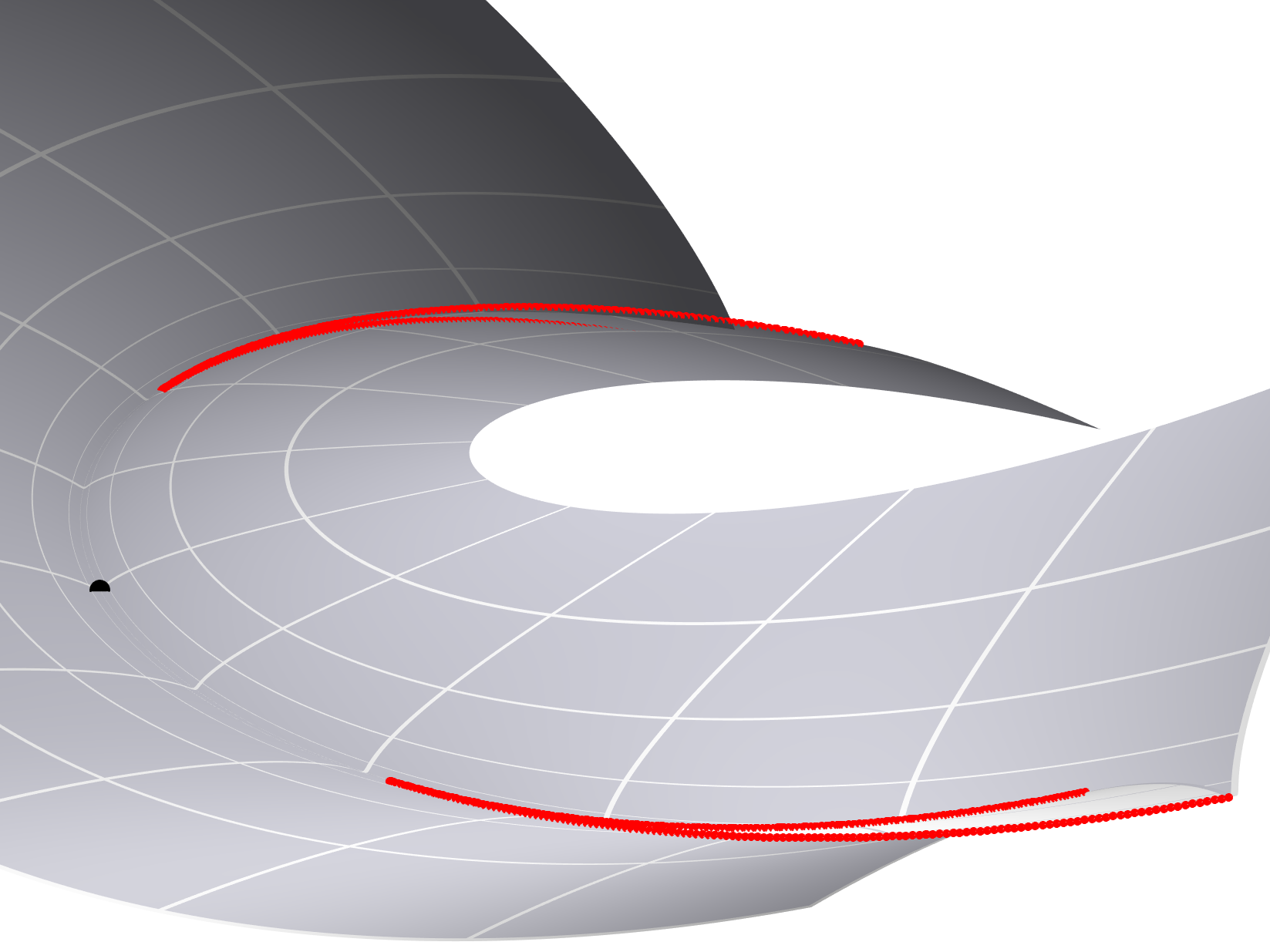}
          \vspace{1ex}
      \end{subfigure}\\
      \begin{subfigure}{0.56\columnwidth}
      \hfil
        \includegraphics[width=0.31\textwidth]{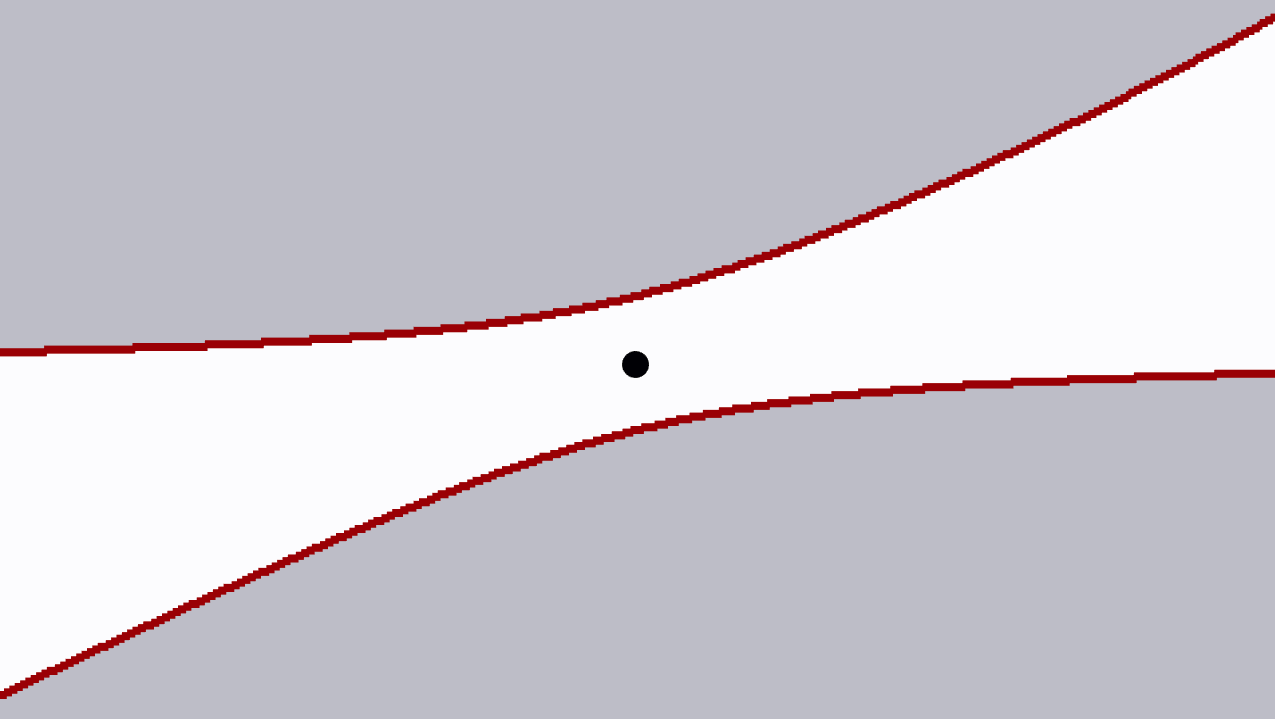}
        \includegraphics[width=0.31\textwidth]{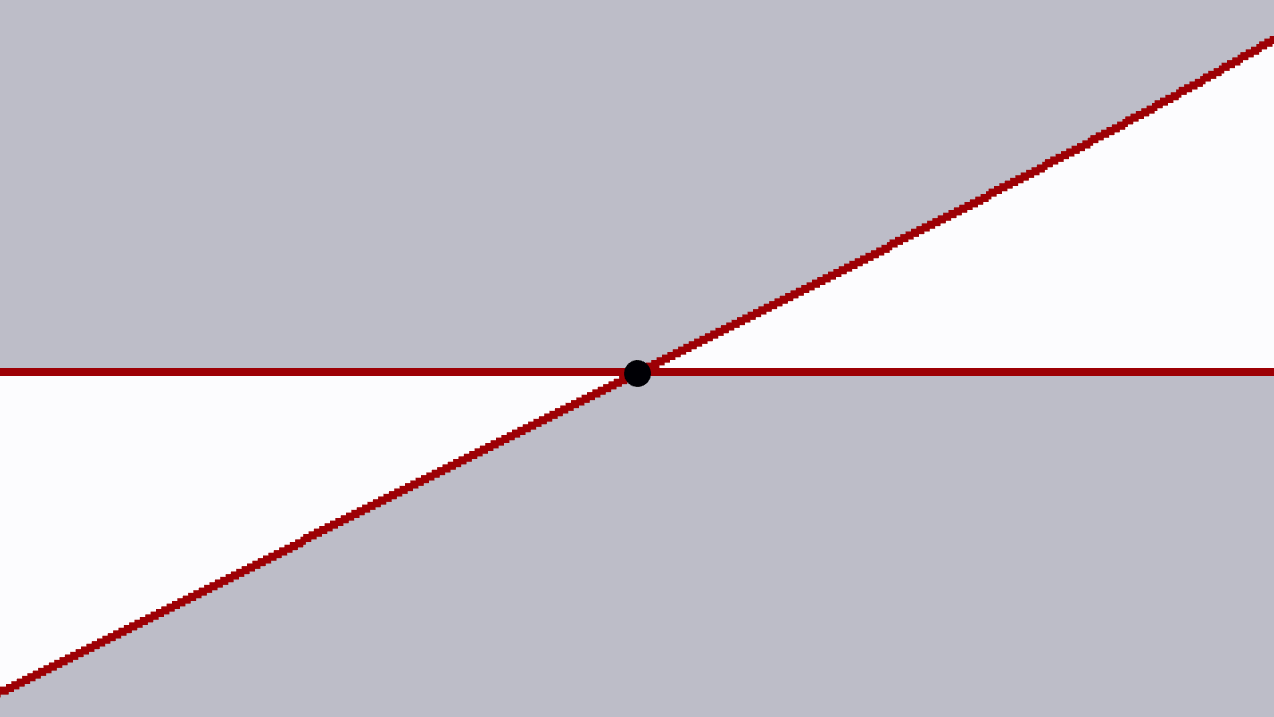}
        \includegraphics[width=0.31\textwidth]{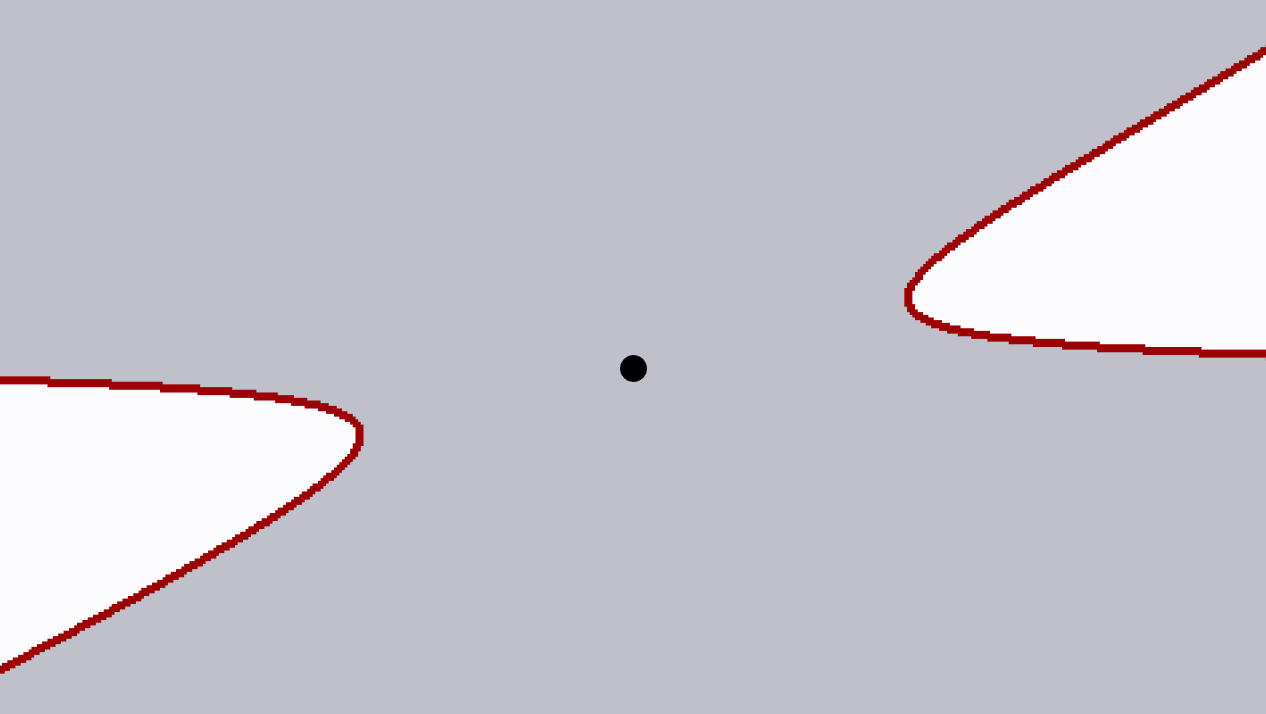}        
      \end{subfigure}
    \end{tabular}
  \end{tabular}
  \caption{
 Shcherbak singularity bifurcation from Example \ref{shcherbakexample}  (see also Figure \ref{fig:Shcherbak}). In the first image, the surface self intersection is shown by black curves. The singular sets in the coordinate domains are shown below to the right.}
\label{figshcherbakbif}
\end{figure}


\section{Local singularities of Lorentzian-harmonic maps into the 2-sphere}\label{sec:Classification}

We are interested in the singularities of germs of analytic Lorentzian-harmonic maps $N$ allowing analytic changes of coordinates in the source and target. Then $N$ is equivalent to the map-germ from $\mathbb R^2,0\to \mathbb R^2,0$ given by $(x,y)\mapsto (u(x,y),v(x,y))$, where $u$ and $v$ are as in \S \ref{sec:harmonicmaps}. We still denote by $N$ the map-germ $(u,v).$

Let $\mathcal{E}(2,1)$ denote the ring of germs of analytic functions $(\mathbb R^2,0)\to \mathbb R$, 
${\mathcal M}_2$ its unique maximal ideal and $\mathcal{E}(2,2)$ the $\mathcal{E}(2,1)$-module 
of analytic map-germs $(\mathbb R^2,0)\to \mathbb R^2$.
Consider the action of the group $\mathcal A$ of pairs of germs of analytic  
diffeomorphisms $(h,k)$ of the source and target on 
${\mathcal M}_2.\mathcal{E}(2,2)$ given by
$k\circ g\circ h^{-1}$, for $g\in {\mathcal M}_2.\mathcal{E}(2,2)$ 
(see, for example, \cite{arnoldetal,Martinet,Wall}). 
A germ $g$ is said to be finitely $\mathcal A$-determined if 
there exists an integer $k$ such that any map-germ with the same $k$-jet as $g$ 
is $\mathcal A$-equivalent to $g$. 
Let $\mathcal A_{k}$ be the subgroup of $\mathcal A$ whose elements
have the identity $k$-jets. 
The group $\mathcal A_k$ is a normal subgroup of $\mathcal
A$. Define
$\mathcal A^{(k)}=\mathcal A/\mathcal A_k$. The elements of $\mathcal A^{(k)}$ are the
$k$-jets of the elements of $\mathcal A$.
The action of $\mathcal A$ on ${\mathcal M}_2.{\mathcal E}(2,2)$
induces an action
of $\mathcal A^{(k)}$ on $J^k(2,2)$ as follows. For $j^kg\in J^k(2,2)$
and $j^k\mathfrak{h}\in \mathcal A^{(k)}$, $j^k\mathfrak{h}.j^kg=j^k(\mathfrak{h}.g).$

The tangent
space to the $\mathcal{A}$-orbit of $g$ at the germ $g$ is given by
 $$
L{\mathcal A}{\cdot}{g}={\mathcal M}_2.\{g_{x},g_{y}\}+g^*({\mathcal M}_2).\{e_1,e_2\},
$$
where $g_{x}$ and $g_y$ are the partial derivatives of $g$, $e_1,e_2$ denote
the standard basis vectors of ${\mathbb R}^2$ considered as elements
of ${\mathcal E}(2,2)$, and $f^*({\mathcal M}_2)$ is the pull-back of the maximal
ideal in ${\mathcal E}_2$. 
The extended tangent space to the $\mathcal{A}$-orbit of $g$ at the germ $g$
is given by
$$
L_e{\mathcal A}{\cdot}{g}={\mathcal E}_2.\{g_{x},g_{y}\}+g^*({\mathcal E}_2).\{e_1,e_2\}.
$$

We ask which finitely $\mathcal A$-determined singularities of map-germs in 
$\mathcal{E}(2,2)$ have a Lorentzian-harmonic map-germ in their $\mathcal A$-orbit, that is, 
which singularities can be represented by a germ of a Lorentzian-harmonic map into the 2-sphere.
We also ask whether an $\mathcal A_e$-versal deformation of the singularity
can be realized by families of Lorentzian-harmonic maps into the 2-sphere. 
(This means that the initial 
Lorentzian-harmonic map-germ can be deformed within the set of Lorentzian-harmonic map-germs and 
the deformation is $\mathcal A_e$-versal.)

The most extensive classification of finitely $\mathcal A$-determined 
singularities of maps germs in $\mathcal{E}(2,2)$ of $rank$ $1$ is carried out by Rieger in \cite{rieger} where he gave the following list of orbits of 
$\mathcal A$-codimension $\le 6$ which includes all the simple germs obtained in \cite{Goryunov} (the parameters $\alpha$ and $\beta$ are moduli and take values in $\mathbb R$ with certain exceptional values removed, 
see \cite{rieger} for details):

\begin{tabular}{ll}
$\bullet$ & $(x,y^2)$ \\
$\bullet$ & $(x,xy+P_1(y))$, $P_1=y^3,\, y^4,\, y^5\pm y^7,\, y^5,\,  
y^6\pm y^8+\alpha y^9,\,  y^6+ y^9$ or $y^7\pm y^9+
\alpha y^{10}+\beta y^{11}$\\
$\bullet$ & $(x,y^3\pm x^ky)$, $k\ge 2$\\
$\bullet$ & $(x,xy^2+ P_2(y))$, $P_2=y^4+y^{2k+1} (k\ge 2),\, 
y^5+y^6,\,  y^5\pm y^9,\,  y^5$ or $y^6+y^7+\alpha y^9$ \\
$\bullet$ & $(x,x^2y+P_3(x,y))$, $P_3=y^4\pm y^5,\,  y^4$ or $
xy^3+\alpha y^5+y^6+\beta y^7$\\
$\bullet$ &$(x,x^3y+\alpha x^2y^2+y^4 +x^3y^2)$
\end{tabular}

\smallskip
The $\mathcal A$-simple map-germs of rank $0$  
are classified in \cite{riegerruas} and are as follows:

\begin{tabular}{ll}
$\bullet$ & $I_{2,2}^{l,m}: (x^2+y^{2l+1},y^2+x^{2m+1}),\, l\ge m\ge 1$\\
$\bullet$ & $I_{2,2}^{l}: (x^2-y^2+x^{2l+1},xy),\, l\ge 1.$
\end{tabular}

\smallskip
We answer the above two questions for certain singularities including those in Rieger's list 
and for the $\mathcal A$-simple \mbox{rank $0$} map-germs.
A map-germ is finitely $\mathcal A$-determined if and only if 
a certain of its $n$-jet is finitely $\mathcal A$-determined. Therefore, we
can use the expressions for the $n$-jets of $u$ and $v$ in \eqref{n-jets(u,v)} and
Theorem \ref{theo:kjetLorentzH} to assert the existence of $N$. For this reason, 
in all of what follows, we take $N$ to be a germ of an analytic Lorentzian-harmonic map into the 2-sphere.
We start with rank 1 germs. 

\begin{theorem}\label{theo:singLOrzHar}
{\rm (i)} The singularities $(x,y^2)$, 
    $(x,xy+y^k)$, $(x,y^3\pm x^ky)$, those with $j^3N\sim_{\mathcal A^{(3)}} (x,xy^2)$ and $(x,x^3y+\alpha x^2y^2+y^4 +x^3y^2)$ 
can be represented by germs of Lorentzian-harmonic maps into the 2-sphere.
Furthermore, there are $\mathcal A_e$-versal deformations of  these singularities by germs of families of Lorentzian-harmonic maps into the 2-sphere.

{\rm (ii)} Map-germs with a 3-jet $\mathcal A^{(3)}$-equivalent to $(x,x^2y)$ cannot be represented by Lorentzian-harmonic map-germs into the 2-sphere.
\end{theorem}

\begin{proof}
We suppose, without loss of generality, that $N_x(0,0)\ne 0,$ so $a_{10}\ne 0$ or $b_{10}\ne 0$. 
In all of what follows, we suppose that $a_{10}\ne 0$. Then $N$ is $\mathcal A$-equivalent to a germ of the form $(x,g(x,y))$, for some germ of an analytic function $g$. We can make successive changes of coordinates in the source and target to obtain the desired $n$-jet of $g$ (we do this with the help of Maple).

The map-germ $N$ is singular if, and only if, 
$a_{10}b_{11}-a_{11}b_{10}=0$. Suppose that this is the case, so $b_{11}=a_{11}b_{10}/a_{10}$. Then the 2-jet of $g$ is given by
$$
j^2g=\frac{2a_{11}}{a_{10}^3}(a_{10}b_{20}-a_{20}b_{10})xy+
\frac{1}{a_{10}^3}\left((a_{10}b_{20}-a_{20}b_{10})a_{11}^2+(a_{10}b_{22}-a_{22}b_{10})a_{10}^2)\right)y^2.
$$

Clearly, we can choose $N$ so that the coefficient of $y^2$ in  $j^2g$ does not vanish. In that case $N\sim_{\mathcal A} (x,y^2)$ and is stable.

\smallskip
Suppose that $a_{11}(a_{10}b_{20}-a_{20}b_{10})\ne 0$ and $(a_{10}b_{20}-a_{20}b_{10})a_{11}^2+(a_{10}b_{22}-a_{22}b_{10})a_{10}^2=0$, so  $b_{22}=(a_{10}^2a_{22}b_{10}-a_{11}^2(a_{10}b_{20}-a_{20}b_{10}))/a_{10}^3$. Then 
$j^2g\sim_{\mathcal A^{(2)}} (x,xy)$. Any finitely $\mathcal A$-determined germ with this 2-jet is equivalent to $(x,xy+P(y))$ for some polynomial function $P$ (\cite{rieger}). The coefficient of $y^k$ in $P(y)$ is equal to $b_{kk}+Q$, with $Q$ a polynomial in $a_{i0},a_{ii},b_{i0},b_{ii}$, $1\le i\le k-1$.
Clearly, we can choose appropriate $b_{kk}$ to represent all finitely $\mathcal A$-determined singularities $(x,xy+P(y))$
with Lorentzian-harmonic map-germs. Also, $\mathcal A_e$-versal deformations of these singularities are of the form 
$(x,xy+P(y)+\sum_{i=1}^ku_iy^i)$, so by deforming $v(0,y)$ we can get $\mathcal A_e$-versal deformations of the singularities of $N$ by Lorientzian-harmonic map-germs.

\smallskip
Suppose that $j^2g\equiv 0$. 
We have two cases to consider: (1) $a_{11}\ne 0$ (then $N$ is the Gauss map of a pseudospherical surface $f$ which is a wave front) 
and $a_{10}b_{20}-a_{20}b_{10}=0$, or (2) $a_{11}=0$ (then $b_{11}=0$, so $N_y(0,0)=0$; $f$ 
is not a wave front by Proposition \ref{prop:f_frontal_wavefront}). 
In all of what follows, we consider the case (1), case (2) follows similarly.
Then the 3-jet of $g$ is given by
$$
\small{
\begin{array}{rl}
a_{10}^4j^3g&=l_1x^2y+l_2xy^2+l_3y^3\\
&=3a_{11}(a_{10}b_{33}-a_{33}b_{10})(a_{11}xy^2-x^2y)+((a_{10}b_{33}-a_{33}b_{10})a_{10}^3-(a_{10}b_{30}-a_{30}b_{10})a_{11}^3)y^3.
\end{array}
}
$$

Following the criteria in Table 6.1 in \cite{izumiyaetal}, 
we have a lips/beaks singularity, i.e.,  $N\sim_{\mathcal A} (x,y^3\pm x^2y)$, 
if and only if $l_3\ne 0$ and $l_2^2-3l_1l_3=(a_{10}b_{30}-a_{30}b_{10})(a_{10}b_{33}-a_{33}b_{10})\ne 0$. 
 The lips (resp. beaks) occurs when $l_3\ne 0$ and $g_y=3l_3y^2+2l_2yx+l_1x^2+O(3)$ has an $A_1^+$ (resp. $A_1^-$)-singularity, that is, when $l_2^2-3l_1l_3<0$ (resp. $>0$).
Clearly, both these singularities can occur.

Suppose that $l_3\ne 0$ and $l_2^2-3l_1l_3=0$. Then $j^3N\sim_{\mathcal A^3}(x,y^3)$. We make successive changes of coordinates to reduce the $k+1$-jet of $N$ to $(x,y^3+c x^ky)$. Using \eqref{sys:Pdes}, we can show that $c$ is a polynomial in $a_{i0},a_{ii},b_{i0},b_{ii}$, $1\le i\le k-1$. When $c$ is considered as a polynomial in  $b_{(k-1)0}$, the coefficient of the linear term $b_{(k-1)0}$ is $2b_{10}b_{11}-(k-1)a_{10}a_{11}$. Therefore, we can choose $u$ and $v$ so that $c\ne 0$, that is, all the singularities in the series
$(x,y^3\pm x^ky)$ can be represented by Lorentzian-harmonic map-germs.
Using the same argument, we can show that these singularities can be $\mathcal A_e$-versally unfolded by germs of families of Lorentzian-harmonic maps.

\smallskip
If we take $l_3=0$ above and $a_{10}b_{33}-a_{33}b_{10}\ne 0$, then 
$j^2N\sim_{\mathcal A^3}(x,xy^2)$. We can show, by similar arguments to those in (iii) that any finitely $\mathcal A$-determined germ of the from $(x,xy^2+P(y))$ can be represented by
a Lorentzian-harmonic map and its singularity can also be $\mathcal A_e$-versally unfolded by germs of families of Lorentzian-harmonic maps.

\smallskip
We take $l_3=l_2$ above. Then we can choose $u,v$ so that 
$N\sim_{\mathcal A}-(x,x^3y+\alpha x^2y^2+y^4 +x^3y^2)$ (the expression of $j^5g$ is too lengthy to reproduce here). We can also $\mathcal A_e$-versally unfold this singularity by germs of families of Lorentzian-harmonic maps.

\smallskip
(ii) If we have $l_2=0$ above, then $l_1=0$. Therefore, map-germs with a 3-jet $\mathcal A^{(3)}$-equivalent to $(x,x^2y)$ cannot be represented by Lorentzian-harmonic map-germs.
\end{proof}

We deal now with rank zero map-germs. 

\begin{theorem}\label{theo:Rank0}
	{\rm (i)} The singularity $I_{2,2}^{l,m}$ can be represented by
	a germ of a Lorentzian-harmonic map into the 2-sphere. There is an $\mathcal A_e$-versal deformation of this singularity by germs of families of Lorentzian-harmonic maps. 
	
	{\rm (ii)} The singularities $I_{2,2}^{l}$ cannot be represented by
	a germ of a Lorentzian-harmonic map into the 2-sphere.
	
	{\rm (iii)} There are no finitely $\mathcal A$-determined germs of Lorentzian-harmonic maps into the 2-sphere with a zero 2-jet at the singular point.
\end{theorem}

\begin{proof}
(i) and (ii) Here $a_{10}=a_{11}=b_{10}=b_{11}=0$, so
$$
j^2N=(a_{20}x^2+a_{22}y^2,b_{20}x^2+b_{22}y^2).
$$

If $a_{20}b_{22}-a_{22}b_{20}\ne 0$, then $j^2N\sim_{\mathcal A^{(2)}}(x^2,y^2)$, 
otherwise it is $\mathcal A^{(2)}$-equivalent to $(x^2\pm y^2,0)$, $(x^2,0)$ or $(0,0)$. Therefore, the singularities $I_{2,2}^{l}$ cannot be represented by
a germ of a Lorentzian-harmonic map into the 2-sphere, which proves (ii). (A geometric argument for excluding this singularity is the following: when $N$ has rank $0$, its singular set contains the two null curves at the singular point, and 
the singular set of the  $I_{2,2}^{l}$-singularity consist of an isolated point.)

Item (i) follows by similar arguments to those in the proof of Theorem \ref{theo:singLOrzHar}.

(iii) We prove by induction using \eqref{sys:Pdes} and the fact that the functions $u$ and $v$ and their first and second order derivatives are zero at the origin, that 
$\partial^{k+1}u/\partial x^k\partial y$, 
$\partial^{k+1}u/\partial x\partial y^k$, 
$\partial^{k+2}u/\partial x^k\partial y^2$ and 
$\partial^{k+2}u/\partial x^2\partial y^k$ all vanish at the origin for $k\ge 0$. Therefore, $u$ can be written in the form $u(x,y)=x^3A(x)+y^3B(y)+x^3y^3C(x,y),$ for some germs of analytic functions $A,B,C$. The same holds for the function $v$. It follows that the singular set $\Sigma$ of $N$ is given by $\lambda(x,y)=(u_xv_y-u_yv_x)(x,y)=x^2y^2\Lambda(x,y),$ for some germ of an analytic function $\Lambda$. Therefore $\Sigma$ has non-isolated singularities along the null curves $x=0$ and $y=0$. It follows by Gaffney's geometric criteria (see Theorem 2.1 in \cite{Wall}) that $N$
is not finitely $\mathcal A$-determined.
\end{proof}


\section{Appendix: Local singularities of Lorentzian-harmonic maps into the plane}\label{sec:Classification_plane}

We consider here local singularities of Lorentzian-harmonic maps ${\mathcal N}:\real^{1,1},0 \to \real^2,0$, that is, 
solutions of the wave equation ${\mathcal N}_{xy}=0$, where $(x,y)=(\bar x-t, \bar x+t)$ are null coordinates for the Lorentz plane. Here we take $\mathcal{N}$ and all other functions to be sufficiently differentiable for the questions that arise.  The general solution has the form
$$
\mathcal N(x,y) =(f_1(x) + g_1(y), f_2(x) + g_2(y)),
$$
i.e., is given as a sum of two arbitrary maps in $x$ and  $y$. 
In what follows we suppose, without loss of generality, that the point of interest is the origin $(0,0)\in \mathbb R^{1,1}$. We denote by $\Sigma$ the singular set of ${\mathcal N}.$

\begin{proposition}\label{prop:LHmaps_plane_gen}
Let ${\mathcal N}$ be a germ, at the origin, of a Lorentzian-harmonic map into the plane. 

{\rm (i)} If ${\mathcal N}_x(0,0)=(0,0)$ (resp. ${\mathcal N}_y(0,0)=(0,0)$) then the null curve $x=0$ (resp. $y=0$) is locally a part of the singular set $\Sigma$ of ${\mathcal N}$.

{\rm (ii)} If $rank (\dd {\mathcal N}_{(0,0)})=0$ then both the null curves $x=0$ and $y=0$ are locally parts of the singular set $\Sigma$ of ${\mathcal N}$.

{\rm (iii)} If $j^2{\mathcal N}=(0,0)$, then $\Sigma$ has non-isolated singularities along the 
null curves $x=0$ and $y=0$.
\end{proposition}

\begin{proof}
The proof is straightforward since $\Sigma$ 
is the zero set of $\lambda=f'_1(x)g'_2(y)- f'_2(x)g'_1(y)$.
For instance, for (iii) we can write ${\mathcal N}(x,y)=(x^3h_1(x)+y^3k_1(y),x^3h_2(x)+y^3k_2(y))$, consequently $\lambda(x,y)=x^2y^2\lambda_1(x,y)$.
\end{proof}

\begin{theorem}\label{theo:LHmaps_plane_poss_sing}
{\rm (i)} The singularities $(x,y^2)$, $(x,xy\pm y^k)$, $(x,y^3\pm x^ky)$, 
those with $j^3{\mathcal N}\sim_{\mathcal A^{(2)}}(x,xy^2)$, $(x,x^3y+\alpha x^2y^2+y^4 +x^3y^2)$ and 
the singularities  $I_{2,2}^{l,m}$ can be represented by germs of Lorentzian-harmonic maps into the plane.

{\rm (ii)} 	Map-germs with a 3-jet $\mathcal A^{(3)}$-equivalent to $(x,x^2y)$ and the singularities $I_{2,2}^{l}$ cannot be represented by Lorentzian-harmonic map-germs into the plane.

{\rm (iii)} There are no finitely $\mathcal A$-determined Lorentzian-harmonic map-germs into the plane with a zero 2-jet.
\end{theorem}

\begin{proof}

(i) Suppose that $j^1{\mathcal N}\ne (0,0)$. Interchanging $x$ and $y$ if necessary, we can change coordinates in the source and write 
${\mathcal N}\sim_{\mathcal A}(x+ y^l,\bar{f}_2(x)+\bar{g}_2(y))$ (here we take $g_1$ to have at least one non-zero derivative at zero). A further change of variable gives  
${\mathcal N}\sim_{\mathcal A}(x,f_2(x-y^l)+g_2(y))$. 

If $l=1$,  we can choose $f_2$ and $g_2$ so as to get the singularities 
$(x,y^2)$, $(x,xy\pm y^k)$, $(x,y^3\pm x^ky)$ and $(x,x^3y+\alpha x^2y^2+y^4 +x^3y^2)$.

If $l=2$,  we can choose $f_2$ and $g_2$ so as to get the finitely $\mathcal A$-determined singularities with 3-jets $\mathcal A^{(3)}$-equivalent to $(x,xy^2)$.

If $l\ge 3$,  ${\mathcal N}\sim_{\mathcal A}(x,y^3h(x,y))$ when $g_2'(0)=g_1''(0)=0$, so $f$ is not finitely $\mathcal A$-determined. Therefore, when $l\ge 3$, the only finitely $\mathcal A$-determined singularity that we can get is $(x,y^2)$ when $g_2''(0)\ne 0$. 

 Suppose now that $j^1{\mathcal N}=(0,0)$. Clearly, we can get all the singularities $I_{2,2}^{l,m}: (x^2+y^{2l+1},y^2+x^{2m+1}),\, l\ge m\ge 1$.

\smallskip
(ii) We have 
$
{\partial^k \lambda}/{\partial x^k} = f^{(k)}_1g'_2- f^{(k)}_2g'_1
,
{\partial^k \lambda}/{\partial x^k} = f'_1g^{(k)}_2- f'_2g^{(k)}_1
$
and 
$
\lambda_{xy}=f''_1g''_2- f''_2g''_1$, so  
$\lambda_{xy}(0,0)=0$ when $\lambda=\lambda_x=\lambda_y=0$ at the origin. Therefore, $j^2\lambda$ has a singularity more degenerate than Morse when $\lambda_{xx}(0,0)=0$ or $\lambda_{yy}(0,0)=0$. In these conditions, we have 
$j^3{\mathcal N}\sim_{\mathcal A^{(3)}}(x,x^2y)$ if, and only if, the kernel direction of $df$ 
is parallel to the kernel direction of $Hess(\lambda)$ at the origin. This could happen if $\lambda_{xx}(0,0)=\lambda_{yy}(0,0)=0$, but that leads to $j^3{\mathcal N}\sim_{\mathcal A^{(3)}}(x,0)$.

\smallskip
(iii) By Proposition \ref{prop:LHmaps_plane_gen}, the singular set of ${\mathcal N}$ has a non-isolated singularity. It follows by Gaffney's geometric criteria (see Theorem 2.1 in \cite{Wall}) that the singularity of ${\mathcal N}$ is not finitely $\mathcal A$-determined.
\end{proof}


\medskip
\noindent
{\bf Acknowledgements:} The research was partially supported through the program ``Research in Pairs''
by the Mathematisches Forschungsinstitut Oberwolfach in 2019. The first named author was partially supported by the Independent Research Fund Denmark, grant number 9040-00196B. The second author acknowledges financial support from S\~ao Paulo Research Foundation (FAPESP) grant \mbox{2019/07316-0} and CNPq grant 303772/2018-2.



\begin{thebibliography}{10}

\bibitem{ArnoldWavefront}
VI~Arnold.
\newblock Wave front evolution and equivariant {M}orse lemma.
\newblock {\em Commun. Pure and Appl. Math.}, 29:557--582, 1976.

\bibitem{arnoldetal}
VI~Arnold, SM~Gusein-Zade, and AN~Varchenko.
\newblock {\em Singularities of differentiable maps {I}}, volume~82 of {\em
  Monographs in Mathematics}.
\newblock Birkh\"auser Boston, Inc., 1985.

\bibitem{jgp}
D~Brander.
\newblock Loop group decompositions in almost split real forms and applications
  to soliton theory and geometry.
\newblock {\em J. Geom. Phys.}, 58:1792--1800, 2008.

\bibitem{singps}
D~Brander.
\newblock Pseudospherical surfaces with singularities.
\newblock {\em Ann. Mat. Pura Appl.}, 196:905--928, 2017.

\bibitem{dbms1}
D~Brander and M~Svensson.
\newblock The geometric {C}auchy problem for surfaces with {L}orentzian
  harmonic {G}auss maps.
\newblock {\em J. Differential Geom.}, 93:37--66, 2013.

\bibitem{dbft}
D~Brander and F~Tari.
\newblock Families of spherical surfaces and harmonic maps.
\newblock {\em Geometriae Dedicata}, 201:203--225, 2019.

\bibitem{bruceParallel}
JW~Bruce.
\newblock Wavefronts and parallels in {E}uclidean space.
\newblock {\em Math. Proc. Cambridge Philos. Soc.}, 93:323--333, 1983.

\bibitem{BruceMarar}
JW~Bruce and WL~Marar.
\newblock Images and varieties.
\newblock {\em J. Math. Sci.}, 82:3633--3641, 1996.

\bibitem{docarmo}
M~do~Carmo.
\newblock {\em Differential geometry of curves and surfaces}.
\newblock Prentice-Hall, 1976.

\bibitem{fukuihasegawa}
T~Fukui and M~Hasegawa.
\newblock Singularities of parallel surfaces.
\newblock {\em Tohoku Math. J.}, 64:387--408, 2012.

\bibitem{Goryunov}
VV~Goryunov.
\newblock Singularities of projections of complete intersections.
\newblock {\em (Russian) Current problems in mathematics}, 22(1983):167--206,
  19.
\newblock Itogi Nauki i Tekhniki, Akad. Nauk SSSR, Vsesoyuz. Inst. Nauchn. i
  Tekhn. Inform., Moscow, 1983. English translation: J. Soviet Math. 27 (1984),
  2785-2811.

\bibitem{guest1997}
M~Guest.
\newblock {\em Harmonic maps, loop groups and integrable systems}.
\newblock Cambridge University Press, 1997.

\bibitem{IshikawaRecogFrontal}
G~Ishikawa.
\newblock Recognition problem of frontal singularities. {P}reprint
  arxiv:1808.09594.

\bibitem{Ishikawa(2005)}
G~Ishikawa.
\newblock Infinitesimal deformations and stability of singular legendre
  submanifolds.
\newblock {\em Asian J. Math.}, 9:133--166, 2005.

\bibitem{Ishikawa_Frontals}
G~Ishikawa.
\newblock Singularities of frontals.
\newblock {\em Adv. Stud. Pure Math.}, 78:55--106, 2018.

\bibitem{ishimach}
G~Ishikawa and Y~Machida.
\newblock Singularities of improper affine spheres and surfaces of constant
  {G}aussian curvature.
\newblock {\em Internat. J. Math.}, 17:269--293, 2006.

\bibitem{izumiyaetal}
S~Izumiya, MC~Romero Fuster, M~Ruas, and F~Tari.
\newblock {\em Differential geometry from a singularity theory viewpoint}.
\newblock World Scientific Publishing Co. Pte. Ltd., Hackensack, NJ, 2016.

\bibitem{izumiyasaji}
S~Izumiya and K~Saji.
\newblock The mandala of {L}egendrian dualities for pseudo-spheres of
  {L}orentz-{M}inkowski space and ``flat'' spacelike surfaces.
\newblock {\em J. Singul.}, 2:92--127, 2010.

\bibitem{izsata}
S~Izumiya, K~Saji, and M~Takahashi.
\newblock Horospherical flat surfaces in hyperbolic 3-space.
\newblock {\em J. Math. Soc. Japan}, 62:789--849, 2010.

\bibitem{KabataR2R2}
Y~Kabata.
\newblock Recognition of plane-to-plane map-germs.
\newblock {\em Topology Appl.}, 202:216--238, 2016.

\bibitem{krsuy}
M~Kokubu, W~Rossman, K~Saji, M~Umehara, and K~Yamada.
\newblock Singularities of flat fronts in hyperbolic space.
\newblock {\em Pacific J. Math.}, 221:303--351, 2005.

\bibitem{looijenga}
EJN Looijenga.
\newblock {\em Structural stability of smooth families of ${C}^{\infty}$ -
  functions}.
\newblock PhD Thesis. Universiteit van Amsterdam, 1974.

\bibitem{Martinet}
J~Martinet.
\newblock {\em Singularities of smooth functions and maps}, volume~58 of {\em
  LMS Lecture Note Series}.
\newblock Cambridge University Press, 1982.

\bibitem{tkmilnor1982}
TK~Milnor.
\newblock Characterizing harmonic immersions of surfaces with indefinite
  metric.
\newblock {\em Proc. Nat. Acad. Sci. U.S.A.}, 79:2143--2144, 1982.

\bibitem{Mond}
D~Mond.
\newblock On the classification of germs of maps from \mbox{$\mathbb R^2$} to
  \mbox{$\mathbb R^3$}.
\newblock {\em Proc. London Math. Soc.}, 50:333--369, 1985.

\bibitem{pohlmeyer}
K~Pohlmeyer.
\newblock Integrable {H}amiltonian systems and interactions through quadratic
  constraints.
\newblock {\em Comm. Math. Phys.}, 46:207--221, 1976.

\bibitem{PreS}
A~Pressley and G~Segal.
\newblock {\em Loop Groups}.
\newblock Oxford Mathematical Monographs. Clarendon Press, Oxford, 1986.

\bibitem{rieger}
JH~Rieger.
\newblock Families of maps from the plane to the plane.
\newblock {\em J. London Math. Soc.}, 36:351--369, 1987.

\bibitem{riegerruas}
JH~Rieger and MAS Ruas.
\newblock Classification of \mbox{$\mathcal A$}-simple germs from $k^n$ to
  $k^2$.
\newblock {\em Compositio Math}, 79:99--108, 1991.

\bibitem{taturu2004}
D~Tataru.
\newblock The wave maps equation.
\newblock {\em Bull. Amer. Math. Soc. (N.S.)}, 41:185--204, 2004.

\bibitem{todaagag}
M~Toda.
\newblock Initial value problems of the sine-{G}ordon equation and geometric
  solutions.
\newblock {\em Ann. Global Anal. Geom.}, 27:257--271, 2005.

\bibitem{Wall}
CTC Wall.
\newblock Finite determinacy of smooth map-germs.
\newblock {\em Bull. London Math. Soc.}, 13:481--539, 1981.

\bibitem{wood1977}
JC~Wood.
\newblock Singularities of harmonic maps and applications of the
  {G}auss-{B}onnet formula.
\newblock {\em Amer. J. Math.}, 99:1329--1344, 1977.

\end{thebibliography}
\end{document}